\documentclass[a4paper,11pt]{amsart}

\usepackage[utf8]{inputenc}
\usepackage[margin=2.5cm]{geometry}
\usepackage{hyperref}
\usepackage{enumerate, amsmath,amsthm,amssymb,amscd,mathrsfs}
\usepackage{colortbl,float,longtable,xcolor,booktabs,bookmark}
\usepackage[all]{xy}
\usepackage{graphicx,tikz-cd}
\usetikzlibrary{trees}
\usetikzlibrary{arrows}

\newcommand{\gap}{\hspace{0.5em}}
\newcommand{\vgap}{\vspace{0.25em}}
\newcommand{\padding}{\rule[-1.45ex]{0pt}{0.2em}\gap}
\newcommand{\oddrow}{\rowcolor[gray]{0.95}}
\newcommand{\evnrow}{}
\newlength{\matrixheight}
\newcommand{\calculatepadding}[1]{\settoheight{\matrixheight}{\hbox{#1}}{#1}}

\theoremstyle{plain}

\newtheorem*{thmA}{Theorem A}
\newtheorem*{thmB}{Theorem B}
\newtheorem*{thmC}{Theorem C}
\newtheorem{thm}{Theorem}[section]
\newtheorem{pro}[thm]{Proposition}
\newtheorem{lem}[thm]{Lemma}

\newtheorem{Lemma}[thm]{Lemma}
\newtheorem{Proposition}[thm]{Proposition}

\newtheorem*{thm*}{Theorem}

\theoremstyle{definition}
\newtheorem{dfn}[thm]{Definition}
\newtheorem{claim}[thm]{Claim}
\newtheorem{exa}[thm]{Example}

\newtheorem{rem}[thm]{Remark}

\newtheorem{Notation}[thm]{Notation}
\newtheorem{Conv}[thm]{Convention}

\newtheorem{Definition}[thm]{Definition}
\newtheorem{Remark}[thm]{Remark}

\newtheorem{say}[thm]{}

\definecolor{wwwwww}{rgb}{0.4,0.4,0.4}

\newcommand{\map}{\dashrightarrow}

\makeatletter
\let\c@equation\c@thm
\makeatother
\numberwithin{equation}{section}

\DeclareMathOperator{\codim}{codim}
\DeclareMathOperator{\Spec}{Spec}
\DeclareMathOperator{\Ext}{Ext}
\DeclareMathOperator{\Aut}{Aut}

\DeclareMathOperator{\Cl}{Cl}
\DeclareMathOperator{\NE}{NE}
\DeclareMathOperator{\Pic}{Pic}
\DeclareMathOperator{\NEbar}{\overline{NE}}
\DeclareMathOperator{\rk}{rk}
\DeclareMathOperator{\divisor}{div}

\DeclareMathOperator{\Eff}{Eff}
\DeclareMathOperator{\Bir}{Bir}

\DeclareMathOperator{\mult}{mult}
\DeclareMathOperator{\SBir}{Bir}

\DeclareMathOperator{\Hom}{Hom}

\DeclareMathOperator{\Proj}{Proj}
\DeclareMathOperator{\uProj}{\underline{Proj}}
\DeclareMathOperator{\Sym}{\mathit{Sym}}

\DeclareMathOperator{\rank}{rank}

\DeclareMathOperator{\PGL}{PGL}

\DeclareMathOperator{\centre}{z}
\DeclareMathOperator{\Supp}{Supp}

\newcommand{\PP}{\mathbb{P}}
\newcommand{\NN}{\mathbb{N}}

\newcommand{\FF}{\mathbb{F}}
\newcommand{\GG}{\mathbb{G}}
\newcommand{\CC}{\mathbb{C}}
\newcommand{\TT}{\mathbb{T}}
\newcommand{\ZZ}{\mathbb{Z}}
\newcommand{\RR}{\mathbb{R}}
\newcommand{\LL}{\mathbb{L}}

\newcommand{\QQ}{\mathbb{Q}}
\renewcommand{\P}{\mathbb{P}}

\newcommand{\Q}{\mathbb{Q}} 
\newcommand{\bR}{\mathbb{R}} 
\newcommand{\OO}{\mathcal{O}}
\newcommand{\cO}{\mathcal{O}}
\newcommand{\cC}{\mathcal{C}}
\newcommand{\cP}{\mathcal{P}}
\newcommand{\cI}{\mathcal{I}}

\renewcommand{\O}{\mathcal{O}} 
\newcommand{\cE}{\mathcal{E}}

\begin{document}

\title[\resizebox{5.9in}{!}{Birational geometry of Calabi--Yau pairs and $3$-dimensional Cremona transformations}]{Birational geometry of Calabi--Yau pairs and $3$-dimensional Cremona transformations}

\author[Carolina Araujo]{Carolina Araujo}
\address{\sc Carolina Araujo\\
IMPA\\
Estrada Dona Castorina 110\\
22460-320 Rio de Janeiro\\ Brazil}
\email{caraujo@impa.br}

\author[Alessio Corti]{Alessio Corti}
\address{\sc Alessio Corti\\
Department of Mathematics\\ Imperial College London\\ 
180 Queen's Gate\\ London SW7 2AZ\\ United Kingdom}
\email{a.corti@imperial.ac.uk}

\author[Alex Massarenti]{Alex Massarenti}
\address{\sc Alex Massarenti\\ Dipartimento di Matematica e Informatica, Universit\`a di Ferrara, Via Machiavelli 30, 44121 Ferrara, Italy}
\email{msslxa@unife.it}

\date{}

\begin{abstract}
  In this paper we develop a framework that allows one to describe the
  birational geometry of Calabi--Yau pairs $(X,D)$.  After
  establishing some general results for Calabi--Yau pairs $(X,D)$ with
  mild singularities, we focus on the special case when $X=\PP^3$ and
  $D\subset \PP^3$ is a quartic surface.  We investigate how the
  appearance of increasingly worse singularities in $D$ enriches the
  birational geometry of the pair $(\PP^3, D)$, and lead to
  interesting subgroups of the Cremona group of $\PP^3$.
\end{abstract}

\maketitle

\tableofcontents

\section{Introduction}
\label{sec:introduction}

\subsection{Overview}

In \cite{Og13, Og13bis}, Oguiso addressed the following question, attributed to Gizatullin:

\medskip

\begin{center}
\emph{Which automorphisms of a smooth quartic surface $D\subset \PP^3$ are induced by  Cremona transformations of $\PP^3$?}
\end{center}

\medskip

\noindent He produced several interesting examples of smooth quartic surfaces in $\PP^3$, among which:
\begin{enumerate}
\item A smooth quartic surface $D\subset \PP^3$ with $\Aut(D)\cong\ZZ$, and such that no nontrivial automorphism of $D$
is induced by a Cremona transformation of $\PP^3$ (see~\cite[Theorem~1.2]{Og13} and~\cite[Theorem~1.8]{Og13bis}). 
\item A smooth quartic surface $D\subset \PP^3$ with
  $\Aut(D)\cong\ZZ_2*\ZZ_2*\ZZ_2$, and such that every automorphism of
  $D$ is induced by a Cremona transformation of $\PP^3$
  (see~\cite[Theorem~1.7]{Og13bis}).
\end{enumerate}

\noindent More recently, Paiva and Quedo produced examples of smooth
quartic surfaces $D\subset \PP^3$ with $\Aut(D)\cong\ZZ_2*\ZZ_2$, and
such that no nontrivial automorphism of $D$ is induced by a Cremona
transformation of $\PP^3$ (see~\cite[Example 1]{Paiva_Quedo}).

The pair $(\PP^3,D)$, where $D\subset \PP^3$ is a smooth quartic
surface, is an example of a \emph{Calabi--Yau (CY) pair}, that is, a
pair $(X,D)$, consisting of a normal projective variety $X$ and an
effective Weil divisor $D$ on $X$ such that $K_X+D\sim 0$.  Oguiso's
results can be interpreted as statements about the \emph{birational
  geometry of the CY pair $(\PP^3,D)$}, which is the theme of this
paper.

Given a CY pair $(X,D)$, there is a rational volume form $\omega$ on
$X$, unique up to scaling by a nonzero constant, such that
$D+\divisor_X \omega =0$. Our goal is to understand birational
self-maps of $X$ that preserve the volume form $\omega$, up to
scaling.  In particular, we are interested in the structure of the
group $\Bir (X,D)$ of such volume preserving birational maps.  When
the pair $(X, D)$ has \emph{canonical singularities} (see
Definition~\ref{dfn:terminalANDcanonical}), a birational self-map of
$X$ is volume preserving if and only if it restricts to a birational
self-map of $D$ (Proposition~\ref{group2}), and hence there is a group
homomorphism
\[
r\colon \Bir(X, D) \ \to \ \Bir(D) \,.
\]
When $D\subset \PP^3$  is a smooth quartic surface, the pair $(\PP^3,D)$ has canonical (even terminal) singularities and $\Bir(D)=\Aut(D)$. So Gizatullin's question can be re-stated as follows:

\begin{center}
\emph{For a smooth quartic surface $D\subset \PP^3$, what is the image of the restriction homomorphism 
\[
r\colon \Bir(\PP^3, D) \ \to \ \Aut(D) \,?
\]}
\end{center}

We are also interested in describing the kernel of $r$. More
generally, in dimension $n\geq 3$, consider a CY pair
$(\PP^n, D_{n+1})$, where $D_{n+1}$ is a degree $n+1$ hypersurface. In
this case, $\Bir (\PP^n, D_{n+1})$ is a subgroup of the Cremona group
$\Bir (\PP^n)$.  If we want to produce interesting subgroups of
$\Bir (\PP^n)$ in this way, then we must pick a hypersurface $D_{n+1}$
with special properties. Indeed, our first main result (Theorem~A)
implies that, if $D_{n+1}$ is general (very general if $n = 3$), then
the group $\Bir (\PP^n, D_{n+1})$ is trivial.  More precisely, if the
pair $(\PP^n, D_{n+1})$ has \emph{terminal singularities} (see
Definition~\ref{dfn:terminalANDcanonical}), and the divisor class
group (that is, the group of integral Weil divisors modulo linear equivalence) 
of $D_{n+1}$ is generated by $\OO_{D_{n+1}}(1)$, then
$\Bir (\PP^n, D_{n+1})=\Aut(\PP^n, D_{n+1})$.\footnote{We denote by
  $\Aut (X,D)$ the group of automorphisms of $X$ that stabilize $D$.}

The $2$-dimensional case $(\PP^2, C)$ has also been considered. When
$C$ is a smooth cubic, the restriction homomorphism
$r\colon \Bir( \PP^2, C) \ \to \ \Aut(C)$ is surjective (see~\cite[Theorem~2.2]{Og13bis}). In this case, generators for $\Bir( \PP^2, C)$ are
described in~\cite[Th\'eor\`eme 1.4.]{Pan07}, while the kernel of $r$
is generated by its elements of degree~$3$~\cite[Theorem~1]{Blanc08}.
When $C$ is the sum of the three coordinate lines, generators for
$\Bir( \PP^2, C)$ were given in~\cite[Theorem~1]{Blanc13}.

Next we turn to the setting of \emph{singular} quartic surfaces
$D\subset \PP^3$. Our results illustrate how the appearance of
increasingly worse singularities on $D$ enriches the birational
geometry of the pair $(\PP^3, D)$.  The simplest case to consider is
when $D\subset \PP^3$ is a quartic surface having a single
$A_1$-singularity\footnote{In this case the pair $(\PP^3, D)$ has
  canonical but not terminal singularities.}, and the divisor class group of
$D$ is generated by $\OO_D(1)$.  Theorem~B gives a complete
description of $\Bir (\PP^3,D)$ in this case.  We then investigate
what happens when the singularities of $D$ become worse.  Theorem~C
addresses the birational geometry of the pair $(\PP^3, D)$ when
$D\subset \PP^3$ is a quartic surface with a single $A_2$-singularity,
and the divisor class group of $D$ is generated by $\OO_D(1)$. At the most
singular end of the spectrum, we find rational quartic surfaces. These
have been proven to be Cremona equivalent to a plane in
~\cite[Theorem~1]{Mel20}. Furthermore, log Calabi--Yau pairs of the
form $(\mathbb{P}^3,D)$ and coregularity at most one have been
recently classified in~\cite{Duc22}.

Our main tool in this paper is a factorisation theorem for volume
preserving birational maps. We briefly explain it here, and refer to
Section~\ref{sec:sarkisov} for details.  The \emph{Sarkisov} program
allows one to factorize any birational map $\psi\in \Bir (\PP^n)$ as a
\emph{chain of links} between Mori fibre spaces:
\[
  \begin{tikzpicture}[xscale=2,yscale=-1.2]
    \node (A0_0) at (0, 0) {$\PP^n=X_0 \ $};
    \node (A0_1) at (1, 0) {$\ X_1\ $};
    \node (A0_2) at (2, 0) {$\ \ \cdots\ \ $};
    \node (A0_3) at (3, 0) {$\ X_{n-1}\ $};
    \node (A0_4) at (4, 0) {$\ X_n=\PP^n.$};
    \node (A1_0) at (0, 1) { $\Spec \CC$ };
    \node (A1_1) at (1, 1) {$\ S_1\ $};
    \node (A1_3) at (3, 1) {$\ S_{n-1}\ $};
    \node (A1_4) at (4, 1) { $\Spec \CC$ };
    \path (A0_0) edge [->,dashed]node [auto] {$\scriptstyle{\psi_1}$} (A0_1);
    \path (A0_1) edge [->,dashed]node [auto] {$\scriptstyle{\psi_2}$} (A0_2);
    \path (A0_2) edge [->,dashed]node [auto] {$\scriptstyle{\psi_{n-1}}$} (A0_3);
    \path (A0_3) edge [->,dashed]node [auto] {$\scriptstyle{\psi_n}$} (A0_4);
    \path (A0_0) edge [->,dashed,bend right=40]node [auto] {$\scriptstyle{\psi}$} (A0_4);
    \path (A0_0) edge [->]node [auto] { } (A1_0);
    \path (A0_1) edge [->]node [auto] { } (A1_1);
   \path (A0_3) edge [->]node [auto] { } (A1_3);
   \path (A0_4) edge [->]node [auto] { } (A1_4);
  \end{tikzpicture}
  \]
  The \emph{Sarkisov} program was established in dimension~3
  in~\cite{Co95}, and in higher dimensions in~\cite{HM13}.  The volume
  preserving version of this theorem was established
  in~\cite{CK16}. It states that, if the birational map $\psi$ is
  volume preserving for some \emph{log canonical} CY pair
  $(\PP^n, D)$, then there is a factorisation as above, and divisors
  $D_i\subset X_i$ making each $(X_i,D_i)$ a mildly singular CY pair
  (see Definition~\ref{dfn:(t,c)etc} for the notion of \emph{(t,~lc)}
  pair), and each $\psi_i$ a volume preserving link between CY pairs
  (see Definition~\ref{def:volume_preserving}).

  A mildly singular CY pair $(X,D)$ together with a Mori fibre space
  structure $X\to S$ is called a \emph{Mori fibered (Mf) CY pair} (see
  Definition~\ref{def3}).  The \emph{pliability} of a Mf CY pair
  $(X, D)$ is the set $\cP(X, D)$ of equivalence classes of Mf CY
  pairs that admit a volume preserving birational map to $(X, D)$ (see
  Definition~\ref{square} for the precise notion of equivalence in
  this setting). Our main theorems not only describe $\Bir (X, D)$,
  but also determine the pliability $\cP(X, D)$ of the CY pairs in
  question.

  Theorem~A states that, if $X$ is a Fano variety with $\rho(X)=1$,
  $(X, D)$ is a CY pair with terminal singularities, and the divisor class
  group of $D$ is generated by the restriction of a divisor on $X$,
  then $\cP(X,D)=\big\{(X,D)\to \Spec \CC\big\}$ is a set with one
  element, and $\Bir (X, D) = \Aut (X,D)$.

  In the setting of Theorem~B, let $D\subset \PP^3$ be a quartic
  surface with a single $A_1$-singularity $z\in D$, and suppose that
  the divisor class group of $D$ is generated by $\OO_D(1)$.  The blow-up
  $X \to \PP^3$ of $z\in \PP^3$, together with the strict transform
  $D_X\subset X$ of $D$, and the $\PP^1$-bundle $X \to \PP^2$ induced
  by the projection from $z$, is a Mf CY pair. Moreover, the
  birational morphism $(X, D_X) \to (\PP^3,D)$ is volume
  preserving. Theorem~B states that the pliability of $(\PP^3, D)$ is
  the set consisting of the two elements $(\PP^3,D)\to \Spec \CC$ and
  $(X,D_X) \to \PP^2$.

  Our Theorem~C describes the pliability of the pair $(\PP^3, D)$ when
  $D\subset \PP^3$ is a quartic surface with a single
  $A_2$-singularity, and the divisor class group of $D$ is generated by
  $\OO_D(1)$. We were surprised to discover just how large this
  pliability is.

\medskip

Throughout the paper we work over the field $\CC$ of complex numbers,
or, more generally, any algebraically closed field of characteristic
zero.

In the remaining of the introduction, we define all the relevant
notions regarding CY pairs and state precisely our main results,
Theorems~A, B and C.  In Section~\ref{sec:sarkisov}, we review the
Sarkisov program for Mf CY pairs.  In Sections~\ref{section_ThmA},
\ref{section:A1} and \ref{section:A2}, we prove Theorems~A, B and C,
respectively. In Section~\ref{sec:extr-contr}, we develop the basic toolkit needed to prove Theorems B and C, and, 
more generally, to study volume preserving birational maps between 3-dimensional Mf CY pairs.
Recently, some of these tools were used in \cite{APZ} to settle Gizatullin's problem for smooth quartic surfaces with Picard number $2$.

\subsection{Calabi-Yau pairs}

\begin{Definition}
  \label{dfn:CYpairs} \label{def:volume_preserving} A
  \emph{Calabi--Yau (CY) pair} is a pair $(X,D)$ consisting of a
  normal projective variety $X$ and an effective integral Weil divisor
  $D$ on $X$ such that $K_X+D$ is linearly equivalent to $0$.  
  (In particular, $K_X+D$ is Cartier.)
  There
  exists a top degree rational differential form $\omega=\omega_{X,D}$
  on $X$, unique up to multiplication by a nonzero constant, such that
  $D+\divisor_X \omega =0$. With a slight abuse of language, we call
  $\omega$ the \emph{volume form} of the CY pair $(X,D)$.

Let $(X,D_X)$ and $(Y, D_Y)$ be CY pairs, with associated volume forms $\omega_{X,D_X}$ and $\omega_{Y,D_Y}$.
A birational map $\varphi
  \colon X \dasharrow Y$ is \emph{volume preserving} if there exists a
  nonzero constant $\lambda \in \CC^\times$ such that 
\[
\varphi^* (\omega_{Y,D_Y}) = \lambda \omega_{X,D_X}.
\]
We abuse language and say that $\varphi$ is a \emph{birational map
  of CY pairs} to mean that it is a volume preserving map. 

We denote by $\Bir(X,D)$ the group of volume preserving birational
self-maps of the CY pair $(X,D)$, and by $\Bir\big((X,D_X),(Y,D_Y)\big)$ the set of volume preserving birational maps between CY pairs $(X,D_X)$ and $(Y,D_Y)$.
 
Volume preserving birational maps are called \emph{crepant birational} in \cite{Ko13}.
\end{Definition}

Next we introduce several natural classes of singularities of pairs.
Given a CY pair $(X,D)$ and a divisorial valuation $E$ of $X$, we denote by $a(E, K_X+D)=a(E,X,D)$ the discrepancy of $E$ with respect to $(X,D)$, as defined in \cite[Definition 2.25]{KM98}.
We denote by $\centre_E X$ the \emph{centre} of $E$ on $X$. 
Depending on the context, we view it as a scheme-theoretic point or a subvariety of $X$.
When $\codim_{(\centre_E X)} X \geq 2$, we say that $E$ has \emph{small centre} on $X$.

\begin{Definition}
  \label{dfn:terminalANDcanonical}
 A pair $(X, D)$ has \emph{terminal singularities} if for all $z\in
 X$:
 \begin{enumerate}[(a)]
 \item if $\codim_z X \leq 2$, then both $X$ and $D$ are smooth at $z$;
  \item if $\codim_z X > 2$, then for all divisorial valuations $E$ of
    $X$ with $\centre_E X= z$, $a(E, K_X+D)>0$.
\end{enumerate}
A pair $(X, D)$ has \emph{canonical singularities} if for all divisorial valuations $E$ with small centre on $X$, $a(E,K_X+D)\geq 0$. 
Following common usage, we say that a pair ``is'' terminal (respectively canonical) if it has terminal (respectively canonical) singularities.
\end{Definition}

\begin{rem}
  \begin{enumerate}[(1)]
  \item Suppose that $X$ is $\mathbb{Q}$-Gorenstein. 
    If the pair $(X,D)$ is terminal (respectively canonical) then both $X$ and
    $D$ are terminal (respectively canonical). In particular, $D$ is normal.
  \item If $X$ is smooth and $(X,D)$ is terminal,
  then for all $z\in X$ with $\codim_z X>2$, $\mult _z(D)<\codim_zX-1$.
\item In particular, if $X$ is a smooth \mbox{$3$-fold}, then $(X,D)$ is
  terminal if and only if $D$ is smooth.
\item If $X$ is smooth and $z\in D$ is an isolated singularity
  of multiplicity $\mult_z D < \dim X-1$ and smooth projectivized tangent cone,
  then $(X,D)$ is terminal.
  \end{enumerate}
\end{rem}

\begin{Definition}\label{dfn:(t,c)etc}
  We say that a pair $(X,D)$ is \emph{(t,~c)} (respectively
  \emph{(t,~lc)}) if $X$ has terminal singularities and the pair
  $(X,D)$ has canonical (respectively log canonical) singularities. We
  say that a pair $(X,D)$ is \emph{$\mathbb{Q}$-factorial} if $X$ is
  $\mathbb{Q}$-factorial.
\end{Definition}

\begin{Definition}\label{def3}
  A \emph{Mori fibered (Mf) CY pair} is a $\mathbb{Q}$-factorial
  \textit{(t,~lc)} CY pair $(X,D)$ together with a Mori fiber space
  structure on $X$, i.e., a morphism $f\colon X\rightarrow S$ such that
  $f_* \mathcal{O}_X=\mathcal{O}_S$, $-K_X$ is $f$-ample, and
  $\rho(X)-\rho(S)=1$.
\end{Definition}

\begin{Definition}\label{square}
  Let $(X,D_X)\rightarrow S_X$ and $(Y,D_Y)\rightarrow S_Y$ be Mf CY
  pairs. A volume preserving birational map
  $f \colon (X,D_X) \dasharrow (Y,D_Y)$ is \emph{square} if it fits into a
  commutative diagram
\[
  \begin{tikzpicture}[xscale=2.3,yscale=-1.2]
    \node (A0_0) at (0, 0) {$X$};
    \node (A0_1) at (1, 0) {$Y$};
    \node (A1_0) at (0, 1) {$S_X$};
    \node (A1_1) at (1, 1) {$S_Y$};
    \path (A0_0) edge [->,dashed]node [auto] {$\scriptstyle{f}$} (A0_1);
    \path (A0_0) edge [->]node [auto] {} (A1_0);
    \path (A0_1) edge [->]node [auto] {} (A1_1);
    \path (A1_0) edge [->,dashed]node [auto] {$\scriptstyle{g}$} (A1_1);
  \end{tikzpicture}
  \]
  where $g$ is birational, and the induced birational map of generic
  fibers $f_\xi \colon X_\xi\dasharrow Y_{\xi}$ is biregular. In this case, we
  say that the Mf CY pairs $(X,D_X)\rightarrow S_X$ and
  $(Y,D_Y)\rightarrow S_Y$ are \textit{square equivalent}.
\end{Definition}

\begin{rem}
  \label{rem:square} For Mf CY pairs $(X,D_X)\rightarrow S_X$ and $(Y,D_Y)\rightarrow S_Y$ to be square equivalent,
  it is necessary that $\dim S_X=\dim S_Y$.
  If $S_X=S_Y=\Spec \CC$, then the Mf CY pairs are square equivalent if
  and only if $(X,D_X)$ and $(Y,D_Y)$ are isomorphic as CY pairs. 
\end{rem}

\begin{Definition}\label{dfn:pliability}
The \textit{pliability} of the Mf CY pair $(X,D_X)$ is the set
\[
\mathcal{P}(X,D_X)\ =\ \frac{\Big\{\mbox{Mf CY pairs } (Y,D_Y)\rightarrow S_Y\: \big| \: \Bir\big((X,D_X),(Y,D_Y)\big)\neq\emptyset \Big\}}{\mbox{square equivalence}}.
\]
We say that the Mf CY pair $(X,D_X)\rightarrow S_X$ is
\textit{birationally rigid} if $\mathcal{P}(X,D_X)$ consists of a
single element (necessarily the element $(X,D_X)\to S_X$).
\end{Definition}

\subsection{Main results}
\begin{Notation}
If $Z$ is a normal variety, we denote by $\Cl Z$ the \emph{divisor class group} of $Z$, i.e., the group of integral Weil divisors on $Z$ modulo linear equivalence.
Equivalently, $\Cl Z$ is the group of isomorphism classes of \emph{divisorial sheaves} on $Z$, i.e., rank~$1$ torsion free saturated coherent sheaves on $Z$.
\end{Notation}

\begin{rem}
Let $X$ be a normal variety that is nonsingular in codimension~$2$, and $D\subset X$ a normal subvariety of codimension~$1$.
Given  an integral Weil divisor $A$ on $X$ such that $\Supp A\not \subset  D$, the restricted Weil divisor $A_{|D}$ can be defined by restricting $A$ as a Cartier divisor outside a subset of codimension $\geq 3$. It is not well-behaved in general.  

On the other hand, we will use frequently the following: If $X$ has terminal singularities, and $D$ and $A$ are $\QQ$-Cartier, then, by~\cite[Proposition~5.26]{KM98}, we have an exact sequence:
\[(0) \to \OO_X(-D+A) \to \OO_X(A)\to \OO_D(A_{|D}) \to (0) \, .\]
In particular, this says that the coherent sheaf
\[
\OO_X(A)/\OO_X(-D+A)
\]
is a rank~$1$ torsion free saturated $\OO_D$-module, i.e., a divisorial sheaf on $D$.
\end{rem}

\begin{thmA}\label{thmA}
  Let $(X,D)\to \Spec \CC$ be a Mf CY pair (so that $X$ is a $\QQ$-factorial terminal Fano variety with $\rho =1$) such that 
  \begin{enumerate}[(i)]
  \item $(X,D)$ is terminal, and
  \item $\Cl D =\ZZ \cdot \OO_D(A_{|D})$, where $A$ is an integral Weil divisor on $X$.
  \end{enumerate}
  Then $(X,D)\to \Spec \CC$ is birationally rigid and $\Bir (X,D) = \Aut (X, D)$.
\end{thmA}

\begin{rem}
  \begin{enumerate}[(1)]
  \item Let $(X,D_X)$ and $(Y,D_Y)$ be \emph{(t,~c)} CY pairs. A birational map
    $\varphi \colon X \dasharrow Y$ is volume preserving if and only
    if $\varphi$ restricts to a birational map $D_X \dasharrow D_Y$ (Proposition~\ref{group1}). 
Thus, in the statement of Theorem A, we can take $\Bir(X,D)$
 na\"{\i}vely to mean the group of birational maps $\varphi\colon
 X\dasharrow X$ that stabilize $D$. Similar
 considerations apply to the statement of Theorem~B below.
\item Assumption~(i) is needed for the statement to hold. 
  Theorems~B and C address the cases when $X=\PP^3$ and $D$ is a generic
  quartic with one singular point of type $A_1$ and $A_2$, respectively. 
In  both cases the pair $(X,D)$ is canonical but not terminal, and $\Bir (X,D) \supsetneq \Aut (X, D)$. 
  \item Assumption~(ii) is also needed for the statement to hold. 
This is illustrated for instance by Oguiso's example (2) mentioned above.
  \item Assumption~(ii) is meaningful even when $n\geq 4$. If
   $X=\PP^4$ and $D$ is a quintic \mbox{$3$-fold} with ordinary
   quadratic singularities, then the pair $(X,D)$ is
   terminal but the restriction map $r\colon \Cl X \to \Cl D$ is not
   necessarily an isomorphism. For instance if $D$ contains a plane,
   then in general $D$ has $16$ ordinary quadratic singularities and
   $\Cl D$ has rank $2$. 
  \end{enumerate}
\end{rem}

\begin{thmB}\label{SBir_A1_sing}
 Let $D\subset \PP^3$ be a quartic surface. Assume the following:
 \begin{enumerate}[(i)]
 \item $D$ is smooth apart from a single $A_1$-singularity $z\in D$, 
 \item $\Cl D = \ZZ \cdot \OO_D(1)$.
 \end{enumerate}
In what follows denote by $f\colon X \to \PP^3$ the blow-up of $z\in \PP^3$,
$D_X\subset X$ the strict transform of $D$, and $\pi \colon X \to
\PP^2$ the $\PP^1$-bundle induced by the projection from $z$. 

Then the following holds:
\begin{enumerate}[(1)]
\item The pliability of the pair $(\PP^3, D)$ is the set with two
  elements 
  $$(\PP^3,D)\to \Spec \CC \text{ \textit{and} } \pi \colon (X, D_X)\to
  \PP^2.$$
\item The restriction  homomorphism $r\colon \Bir(\PP^3, D) \ \to \ \Bir(D)$ induces
an exact sequence of groups:
\[
1\to \GG  \to \Bir(\PP^3, D)\to \Bir D
\to 1 ,
\]
where $\GG$ is the twist of $\GG_m$ corresponding to the quadratic
extension $\CC(x, y)\subset \CC(D)$.
More precisely, the morphism $\pi \colon D_X \to \PP^2$ is finite of degree 2 and
  identifies the function field $\CC (D_X)$ with a quadratic extension
  $\CC (x,y)(\sqrt{a})$ where $a=a_6(x,y)\in \CC[x,y]$ is a sextic
  polynomial. The group $\GG$ is the form of
  $\GG_m$ over $\CC(x,y)$ whose $\CC(x,y)$-points are the solutions of
  Pell's equation
$$
U^2-aV^2=1.
$$
\end{enumerate}
\end{thmB}

\begin{rem}
  \label{rem:class_rankA1}
  Assumptions~(i) and~(ii) are satisfied for a very general quartic surface $D$ with
  an $A_1$-singularity. Indeed, consider the moduli space of lattice-polarized K3 surfaces containing the
  rank-$2$ lattice~$N=\ZZ^2$ with quadratic form
  \[
  \left(\begin{array}{cc}
4 & 0 \\
0 & -2
\end{array}\right).
  \]
 in their Picard group. A very general element of this moduli space is
 the minimal resolution of a quartic surface $D$ satisfying
 assumptions~(i) and~(ii). 
\end{rem}

\begin{thmC}
  \label{thm:A2singularity}
 Let $D\subset \PP^3$ be a quartic surface. Assume the following:
 \begin{enumerate}[(i)]
 \item $D$ is smooth apart from a single $A_2$-singularity $z\in D$,
 \item $\Cl D = \ZZ \cdot \OO_D(1)$.
 \end{enumerate}
 Then the pliability of the pair $(\PP^3, D)$ is the
(infinite) set consisting of square equivalence classes of the following Mf CY pairs:
objects~$1$, $2$, $3^a$, $3^b$, the
$3$-parameter family of objects~$4$, and the $7$-parameter family of
objects~$5^a$ constructed and displayed in
Table~\ref{table:1}.  The
paragraph following Remark~\ref{rem:class_rankA2} explains how to read Table~\ref{table:1}.
\end{thmC}

\begin{rem}
  \label{rem:class_rankA2}
  Assumptions~(i) and~(ii) are satisfied for a very general quartic surface $D$ with
  an $A_2$-singularity. Indeed, consider the moduli space of
  lattice-polarized K3 surfaces containing the rank-$3$
  lattice~$N=\ZZ^3$ with quadratic form
  \[
\left(\begin{array}{ccc}
4 & 0 & 0 \\
0 & -2 & 1\\
0 & 1 & -2
\end{array}\right)
\] 
in their Picard group. A very general element of this moduli space is
 the minimal resolution of a quartic surface $D$ satisfying
 assumptions~(i) and~(ii). 
\end{rem}

\begin{say}
In this paper we often work with hypersurfaces and complete intersections in \emph{geometric invariant theory} (GIT) quotients $F=\CC^m/\!/_\omega\TT$. We summarise our key notation and conventions for working with these. Fix a lattice $\LL\cong \ZZ^r$ and let $\TT=\Spec \CC[\LL^\star]$ be the $r$-dimensional torus with character group $\LL^\star = \Hom (\LL,\ZZ)$.\footnote{The odd-looking notation is compatible with common practice in toric geometry.} Consider the action of $\TT$ on $\CC^m$ given by a lattice homomorphism $\ZZ^m \to \LL^\star$. We typically choose a basis of $\LL^\star$ and represent this homomorphism by an $r\times m$ integer matrix $D$ --- the matrix of \emph{weights} --- whose columns we denote by $D_i$, $i=1,\dots, m$ (we often abuse notation and identify $D_i$ with the corresponding homogeneous coordinate function $x_i$ on $\CC^m$). 
Denote by $\cC = \langle D_1,\dots, D_m \rangle_+\subset \LL^\star\otimes \RR$ the convex cone spanned by the columns of $D$. An element $\omega \in \cC$ is called a \emph{stability condition.} There is a wall-and-chamber decomposition of $\cC$, where the walls are the $(r-1)$-dimensional cones of the form $\langle D_{i_1},\dots D_{i_{r-1}} \rangle_+$. If $\omega$ is in the interior of a chamber, the \emph{irrelevant ideal} $I_\omega \subset \CC[x_1,\dots, x_m]$ is
\[
I_\omega = \left(x_{i_1}\cdots x_{i_r} \mid \omega \in \langle D_{i_1},\dots D_{i_r} \rangle_+\right)\,,
\]
the \emph{unstable locus} is $Z_\omega=V(I_\omega)\subset \CC^m$, and the GIT quotient is
\[
F=\CC^m/\!/_\omega \TT= \left(\CC^m\setminus Z_\omega\right)/\TT \,.
\]
The irrelevant ideal, unstable locus, and GIT quotient only depend on the chamber containing $\omega$ and this chamber is identified with the ample cone of $F$. By construction, $F$ is covered by affine open subsets
\[
U_{i_1\cdots i_r} =
\bigl\{ \mathbf{a}\in \CC^m\mid a_j\neq 0 \ \forall j\in\{i_1,\dots,i_r\}\bigr\} /\TT \, ,
\] 
indexed by tuples $i_1,\dots, i_r$ such that $\omega \in \langle D_{i_1},\dots, D_{i_r} \rangle_+$.
Moreover, it is clear that
\[
U_{i_1\cdots i_r} =
V_{i_1\cdots i_r}/G,\quad \text{where} \quad V_{i_1\cdots i_r}=
\bigl\{ \mathbf{a}\in \CC^m\mid  a_j=1 \ \forall j\in\{i_1,\dots,i_r\}\bigr\}
\]
and $G$ is a finite abelian group (that needs to be computed explicitly in every case). 

A character $\chi \in \cC$ gives a line bundle $L_\chi$ on $F$ such that
\[
H^0\left(F, L_\chi\right) 
=
\bigl\{ f \in \CC[x_1, \dots, x_m] \mid \forall \,g\in \TT \, , \, \forall \, \mathbf{a}\in \CC^m \, , \, f(g\mathbf{a})=\chi(g) f(\mathbf{a})\bigr\} \,.
\]
In this paper we often fix a weight matrix, a character $\chi\in \LL^\star$, and work with a hypersurface $X=(f=0)$ where $f\in H^0(F,L_\chi)$ is an explicit polynomial.
\end{say}

\begin{center}
\begin{footnotesize}
\begin{longtable}{ccccc}
\caption{Mf CY pairs $(X^\dagger,D^\dagger)$ birational to $(\PP^3, D)$ where $D\subset \PP^3$
  is a quartic surface with one $A_2$-singularity.} \label{table:1} \\
\toprule
\multicolumn{1}{c}{Object}&
\multicolumn{1}{c}{$X^\dagger \subseteq$ Ambient}&\multicolumn{1}{c}{Ambient coords. \& wts.}&
\multicolumn{1}{c}{Eqn. of $X^\dagger$ in Ambient}&
\multicolumn{1}{c}{Eqn. of $D^\dagger$ in $X^\dagger$} \\
\midrule
\endfirsthead
\multicolumn{4}{l}{\tiny Continued from previous page.}\\
\addlinespace[1.7ex]
\midrule
\multicolumn{1}{c}{Object}&
\multicolumn{1}{c}{$X^\dagger \subseteq$ Ambient}&\multicolumn{1}{c}{Ambient coords. \& wts.}&
\multicolumn{1}{c}{Eqn. of $X^\dagger$  in Ambient}&
\multicolumn{1}{c}{Eqn. of $D^\dagger$  in $X^\dagger$} \\
\midrule
\endhead
\midrule
\multicolumn{4}{r}{\tiny Continued on next page.}\\
\endfoot
\bottomrule
\endlastfoot
\oddrow $1$& $X^\dagger = \PP^3$&
\vgap\calculatepadding{$\begin{array}{cccc}x_0&x_1&x_2&x_3 \\ \hline
                          1&1&1&1 \end{array}$}\padding&
$0$&$x_0x_1x_3^2+Bx_3+C$ \\
\evnrow $2$&$X^\dagger = \FF_1^3$ &
\vgap\calculatepadding{$\begin{array}{ccccc}x_0&x_1&x_2&x_3&x \\ \hline
                                            1&1&1&0&-1\\0&0&0&1&1 \end{array}$}\padding&
$0$ &$x_0x_1x_3^2+Bx_3x+Cx^2$ \\
\oddrow $2^a$&$X^\dagger = \FF_2^3$ &
\vgap\calculatepadding{$\begin{array}{ccccc}x_0&x_1&x_2&x_3&x \\ \hline
                                            1&1&1&0&-2\\0&0&0&1&1 \end{array}$}\padding&
$0$ &$x_0x_3^2+Bx_3x+x_1Cx^2$\\
\evnrow $2^b$&$X^\dagger = \FF_2^3$&
\vgap\calculatepadding{$\begin{array}{ccccc}x_0&x_1&x_2&x_3&x \\ \hline
                                            1&1&1&0&-2\\0&0&0&1&1 \end{array}$}\padding&
$0$&$x_1x_3^2+Bx_3x+x_0Cx^2$\\
\oddrow $3^a$&$X^\dagger = \PP(1^3,2)$&
\vgap\calculatepadding{$\begin{array}{cccc}x_0&x_1&x_2&y \\ \hline 
1&1&1&2 \end{array}$}\padding&
$0$&$x_0y^2+By+x_1C$ \\
\evnrow $3^b$& $X^\dagger = \PP(1^3,2)$&\vgap\calculatepadding{$\begin{array}{cccc}
x_0&x_1&x_2&y \\ \hline 1&1&1&2 \end{array}$}\padding&$0$
                                                                          &$x_1y^2+By+x_0C$
  \\
\oddrow $4$& $X_4\subset\PP(1^3,2^2)$&\vgap\calculatepadding{$\begin{array}{ccccc}
x_0&x_1&x_2&y_0&y_1 \\ \hline
                                                                1&1&1&2&2 \end{array}$}\padding&\begin{minipage}{4.5cm}\begin{multline*}
                                                                y_0y_1+C-\\-L(x_0y_1-x_1y_0-B) \end{multline*}\end{minipage}
                                                                          &$x_0y_1-x_1y_0-B$
  \\
\evnrow $5^a$& $X_4\subset\PP(1^4,2)$&\vgap\calculatepadding{$\begin{array}{ccccc}
x_0&x_1&x_2&x_3&y \\ \hline
                                                             1&1&1&1&2 \end{array}$}\padding
&\begin{minipage}{4.5cm}\begin{multline*}y(cy+Q)-C+\\+x_3\bigl((x_0+cx_1)y+x_1Q+B\bigr)\end{multline*}\end{minipage}&
$y+x_1x_3$
  \\
\oddrow $5^b$& $X_4\subset\PP(1^4,2)$&\vgap\calculatepadding{$\begin{array}{ccccc}
x_0&x_1&x_2&x_3&y \\ \hline
                                                             1&1&1&1&2 \end{array}$}\padding
&\begin{minipage}{4.5cm}\begin{multline*}y(cy-Q)-C+\\+x_3\bigl((x_0+cx_1)y-x_0Q+B\bigr)\end{multline*}\end{minipage}&
$y+x_0x_3$
  \\
\addlinespace[1.1ex]
\bottomrule
\end{longtable}
\end{footnotesize}
\end{center}

\paragraph{\textbf{How to read Table~\ref{table:1}}} 

The rows of Table~\ref{table:1} display pairs $(X^\dagger,D^\dagger)$
birational to the pair $(\PP^3, D)$ of Theorem~C. 

The first row of Table~\ref{table:1} displays the pair $(\PP^3,D)$ itself. We have
chosen homogeneous coordinates $x_0, \dots, x_3$ on $\PP^3$ such that
the singular point $z\in D$ is the point $z=[0:0:0:1]$, and the
tangent cone to $D$ at $z$ is $(x_0x_1=0)\subset \PP^3$. In these
coordinates, the equation of $D$ is written as
\[
D = \Bigl(x_0x_1x_3^2+Bx_3+C
=0\Bigr) \subset \PP^3,
\]
where $B=B_3(x_0,x_1,x_2)$ and $C=C_4(x_0,x_1,x_2)$ are homogeneous
forms of degree~$3$ and~$4$. The information displayed on the first
row is self-explanatory: the second column states that the pair
$(X^\dagger,D^\dagger)$ is contained in the ambient $\PP^3$, the
third column names homogeneous coordinates on $\PP^3$, all of weight one, the
fourth column displays the equation of $X^\dagger\subset \PP^3$ --- the zero
``equation'', since $X^\dagger =\PP^3$ in this
case --- and the last column displays the equation of $D^\dagger$ that we just
explained.

For all integers $k\geq 0$, we denote by
$\FF_k^3\to \PP^2$ the $\PP^1$-bundle $\PP\big(\cO\oplus \cO(k)\big)$ over $\PP^2$
with homogeneous coordinates and weights
\[
\begin{array}{ccccc}
x_0 & x_1 & x_2 & x_3 & x \\
\hline
1 & 1 & 1 & 0 & -k \\
0 & 0 & 0 & 1 &  1 
\end{array}
\]
The second row of the table displays object~$2$, that is, the pair
$(\FF^3_1,D^\dagger)$ where the equation of $D^\dagger$ is as
shown in the last column. 

In the row that displays object~$4$, $L=L(x_0,x_1,x_2)$ is a homogeneous linear
form; in the rows that display objects~$5^a$ and~$5^b$, $Q=Q(x_0,x_1,x_2)$ is
a homogeneous quadratic form and $c\neq 0$.  

\begin{rem}
  At the end of Section~\ref{sec:sarkisov}, we exhibit explicit volume
  preserving birational maps between the Mf CY pairs in
  Table~\ref{table:1}.  These maps are constructed as Sarkisov links,
  and this is how they naturally appear in the proof of Theorem~C. In
  particular, our constructions show that object~$2$ is square
  equivalent to objects~$2^a$ and~$2^b$. Moreover, in
  Example~\ref{rem:objects5} we show that families~$5^a$ and~$5^b$ are
  isomorphic. This is why objects~$2^a$, $2^b$, and the family
  $5^b$ appear in Table~\ref{table:1} while they are omitted in the
  statement of Theorem~C.
\end{rem}

\begin{rem}
  In order to prove Theorem~B and Theorem~C, we consider a Mf CY pair
  $(Y,D_Y)/T$, and a volume preserving birational map
  $\varphi\colon (\mathbb{P}^3,D)\dasharrow (Y,D_Y)$ that is not
  biregular. The proofs proceed by studying explicitly the links of a
  Sarkisov factorisation of $\varphi$. To control these links, and the
  divisorial contractions, flips, flops and antiflips that constitute
  them, we need some explicit classification results for divisorial
  contractions and analytic neighbourhoods of curves. We develop this
  material further than strictly needed for the proof of Theorem~B,
  enough for what we need in the proof of Theorem~C. These results,
  contained in Section~\ref{sec:extr-contr}, are of
  independent interest; they represent the first steps towards
  developing a technology for working explicitly with the volume
  preserving birational maps of \mbox{$3$-fold} Mf CY pairs.
\end{rem}

\paragraph{\textbf{Notation}}
    If $Z$ is a normal projective variety, and $D_1, D_2$ are $\RR$-Cartier divisors on $Z$, the notation $D_1\equiv D_2$ means $D_1$ is numerically equivalent to $D_2$.

    We denote by $N^1(Z)$ the vector space of $\RR$-Cartier divisors on $Z$ modulo numerical equivalence, and by $N_1(Z)$ the dual space of $1$-cycles with real coefficients on $Z$ modulo numerical equivalence.

    We denote by $\Eff (Z)\subset N^1(Z)$ the convex cone spanned by effective divisors, and by $\overline{\Eff} (Z)$ its closure. 

\medskip

\paragraph{\textbf{Acknowledgements}}

Carolina Araujo was partially supported by grants from CNPq, Faperj and CAPES/COFECUB. 
Part of this work was developed during the authors' visit
to ICTP, funded by Carolina Araujo's ICTP Simons Associateship. We
thank ICTP for the great working conditions, and Simons Foundation for
the financial support.

Alessio Corti was partially supported by EPSRC Programme Grant
EP/N03189X/1 \emph{Classification, computation and construction: new
  problems in geometry}. This research was started during a visit of
Alessio Corti to IMPA funded by CNPq Visiting Researcher
grant. 

Alex Massarenti is a member of the Gruppo Nazionale per le Strutture
Algebriche, Geometriche e le loro Applicazioni of the Istituto
Nazionale di Alta Matematica F. Severi (GNSAGA-INDAM).

We are particularly grateful to an anonymous referee who studied an earlier version very carefully and sent us a long list of corrections and recommendations. This revision benefits from those comments; the errors that remain are of course ours.

\section{The Sarkisov program for Mf CY pairs}
\label{sec:sarkisov}

In this section we review the factorisation theorem for volume preserving birational
maps between Mf CY pairs established in~\cite{CK16}. 
This is the main tool in the proof of our results.
We start by recalling a valuative interpretation of the volume preserving condition.  

\begin{pro}[{\cite[Remark~1.7]{CK16}}]
  \label{pro:equivalentconditions}
Let $(X,D_X)$ and $(Y,D_Y)$ be CY pairs, and $f \colon X \dasharrow Y$ an arbitrary birational map. 
The following conditions are equivalent:
\begin{enumerate}[(1)]
\item The map $f\colon (X,D_X)\dasharrow(Y,D_Y)$ is volume preserving.
\item For every divisorial valuation $E$, the discrepancies of $E$ with respect to the pairs $(X,D_X)$
  and $(Y,D_Y)$ are equal: $a(E, K_X+D_X)=a(E, K_Y+D_Y)$.
\item Let 
 \[
  \begin{tikzpicture}[xscale=1.5,yscale=-1.2]
    \node (A0_1) at (1, 0) {$W$};
    \node (A1_0) at (0, 1) {$X$};
    \node (A1_2) at (2, 1) {$Y$};
    \path (A1_0) edge [->,dashed]node [auto] {$\scriptstyle{f}$} (A1_2);
    \path (A0_1) edge [->]node [auto] {$\scriptstyle{q}$} (A1_2);
    \path (A0_1) edge [->]node [auto,swap] {$\scriptstyle{p}$} (A1_0);
  \end{tikzpicture}
  \]
  be a common log resolution of the pairs $(X,D_X)$ and $(Y,D_Y)$. The
  birational map $f$ induces an isomorphism
  $f_* \colon \Omega^n_{k(X)/k}\rightarrow\Omega^n_{k(Y)/k}$, where
  $n$ denotes the common dimension of $X$ and $Y$.  Abusing notation,
  we write
$$
p^{*}(K_X+D_X)= q^{*}(K_Y+D_Y)
$$
to mean that for all meromorphic differentials $\omega\in \Omega^n_{k(X)/k}$, we have
$$
p^*\big(D_X + \divisor_X(\omega)\big) = q^*\big(D_Y + \divisor_Y(f_*\omega)\big).
$$
The condition is: for some (or equivalently for all) common log
resolution as above, we have $p^{*}(K_X+D_X)= q^{*}(K_Y+D_Y)$.
\end{enumerate} 
\end{pro}

\begin{say}\label{sec:terminology-1}
A \emph{Mori divisorial contraction} is a divisorial
contraction $f\colon Z \to X$ from a $\QQ$-factorial terminal
variety $Z$, associated to an extremal ray $R\subset \NEbar(Z)$ such that $K_Z\cdot R<0$. In
particular, $X$ also has $\QQ$-factorial terminal singularities.

If $(Z, D_Z)$ and $(X,D_X)$ are (t,~lc) CY pairs, then a Mori divisorial contraction $f\colon Z \to X$
is volume preserving as a map of CY pairs if and only if 
$K_Z+D_Z=f^* (K_X+D_X)$, in the sense of Proposition~\ref{pro:equivalentconditions}(3).
In this case, we have $D_X=f_* D_Z$.

A \emph{Mori flip} is a flip $\varphi\colon Z \dasharrow Z^\prime$
from a $\QQ$-factorial terminal variety $Z$, associated to an extremal
ray $R\subset \NEbar(Z)$ such that $K_Z\cdot R<0$.  In particular,
$Z^\prime$ also has $\QQ$-factorial terminal singularities.  An
\emph{antiflip} is the inverse of a Mori flip.  A \emph{Mori flop} is
a flop $\varphi\colon Z \dasharrow Z^\prime$ between $\QQ$-factorial
terminal varieties, associated to an extremal ray $R\subset \NEbar(Z)$
such that $K_Z\cdot R=0$.

Let $(Z,D_Z)$ and $(Z^\prime, D_{Z^\prime})$ be (t,~lc) CY pairs, and 
$\varphi\colon Z \dasharrow Z^\prime$ a Mori flip, flop or antiflip.
Then $\varphi\colon (Z,D_Z) \dasharrow (Z^\prime, D_{Z^\prime})$
is volume preserving if and only if $D_{Z^\prime}=\varphi_* D_Z$.
\end{say}

Next we recall the definition of the four types of Sarkisov links from \cite{Co95}. 
These links are constructed by running a relative log MMP with respect to a suitable fibration of relative Picard rank $2$, called a 
\emph{rank 2 fibration}. 
This rank 2 fibration can be obtained from a Mori fiber space $X\to S$ in two different ways. 
In the first two types of link, the rank 2 fibration $Z\to S$ is obtained by composing a suitable Mori divisorial contraction $Z\to X$ with the Mori fiber space structure $X\to S$. 
In the remaining two types of link, the rank 2 fibration $X\to T$ is obtained by composing the Mori fiber space structure $X\to S$ with a suitable fibration $S\rightarrow T$ of relative Picard rank $1$.
In either case, the link is completely determined by the rank 2 fibration by running a relative log MMP. This process is usually referred to as the \emph{2-ray game}.

\begin{say}[Sarkisov links]\label{links}
In what follows, $X\to S$ and  $X'\to S'$ always stand for Mori fiber spaces. 
 \begin{enumerate}[(I)]
  \item A \emph{Sarkisov link of type~(I)} is a commutative diagram
  \[
  \begin{tikzpicture}[xscale=1.5,yscale=-1.2]
    \node (A0_0) at (0, 0) {};
    \node (A0_1) at (1, 0) {$Z$};
    \node (A0_2) at (2, 0) {$X'$};
    \node (A1_0) at (0, 1) {$X$};
    \node (A1_2) at (2, 1) {$S'$};
    \node (A2_0) at (0, 2) {$S$};
    \path (A1_2) edge [->]node [auto] {$\scriptstyle{}$} (A2_0);
    \path (A0_1) edge [->]node [auto] {$\scriptstyle{}$} (A1_0);
    \path (A0_1) edge [->,dashed]node [auto] {$\scriptstyle{}$} (A0_2);
    \path (A0_2) edge [->]node [auto] {$\scriptstyle{}$} (A1_2);
    \path (A1_0) edge [->]node [auto] {$\scriptstyle{}$} (A2_0);
  \end{tikzpicture}
  \]
where $Z\to X$ is a Mori divisorial contraction, and
$Z\dasharrow X^\prime$ is a sequence of Mori flips,
flops and antiflips.
\item A \emph{Sarkisov link of type~(II)} is a commutative diagram
  \[
  \begin{tikzpicture}[xscale=1.5,yscale=-1.2]
    \node (A0_1) at (1, 0) {$Z$};
    \node (A0_2) at (2, 0) {$Z'$};
    \node (A1_0) at (0, 1) {$X$};
    \node (A1_3) at (3, 1) {$X'$};
    \node (A2_0) at (0, 2) {$S$};
    \node (A2_3) at (3, 2) {$S$};
    \path (A0_1) edge [->]node [auto] {$\scriptstyle{}$} (A1_0);
    \path (A1_3) edge [->]node [auto] {$\scriptstyle{}$} (A2_3);
    \path (A2_0) edge [-,double distance=1.5pt]node [auto] {$\scriptstyle{}$} (A2_3);
    \path (A1_0) edge [->]node [auto] {$\scriptstyle{}$} (A2_0);
    \path (A0_2) edge [->]node [auto] {$\scriptstyle{}$} (A1_3);
    \path (A0_1) edge [->,dashed]node [auto] {$\scriptstyle{}$} (A0_2);
  \end{tikzpicture}
  \]
where $Z\to X$ and $Z' \to X'$ are Mori divisorial
contractions, and $Z\dasharrow Z'$ is a sequence of Mori flips,
flops and antiflips.
Moreover, under their natural identification, the image of the effective cone $\Eff(Z)=\Eff(Z')$ in $N^1(Z/S)=N^1(Z'/S)$ is generated by the classes of the exceptional divisors of $Z\to X$ and $Z' \to X'$.
\item A \emph{Sarkisov link of type~(III)} is the inverse of a link of type~(I).
\item A \emph{Sarkisov link of type~(IV)} is a commutative diagram
  \[
  \begin{tikzpicture}[xscale=1.5,yscale=-1.2]
    \node (A0_0) at (0, 0) {$X$};
    \node (A0_2) at (2, 0) {$X'$};
    \node (A1_0) at (0, 1) {$S$};
    \node (A1_2) at (2, 1) {$S'$};
    \node (A2_1) at (1, 2) {$T$};
    \path (A1_2) edge [->]node [auto] {$\scriptstyle{}$} (A2_1);
    \path (A0_0) edge [->]node [auto] {$\scriptstyle{}$} (A1_0);
    \path (A1_0) edge [->]node [auto] {$\scriptstyle{}$} (A2_1);
    \path (A0_2) edge [->]node [auto] {$\scriptstyle{}$} (A1_2);
    \path (A0_0) edge [->,dashed]node [auto] {$\scriptstyle{}$} (A0_2);
  \end{tikzpicture}
  \]
  where $X\dasharrow X'$ is a sequence of Mori flips, flops and
  antiflips, and $S\rightarrow T$ and $S'\rightarrow T$ are fibrations of relative Picard rank $1$.
  Moreover, under their natural identification, the image of the effective cone $\Eff(X)=\Eff(X')$ in $N^1(X/T)=N^1(X'/T)$ is generated by the pullback classes of ample divisor on $S$ and $S'$.
  \end{enumerate}
\end{say}

It is shown in \cite{Co95} and \cite{HM13} that every birational map
between Mori fiber spaces is a composition of Sarkisov links.  Next we
explain the version of the Sarkisov program for volume preserving
birational maps between Mf CY pairs established in~\cite{CK16}.

\begin{Definition}[{\cite[Definition 1.12 and Remark 1.13]{CK16}}] \label{dfn:4} A \emph{volume preserving Sarkisov
    link} is a Sarkisov link as described in Paragraph~\ref{links}
  above with the following additional data and property: there are
  divisors $D_X$ on $X$, $D_{X'}$ on $X'$, $D_Z$ on $Z$, and $D_{Z'}$
  on $Z'$, making $(X,D_X)$, $(X',D_{X'})$, $(Z,D_Z)$ and
  $(Z',D_{Z'})$ (t,~lc) CY pairs, and all the divisorial contractions,
  Mori flips, flops and antiflips that constitute the Sarkisov link
  are volume preserving for these CY pairs.
\end{Definition}

At the end of this section we will construct explicit volume
preserving Sarkisov links between the Mf CY pairs displayed in Table~1
of the introduction.

\begin{thm}[{\cite[Theorem 1.1]{CK16}}]\label{fact}
  A volume preserving birational map between Mf CY pairs is a
  composition of volume preserving Sarkisov links.
\end{thm}

Theorem~\ref{fact} provides an effective tool to investigate the group
$\Bir(X,D)$ when $(X,D)$ is a Mf CY pair. In this paper we restrict
ourselves to canonical pairs. This is the case, for instance, when $X$
is smooth and $D\subset X$ is an irreducible hypersurface with
canonical singularities.  In this case, the theory is greatly
simplified for the following reason.

\begin{Proposition}\label{group1}\label{group2} \label{group3}
  Let $(X,D_X)$ and $(Y,D_Y)$ be (t,~c) CY pairs, and
  $f\colon X\dasharrow Y$ an arbitrary birational map.  Then
  $f\colon (X,D_X)\dasharrow(Y,D_Y)$ is volume preserving if and only
  if $f_{*}D_X = D_Y$ and $f^{-1}_{*}D_Y = D_X$. (This condition is
  equivalent to asking that the restriction of $f$ to each component
  of $D_X$ is a birational map to a component of $D_Y$, and the same
  for $f^{-1}$).

In particular, if $(X,D)$ is a (t,~c)  CY pair with $D$ irreducible, then
$\Bir (X,D)$ coincides with the group of birational self-maps $f\colon X \dasharrow
X$ such that $f(D)= D$. 
\end{Proposition}

\begin{Remark}
  \begin{enumerate}[(1)]
  \item In many cases of interest, the assumption that $(X,D)$ is
    (t,~c) implies that $D$ is irreducible.  This is the case for
    instance when $X$ is Fano with $\dim(X)>1$.  Indeed, if $(X,D)$ is
    (t,~c), then $D$ is normal, and thus irreducibility of $D$ is
    equivalent to connectedness of $D$. A key case when
    connectedness fails is $X=\PP^1$, $D=\{0,\infty\}$, and one feels
    that all examples must in some way be related to this (see the
    main result of~\cite{MR3905556} for a statement along these lines).
  \item The characterisation of the volume preserving condition stated
    in Proposition~\ref{group3} does not hold in general for
    birational maps between (t,~lc) CY pairs. For example, consider
    the divisor $D=L_0+ L_1+ L_2$ on $\mathbb{P}^2$, where the $L_i$
    are the three coordinate lines.  It is (t,~lc) but not (t,~c).
    One checks easily that the standard Cremona transformation
$$
\begin{array}{cccc}
f\colon & (\mathbb{P}^2,D) & \dasharrow & (\mathbb{P}^2,D)\\
 & (x_0:x_1:x_2) & \longmapsto & (x_1x_2:x_0x_2:x_0x_1)
\end{array}
$$ 
is volume preserving, but it does not restrict to a birational
self-map of $D$.  Indeed, $f=f^{-1}$ contracts the lines $L_i$ to the
three coordinate points of $\P^2$.
  \end{enumerate}
\end{Remark}

\begin{proof}[Proof of Proposition~\ref{group3}]
If $(X,D_X)$ is a (t,~c)  pair, then any divisor $E$ over $X$ such that $a(E, K_X+D_X)=-1$ must be a component of $D_X$. 
Therefore, if $f\colon(X,D_X)\dasharrow (Y,D_Y)$ is a volume preserving birational map between (t,~c)  CY pairs, then $f$ does not contract any component of $D_X$, and $f^{-1}$ does not contract any component of $D_Y$. Hence, $f_{*}D_X = D_Y$ and $f^{-1}_{*}D_Y = D_X$.

To prove the converse, 
consider a common log resolution
 \[
  \begin{tikzpicture}[xscale=1.5,yscale=-1.2]
    \node (A0_1) at (1, 0) {$W$};
    \node (A1_0) at (0, 1) {$X$};
    \node (A1_2) at (2, 1) {$Y$\, .};
    \path (A1_0) edge [->,dashed]node [auto] {$\scriptstyle{f}$} (A1_2);
    \path (A0_1) edge [->]node [auto] {$\scriptstyle{q}$} (A1_2);
    \path (A0_1) edge [->]node [auto,swap] {$\scriptstyle{p}$} (A1_0);
  \end{tikzpicture}
  \]
Assume that $f_{*}D_X = D_Y$ and $f^{-1}_{*}D_Y = D_X$. Then the strict transforms of $D_X$ via $p$ and of $D_Y$ via $q$ coincide. We denote them by $D_W$. 

Since $(X,D_X)$ and $(Y,D_Y)$ are canonical, we have
\begin{itemize}
\item[(i)] $K_W + D_W = p^{*}(K_X+D_X)+ \sum_{i}a_iE_i \sim \sum_{i}a_iE_i$, where the $E_i$ are $p$-exceptional and $a_i\geq 0$;
\item[(ii)] $K_W + D_W = q^{*}(K_Y+D_Y)+ \sum_{j}b_jF_j \sim \sum_{j}b_jF_j$, where the $F_j$ are $q$-exceptional and $b_j\geq 0$. 
\end{itemize}

Applying $q_{*}$ to (i) yields
$$
q_{*}(K_W + D_W) = \sum_{i}a_iq_{*}E_i \sim 0.
$$
This shows that, whenever $a_i> 0$, the divisor $E_i$ is $q$-exceptional.
Similarly,  whenever $b_j> 0$, the divisor $F_j$ is $p$-exceptional.
So, up to relabelling, we may assume that $E_i=F_i$, and we have 
$$
\sum_{i}a_iE_i \sim \sum_{i}b_iE_i.
$$
The negativity lemma \cite[Lemma 3.39]{KM98} then implies that $a_i = b_i$, and so $p^{*}(K_X+D_X) = q^{*}(K_Y+D_Y)$.
\end{proof}

The next lemma ensures that canonicity is preserved when we run a volume preserving Sarkisov program for canonical Mf CY pairs.

\begin{Lemma}\label{D_normal}
Let $f\colon (X,D_X)\map (Y,D_Y)$ be a volume preserving birational
map between $(t, ~lc)$ CY pairs. 
Then $(X,D_X)$ is canonical if and only if so is
$(Y,D_Y)$.
\end{Lemma}

\begin{proof}
First of all, note that if $(X,D_X)$ is canonical, then $D_X$ is normal. Write
\[
D_X=\sum D_i,
\] 
where the $D_i$ are the connected components of $D_X$. The $D_i$ are the only
divisorial valuations with $a(D_i,K_X+D_X)<0$. 
Since the map is volume
preserving, for every divisorial valuation $E$, $a(E,K_Y+D_Y)=a(E,
K_X+D_X)$. To show that the pair $(Y,D_Y)$ is canonical, all we need to show is that no $D_i$ is 
exceptional over $Y$. Consider a common log resolution
 \[
  \begin{tikzpicture}[xscale=1.5,yscale=-1.2]
    \node (A0_1) at (1, 0) {$W$};
    \node (A1_0) at (0, 1) {$X$};
    \node (A1_2) at (2, 1) {$Y\, ,$};
    \path (A1_0) edge [->,dashed]node [auto] {$\scriptstyle{f}$} (A1_2);
    \path (A0_1) edge [->]node [auto] {$\scriptstyle{q}$} (A1_2);
    \path (A0_1) edge [->]node [auto,swap] {$\scriptstyle{p}$} (A1_0);
  \end{tikzpicture}
  \]
denote by $\widetilde D_i$ the strict transform of $D_i$, and by
$D_W= \sum \widetilde D_i$ the strict transform of $D_X$.
In order to show that $(Y,D_Y)$ is canonical, we need to show that no $\widetilde D_i$ is $q$-exceptional.

Since $(X,D_X)$ is (t,~c), we have
\[
K_W+D_W=p^* (K_X+D_X)+\sum a_j E_j,
\]
with $a_j\geq 0$. If $a_j>0$, then $E_j$ is $p$-exceptional. 
The volume preserving condition says 
that $p^* (K_X+D_X)=q^* (K_Y+D_Y)$, and hence
\[
K_W+D_W = q^* (K_Y+D_Y) +\sum a_j E_j.
\]
Since $q_{*}D_W = D_Y$, if $a_j>0$ then $E_j$ is
$q$-exceptional. 
By~\cite[Theorem~5.48]{KM98}, the support of $D_W$ is connected in the
neighbourhood of every fibre of $q$. 
Since the $\widetilde D_i$ are pairwise disjoint, connectedness implies that, if 
$\widetilde D_i$ is $q$-exceptional, then it is not contained in the support of 
$q^* D_Y$. Hence, if 
$\widetilde D_i$ is $q$-exceptional, then $-1=a(\widetilde D_i, K_Y+D_Y)=a(\widetilde D_i, K_Y)$,
contradicting the assumption that $Y$ has terminal singularities.
We conclude that no $\widetilde D_i$ is $q$-exceptional, and thus $(Y,D_Y)$ is canonical. 
\end{proof}

\medskip

In the remaining part of this section, we revisit the Mf CY pairs in
Table~\ref{table:1}, and construct explicit volume preserving Sarkisov
links connecting them.  These links will appear crucially in the proof
of Theorem~C.  They are summarised in the following diagram:
$$
\xymatrix{\text{obj.~$1$}\ar@{-->}[dr]_{\epsilon_a, \epsilon_b} &
  \text{obj.~$2$} \ar@{->}[l]_\sigma & 
\text{obj.~$2^a$, $2^b$} \ar@{-->}[l]_{\nu^a, \nu^b}
\ar@{->}[dl]^{\chi^a, \chi^b}\\
 & \text{obj.~$3^a$, $3^b$} \ar@{-->}[dl]_{\phi^a, \phi^b}
 \ar@{-->}[dr]^{\psi^a,\widetilde{\psi}^b}& \\
 \text{obj.~$4$}& &\text{obj.~$5^a$} }
$$
In Example~\ref{rem:objects5} we produce an isomorphism between the
family~$5^a$ and the family~$5^b$.

\paragraph{\textbf{Object~1}} The pair $(\PP^3,D_4)$, where 
\[
D_4 = \Bigl(x_0x_1x_3^2+Bx_3+C
=0\Bigr).
\]
Here and in what follows, $B=B_3(x_0,x_1,x_2)$ and $C=C_4(x_0,x_1,x_2)$ are fixed homogeneous forms of the indicated degree. 
Recall that $D_4$ is assumed to be smooth apart from a single $A_2$-singularity at $[0:0:0:1]$.
\paragraph{\textbf{Object~2}} The pair $\left(\FF^3_1, D_{\binom{2}{2}}\right)$, where $\FF^3_1$ is 
the blow-up of $\PP^3$ at a point or, equivalently,
the 
$\PP^1$-bundle $\PP\big(\cO\oplus \cO(1)\big)$ over $\PP^2$
with homogeneous coordinates and weights
\[
\begin{array}{ccccc}
x_0 & x_1 & x_2 & x_3 & x \\
\hline
1 & 1 & 1 & 0 & -1 \\
0 & 0 & 0 & 1 &  1 
\end{array} \, ,
\]

\[
\text{ and } \ \ \ D_{\binom{2}{2}}=
\Bigl(
x_0x_1x_3^2+B x_3x+C x^2
=0\Bigr).
\]

\begin{exa}[Map from object~$2$ to object~$1$]
  \label{link:1->2}
The blow-up $\sigma \colon \left(\FF_1^3,D_{\binom{2}{2}}\right) \to (\PP^3, D_4)$ 
of the point $[0:0:0:1]\in \PP^3$ is a volume preserving Sarkisov link of type~(I). 
In coordinates,  $\sigma \colon {\FF_1^3}_{(x_0,x_1,x_2,x_3,x)} \to {\PP^3}_{(x_0,x_1,x_2,x_3)}$ is given by 
$$
(x_0,x_1,x_2,x_3,x) \mapsto \left(x_0,x_1,x_2,\frac{x_3}{x}\right).
$$
\end{exa}
\paragraph{\textbf{Object~$2^a$}} The pair $\left(\FF^3_2, D_{\binom{1}{2}}^a\right)$, where
$\FF^3_2$ is the $\PP^1$-bundle $\PP\big(\cO\oplus \cO(2)\big)$ over $\PP^2$
with homogeneous coordinates and weights
\[
\begin{array}{ccccc}
x_0 & x_1 & x_2 & x_3 & x \\
\hline
1 & 1 & 1 & 0 & -2 \\
0 & 0 & 0 & 1 &  1 
\end{array} \, ,
\]

\[
\text{ and } \ \ \ D_{\binom{1}{2}}^a= 
\Bigl(x_0x_3^2+Bx_3x+x_1Cx^2
=0\Bigr).
\]

\begin{exa}[Map from object~$2^a$ to object $2$]
  \label{link:2flat->2}
Consider the rational map 
\[
\nu^a \colon \left(\FF_2^3,D_{\binom{1}{2}}^a\right)\dasharrow  \left(\FF_1^3,D_{\binom{2}{2}}\right) 
\]
obtained by blowing-up the curve 
$(x_1=x_3=0)\subset D_{\binom{1}{2}}^a$, and then blowing-down the strict transform of the 
divisor $(x_1=0)$ onto the curve $(x_1 = x = 0)\subset D_{\binom{2}{2}}$. 
In coordinates, $\nu^a \colon {\FF_2^3}_{(x_0,x_1,x_2,x_3,x)}\dasharrow {\FF_1^3}_{(x_0,x_1,x_2,x_3,x)}$
is given by 
$$
(x_0,x_1,x_2,x_3,x) \mapsto (x_0,x_1,x_2,x_3, xx_1).
$$
It is a volume preserving Sarkisov link of type~(II).

 Similarly, there is a volume preserving Sarkisov link of type~(II) $\nu^b
  \colon \text{obj.~$2^b$}\dasharrow \text{obj.~$2$}$.

Note that the Mf CY pairs ~$2^a$, $2^b$ and $2$ are all square equivalent. 
\end{exa}

\paragraph{\textbf{Object~$3^a$}} The pair $(\PP(1^3,2), D_5^a)$, where
\[
D_5^a=\Bigl( x_0y^2+B y+x_1C =0\Bigr).
\]
Note that $D_5^a\subset \PP(1^3,2)$ is not a
general degree five surface in $\PP(1^3,2)$. Indeed $D_5^a$ always
contains the curve $y=x_1=0$.

\paragraph{\textbf{Object~$3^b$}} The pair $(\PP(1^3,2), D_5^b)$, where
\[
D_5^b=\Bigl(x_1y^2+By+x_0C
=0\Bigr).
\]

  Object~$3^b$ is object~$3^a$ with $x_0,x_1$ swapped in the same way as
  object~$2^b$ is object~$2^a$ with $x_0,x_1$
  swapped. However, in general, object~$3^b$ is not square equivalent to object~$3^a$.
  Indeed, since $X^\dagger=\PP(1,1,1,2)$ is a Mori fibre space
  over $\Spec \CC$, the two objects are square equivalent if and only
  if they are isomorphic.

\begin{exa}[Map from object~$1$ to object~$3^b$]
  \label{link:1->3a}
 Consider the birational map $\epsilon_b \colon (\PP^3, D_4)\dasharrow (\PP(1^3, 2),D_5^b)$
obtained by a weighted blow-up of $[0:0:0:1]\in \PP^3$ with weights $(2,1,1)$, followed by the contraction of the strict transform of the divisor $(x_0=0)$ onto the curve $(y = x_0 = 0)\subset D_5^b$. 
In coordinates, $\epsilon_b \colon \PP^3_{(x_0,x_1,x_2,x_3)}\dasharrow  \PP(1^3,2)_{(x_0,x_1,x_2,y)}$
is given by 
$$(x_0,x_1,x_2,x_3)\mapsto (x_0,x_1,x_2,x_3x_0).$$
It is a volume preserving Sarkisov link of type~(II). Similarly, there
is a volume preserving Sarkisov link of type~(II),
$\epsilon_a\colon \text{obj.~$1$}\dasharrow \text{obj.~$3^a$}$.
\end{exa}

\begin{exa}[Maps from object~$2^b$ to object~$3^b$ and from object~$2^a$ to
  object~$3^a$]
  \label{link:3a->2sharp}
The blow-up $\chi^b\colon \left(\FF^3_2, D_{\binom{1}{2}}^b\right)\to (\PP(1^3, 2), D_5^b)$ of the singular point $[0:0:0:1]\in \PP(1^3, 2)$  is a volume preserving Sarkisov link of type~(I). 
In coordinates, $\chi^b\colon {\FF_2^3}_{(x_0,x_1,x_2,x_3,x)}\to \PP(1^3, 2)_{(x_0,x_1,x_2,y)}$ is given by 
$$(x_0,x_1,x_2,x_3,x)\mapsto \left(x_0,x_1,x_2,\frac{x_3}{x}\right).$$
From the description in coordinates, it is straightforward to check
that $\chi^b=\epsilon_b \circ \sigma \circ \nu^b$.
Similarly, there is a volume preserving Sarkisov link of type~(I),
$\chi^a \colon \text{obj.~$2^a$}\to \text{obj.~$3^a$}$, and
$\chi^a=\epsilon_a \circ \sigma \circ \nu^a$.
\end{exa}

\paragraph{\textbf{$3$-parameter family of objects~4}} 
Let $D_{3,4}\subset \PP(1^3,2^2)_{(x_0,x_1,x_2,y_0,y_1)}$ be the complete
intersection given by equations:
\[
D_{3,4}=
\begin{cases}
  y_0y_1+C&=0,\\
  x_0y_1-x_1y_0-B&=0.
\end{cases}
\]
For each fixed linear form $L=L(x_0,x_1,x_2)$, we consider the pair
$(X_4,D_{3,4})$, where $X_4$ is the following quartic containing $D_{3,4}$:
\[
X_4=\Bigl(y_0y_1+C-L(x_0y_1-x_1y_0-B)
=0\Bigr).
\]

\begin{exa}[Map from object~$3^b$ to object~$4$]
  \label{link:3b->4}
For each fixed linear form $L=L(x_0,x_1,x_2)$, we construct a volume preserving
Sarkisov link of type~(II),  $\phi^b\colon (\PP(1^3,2),D_5^b)\dasharrow (X_4, D_{3,4})$.
In coordinates, the map $\phi^b \colon \PP(1^3,2)_{(x_0,x_1,x_2,y)}\dasharrow
\PP(1^3,2^2)_{(x_0, x_1, x_2, y_0, y_1)}$ is given by
\begin{equation}\label{phi^b}
  (x_0, x_1, x_2, y) \mapsto \left(x_0, x_1, x_2, y , -x_1L -\frac{x_0x_1L^2+BL+C}{y-x_0L}\right).
\end{equation}
This map is obtained by first blowing-up the curve $\Gamma \subset \PP(1,1,1,2)$ defined by:
\[
\Gamma = 
\left\lbrace\begin{array}{l}
Q = y-x_0L = 0,\\ 
F = x_0x_1L^2+BL+C=0,
\end{array}\right. 
\]
and then blowing-down the strict transform of the divisor
$(y-x_0L=0)$.

In what follows, we describe this map in detail, and deduce expression \eqref{phi^b} above. 

First note that $\Gamma\subset D_5^b$:
\begin{equation}
  \label{eq:1}
  x_1y^2+By+x_0C=Q\left(x_1(y+x_0L)+B\right)+Fx_0.
\end{equation}
In order to describe the blow-up of $\PP(1,1,1,2)$ along $\Gamma$, 
consider the toric variety $\FF$ with coordinates and weight matrix:
$$
\begin{array}{cccccc}
z_0 & z_1 & z_2 & z_3 & u & v\\ 
\hline
1 & 1 & 1 & 2 & 0 & -2 \\ 
0 & 0 & 0 & 0 & 1 & 1
\end{array}
$$ 
and stability condition chosen so that the nef cone of $\FF$ is the
span $\langle z_i,u \rangle_+$:
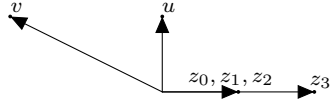
$$
\begin{tikzpicture}[line cap=round,line join=round,>=triangle 45,x=1.0cm,y=1.0cm]
\clip(-2.2,-0.1) rectangle (2.2,1.2);
\draw [->,line width=0.4pt] (0.,0.) -- (1.,0.);
\draw [->,line width=0.4pt] (0.,0.) -- (2.,0.);
\draw [->,line width=0.4pt] (0.,0.) -- (0.,1.);
\draw [->,line width=0.4pt] (0.,0.) -- (-2.,1.);
\begin{scriptsize}
\draw [fill=black] (1.,0.) circle (0.5pt);
\draw[color=black] (0.9,0.16) node {$z_0,z_1,z_2$};
\draw [fill=black] (2.,0.) circle (0.5pt);
\draw[color=black] (2.082474878015239,0.12781237882652488) node {$z_3$};
\draw [fill=black] (0.,1.) circle (0.5pt);
\draw[color=black] (0.08478057481659083,1.1207491922507038) node {$u$};
\draw [fill=black] (-2.,1.) circle (0.5pt);
\draw[color=black] (-1.9129137283820574,1.1207491922507038) node {$v$};
\end{scriptsize}
\end{tikzpicture}
$$
This choice of stability condition gives the irrelevant ideal
$(z_0,z_1,z_2,z_3)(u,v)$ and yields a $\PP^1$-bundle
morphism $\pi \colon \FF \to \PP(1,1,1,2)$.
The other contraction of $\FF$ is the divisorial
contraction $\pi^\prime \colon \FF\to \PP(1,1,1,2,2)$ that maps the divisor $v=0$ to the point $[0:0:0:0:1]\in
\PP(1,1,1,2,2)$. In coordinates,  $\pi^\prime \colon \FF_{(z_0,z_1,z_2,z_3,u,v)}\to \PP(1,1,1,2,2)_{(x_0,x_1,x_2,y_0,y_1)}$
is given by 
\[
(z_0,z_1,z_2,z_3,u,v) \mapsto \left(z_0,z_1,z_2,z_3,\frac{u}{v}\right).
\]
The blow-up $Z$ of $\mathbb{P}(1,1,1,2)$ along $\Gamma$ is cut out in $\FF$ by the equation
$$uQ+vF = 0.$$
Let us describe the equation of $\pi^\prime (Z)\subset \mathbb{P}(1,1,1,2,2)$. From $\frac{u}{v} = -\frac{F}{Q}$, we get
$y_1Q + F=0$ and hence $\pi^\prime (Z)$ is cut out by the equation
$$
y_1(y_0-x_0L) +x_0x_1L^2+BL+C=0.
$$
Combining with Equation~\eqref{eq:1}, we see that the strict
transform of $D_5^b$ in $\pi^\prime (Z)$ is cut out by the equation
\[
x_0y_1=x_1y_0+x_0x_1L+B.
\]
In coordinates, the composition $\pi^\prime\circ \pi \colon \PP(1^3,2)_{(x_0,x_1,x_2,y)}\to \pi^\prime (Z)\subset \PP(1,1,1,2,2)_{(x_0,x_1,x_2,y_0,y_1)}$
is given by 
\[
  (x_0, x_1, x_2, y) \mapsto \left(x_0, x_1, x_2, y , -\frac{F}{Q}\right).
\]
Next we compose it with the automorphism of $\PP(1,1,1,2,2)_{(x_0,x_1,x_2,y_0,y_1)}$ given in coordinates by 
\[
(x_0,x_1,x_2,y_0,y_1) \mapsto (x_0,x_1,x_2,y_0,y_1-x_1L).
\]
It is immediate to check that, in coordinates, the composed map $\phi^b \colon \PP(1^3,2)_{(x_0,x_1,x_2,y)}\dasharrow
\PP(1^3,2^2)_{(x_0, x_1, x_2, y_0, y_1)}$ is given by \eqref{phi^b}.
The image of $Z$ is given by the equation
\[
0= (y_1+x_1L)(y_0-x_0L)+x_0x_1L^2+BL+C=y_0y_1+C-L(x_0y_1-x_1y_0-B),
\]
which is precisely the equation of $X_4$. 
The strict transform of $D_5^b$ is the surface of $\PP(1^3,2^2)$ given by
the equations
\[
\begin{cases}
  y_0y_1+C-L(x_0y_1-x_1y_0-B) &=0,\\
  x_0(y_1+x_1L)-x_1(y_0+x_0L)-B&=x_0y_1-x_1y_0-B=0,
\end{cases}
\]
which are precisely the equations of $D_{3,4}$.

Similarly, one can construct a volume preserving Sarkisov link of type~(II),
$\psi^a\colon \text{obj.~$3^a$}\dasharrow \text{obj.~$4$}$.
\end{exa}

\paragraph{\textbf{7-parameter family of objects~$5^a$}} 
For each fixed quadratic form $Q=Q(x_0,x_1,x_2)$, we consider the pair
$(X_4,D_{2,4})$, where $X_4\subset \PP(1^4,2)_{(x_0,x_1,x_2,x_3,y)}$ is the hypersurface given by
equation
\[
y(cy+Q)-C+x_3\bigl((x_0+cx_1)y + x_1 Q+B \bigr)=0,
\]
and $D_{2,4}$ is cut out in $X_4$ by the equation $(y+x_1x_3=0)$.

\begin{rem}
  \label{rem:1}
  The variety $X_4\subset \PP(1^4,2)$ has a $cA_2$-singularity\footnote{A singularity
    that is locally analytically equivalent to $uv+w^3+z^3=0$ in
    $\CC^4$.}  at the point $[0:0:0:1:0]$ and $D_{2,4}$ is
  isomorphic to the original $D\subset \PP^3$, as one can check by
  plugging in $y=-x_1x_3$ and literally getting the equation of
  $D\subset \PP^3$ back.
\end{rem}

\begin{exa}[Isomorphism of objects~$5^a$ and~$5^b$]
  \label{rem:objects5}
  Table~1 also lists a 7-parameter family of objects~$5^b$. The substitution
\[
\widetilde{y}=-y-x_3(x_0+x_1)
\]
transforms object~$5^a$ isomorphically to object~$5^b$.

Indeed, rewriting the equation slightly and substituting, we get:
  \begin{multline*}
    y\Bigl(cy+Q\Bigr)-C+x_3\Bigl((x_0+cx_1)y+x_1Q+B \Bigr) =\\
    = y\Bigl(cy+x_3(x_0+cx_1)\Bigr) + Q \Bigl(y+x_1x_3\Bigr) -C+x_3B=\\=
     \Bigl(c\widetilde{y}+x_3(x_0+cx_1)\Bigr) \widetilde{y}
     -Q\Bigl(\widetilde{y}+x_0x_3\Bigr)-C+x_3B =\\=
      \widetilde{y}\Bigl(c\widetilde{y}-Q\Bigr)-C+x_3\Bigl((x_0+cx_1)\widetilde{y}-x_0Q+B \Bigr).     
  \end{multline*}
\end{exa}

\begin{exa}[Maps from objects~$3^\bullet$ to objects~$5^\bullet$]
  \label{link:3a->5}
  For each quadratic form $Q=Q(x_0,x_1,x_2)$, we construct a volume preserving Sarkisov link of type~(II),
 $\psi^a\colon (\PP(1^3,2), D_5^a)
  \dasharrow (X_4, D_{2,4})$.
 In coordinates, the map  
$\psi^a\colon \PP(1^3,2) _{(x_0,x_1,x_2,y)}\to\PP(1^4,2)_{(x_0,x_1,x_2,x_3,y)}$ is given by:
\begin{equation} \label{psi^b}
  (x_0,x_1,x_2,y) \mapsto \left(x_0, x_1, x_2,
  -\frac{y(cy+Q)-C}{x_0y+B+x_1(cy+Q)}, y\right).
\end{equation}
This map is obtained by first blowing-up the curve $\Gamma \subset  \PP(1^3,2)$ defined
by equations:
\[
\left\lbrace\begin{array}{l}
F_3 = x_0y+B+x_1(cy+Q)  = 0,\\ 
G_4 = y(cy+Q) -C= 0,
\end{array}\right. 
\] 
and then blowing-down the strict transform of the divisor $\big(x_0y+B+x_1(cy+Q)=0\big)$.

Let us describe this map in detail.
First note that $\Gamma\subset D_5^a$:
\begin{multline}
\label{eq:1bis}
x_0y^2+By+x_1C\ =\ x_0y^2+By+x_1y(cy+Q)-x_1y(cy+Q)+x_1C\ =\\=\ y(x_0y+B+x_1(cy+Q))-x_1(y(cy+Q) - C)\ =\ yF_3-x_1G_4.  
\end{multline}
In order to describe the blow-up of $\PP(1,1,1,2)$ along $\Gamma$, 
consider the toric variety $\FF$ with coordinates and weight matrix:
$$
\begin{array}{cccccc}
x_0 & x_1 & x_2 & y & u & v\\ 
\hline
1 & 1 & 1 & 2 & 0 & -1\\ 
0 & 0 & 0 & 0 & 1 & 1 
\end{array} \, ,
$$ 
and stability condition chosen so that the nef cone of $\FF$ is the
span $\langle x_i, u\rangle_+$: 
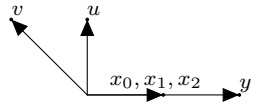
$$
\begin{tikzpicture}[line cap=round,line join=round,>=triangle 45,x=1.0cm,y=1.0cm]
\clip(-2.2,-0.1) rectangle (2.2,1.2);
\draw [->,line width=0.4pt] (0.,0.) -- (1.,0.);
\draw [->,line width=0.4pt] (0.,0.) -- (2.,0.);
\draw [->,line width=0.4pt] (0.,0.) -- (0.,1.);
\draw [->,line width=0.4pt] (0.,0.) -- (-1.,1.);
\begin{scriptsize}
\draw [fill=black] (1.,0.) circle (0.5pt);
\draw[color=black] (0.9,0.16) node {$x_0,x_1,x_2$};
\draw [fill=black] (2.,0.) circle (0.5pt);
\draw[color=black] (2.082474878015239,0.12781237882652488) node {$y$};
\draw [fill=black] (0.,1.) circle (0.5pt);
\draw[color=black] (0.08478057481659083,1.1207491922507038) node {$u$};
\draw [fill=black] (-1.,1.) circle (0.5pt);
\draw[color=black] (-0.9129137283820574,1.1207491922507038) node {$v$};
\end{scriptsize}
\end{tikzpicture}
$$
This choice of stability condition gives the irrelevant ideal
$(x_0,x_1,x_2,y)(u,v)$, and yields a $\PP^1$-bundle morphism
$\pi \colon \FF \to \PP(1,1,1,2)$.  The other contraction of $\FF$ is
the divisorial contraction $\pi^\prime \colon \FF\to \PP(1,1,1,1,2)$
that maps the divisor $(v=0)$ to the point
$[0:0:0:1:0]\in \PP(1,1,1,1,2)$. In coordinates,
$\pi^\prime \colon \FF_{(x_0,x_1,x_2,y,u,v)}\to
\PP(1,1,1,1,2)_{(x_0,x_1,x_2,x_3,y)}$ is given by
\begin{equation}
  \label{eq:tricky_blowup}
(x_0,x_1,x_2,y,u,v) \mapsto \left(vx_0,vx_1,vx_2, u, v^2y\right)\,.  
\end{equation}
The blow-up $Z$ of $\mathbb{P}(1,1,1,2)$ along $\Gamma$ is cut out in $\FF$ by the equation
$$
uF_3+vG_4 = 0. 
$$
Let us describe the equation of $\pi^\prime (Z)\subset \mathbb{P}(1,1,1,1,2)$. 
From $\frac{u}{v} = -\frac{G_4}{F_3}$, we get
$x_3F_3 + G_4=0$, and hence $\pi^\prime (Z)$ is cut out by the equation
\[
y(cy+Q)-C+x_3(x_0y+B+x_1(cy+Q))=0,
\]
which is precisely the equation of $X_4$.
Combining with Equation~\eqref{eq:1bis}, we see that the strict
transform of $D_5^a$ in $X_4$ is cut out by the equation
\[
y+x_1x_3=0,
\]
and so it is precisely the divisor $D_{2,4}$.

In coordinates, the composed map  
$\psi^a\colon \PP(1^3,2) _{(x_0,x_1,x_2,y)}\dasharrow\PP(1^4,2)_{(x_0,x_1,x_2,x_3,y)}$ is given by:
$$
  (x_0,x_1,x_2,y) \mapsto \left(x_0, x_1, x_2,
  -\frac{G_4}{F_3}, y\right),
$$
which is precisely \eqref{psi^b} above. 

Similarly, there is a volume preserving Sarkisov link of type~(II),
$\psi^b\colon \text{obj.~$3^b$}\dasharrow \text{obj.~$5^b$}$.

Finally, we denote by $\widetilde{\psi}^a\colon
\text{obj.~$3^a$}\dasharrow \text{obj.~$5^b$}$ the composition of
$\psi^a$ with the isomorphism in Example~\ref{rem:objects5}, and
similarly $\widetilde{\psi}^b\colon
\text{obj.~$3^b$}\dasharrow \text{obj.~$5^a$}$ the composition of
$\psi^b$ with the inverse of the isomorphism in Example~\ref{rem:objects5},
\end{exa}

\section{Proof of Theorem~A}\label{section_ThmA}

The prototype of CY pairs adressed in Theorem~A is
$(\mathbb{P}^n,D_{n+1})$, where $D_{n+1}$ is a hypersurface of degree
$n+1$ in $\mathbb{P}^n$, $n\geq 3$. The assumptions in particular say that
$D_{n+1}$ is factorial, has terminal singularities, and
$\Pic (D_{n+1})= \Cl (D_{n+1}) =  \big\langle \cO_{\mathbb{P}^n}(1)_{|D_{n+1}}
\big\rangle$.  Let us sketch the proof of Theorem~A in this case.

Suppose that $\psi \colon (\mathbb{P}^n,D_{n+1})/\Spec \CC\dasharrow (X,D)/T$ is a
volume-preserving birational map between Mf CY pairs that is not
biregular. Then the first step of a Sarkisov factorisation of $\psi$
is a divisorial contraction $\pi\colon Y\rightarrow\mathbb{P}^n$ with
centre $Z\subset \mathbb{P}^n$.  In Proposition~\ref{divextter}, we
show that $Z\subset D_{n+1}$ and $\codim_{\mathbb{P}^n}(Z)=2$.  The
assumption that $\Pic (\mathbb{P}^n)\to \Cl D_{n+1}$ is an isomorphism implies that
$Z$ is a complete intersection $Z= D_{n+1}\cap D_{d}$, where
$D_{d}\subset \mathbb{P}^n$ is a hypersurface of degree $d$.  Then we
show in Lemma~\ref{l1} that the cone of effective divisors of $Y$ is
$$
\Eff(Y)\ = \ \langle E, \widetilde{D}_b\rangle_+ \ ,
$$
where $E$ denotes the exceptional divisor of $\pi$, $b=\min\{d,n+1\}$, and $\widetilde{D}_b$ denotes the strict transform of ${D}_b$ in $Y$.
Therefore, the first link in a Sarkisov factorisation of $\psi$ is:
 \[
  \begin{tikzpicture}[xscale=1.5,yscale=-1.2]
    \node (A0_1) at (1, 0) {$Y$};
    \node (A0_2) at (2, 0) {$Y'$};
    \node (A1_0) at (0, 1) {$\mathbb{P}^n$};
    \node (A1_3) at (3, 1) {$W$ \ ,};
    \path (A0_1) edge [->,swap]node [auto] {$\scriptstyle{\pi}$} (A1_0);
    \path (A0_2) edge [->]node [auto] {$\scriptstyle{\pi'}$} (A1_3);
    \path (A0_1) edge [->,dashed]node [auto] {$\scriptstyle{\chi}$} (A0_2);
  \end{tikzpicture}
  \]
  where $\chi$ is a composition of Mori flips, flops and antiflips,
  and $\pi'\colon Y'\rightarrow W$ is either a divisorial contraction
  or a Mori fibre space that contracts the strict transform of
  $\widetilde{D}_b$.  In any case, $-K_{Y'}$ is $\pi'$-ample. Since
  $-K_{Y'}\sim \chi_*\widetilde{D}_{n+1}$, we must have $b=d<n+1$, and
  $\pi'\colon Y'\to W$ is a divisorial contraction with exceptional
  divisor $\widetilde{D}_b$.  A straightforward computation in the
  proof of Proposition~\ref{codim2} shows that in this case $W$ has
  worse than terminal singularities, which is not allowed in a Sarkisov
  factorisation of $\psi$.  This contradiction shows that the original
  map $\psi\colon \mathbb{P}^n\dasharrow X$ is an isomorphism.

\begin{Proposition}\label{divextter}
Let $(X,D_X)$ be a (t,~lc) CY pair, and $f\colon (Y,D_Y)\rightarrow (X,D_X)$ a volume preserving divisorial contraction with centre $Z\subset X$. Then $Z\subset D_X$.

Suppose moreover that $(X,D_X)$ is canonical, and that $D_X$ is terminal at the generic point of $Z$. Then $\codim_X(Z)=2$, and $D_Y$ is the strict transform of $D_X$ in $Y$.
\end{Proposition}

\begin{proof}
  Denote by $\widetilde{D}_X$ the strict transform of $D_X$ in $Y$,
  and by $E$ the exceptional divisor of $f$.  Since
  $f\colon (Y,D_Y)\rightarrow (X,D_X)$ is volume preserving, we have
  $D_Y=\widetilde{D}_X + mE$, with $m\in\{0,1\}$.  Suppose that
  $Z\not\subset D_X$. Then $\widetilde{D}_X = f^{*}D_X$, and the
  equality $K_Y+D_Y = f^*(K_X+D_X)$ would imply that
  $K_Y = f^{*}K_X-mE$, contradicting the fact that $X$ is terminal.
  So we conclude that $Z\subset D_X$.

  Now suppose that $(X,D_X)$ is canonical. By
  Lemma~\ref{D_normal}, $(Y,D_Y)$ is also canonical. It follows from
  Proposition~\ref{group2} that $m=0$, and $D_Y=\widetilde{D}_X$. In particular
  ${D}_Y$ has normal support. Denote by
  $\overline{f}\colon {D}_Y\rightarrow D_X$ the restriction of
  $f\colon Y\to X$ to ${D}_Y$. 
  Since $X$ and $Y$ are smooth in codimension two, we can restrict to ${D}_Y$ the equality
$$
K_Y+D_Y \ = \ f^*(K_X+D_X)
$$
and, applying adjunction, conclude that 
$$
K_{{D}_Y} \ = \overline{f}^{*}K_{D_X}.
$$
If moreover $D_X$ is terminal at the generic point of $Z$, then this last equality implies that $E\cap {D}_Y$ is not exceptional for $\overline{f}\colon {D}_Y\rightarrow D_X$ over the generic point of $Z$. 
It follows that $\codim_{D_X}(Z)=1$ and thus $\codim_X(Z)=2$.
\end{proof}

\begin{Lemma}\label{l1}
  Let $X$ be a $\Q$-factorial terminal Fano variety with $\rho =1$, and denote by $H_X$ a positive generator of $\Cl (X)/\text{torsion}$. Let $H_a, H_b\subset X$ be two hypersurfaces with $H_a \equiv aH_X$ and $H_b\equiv bH_X$ for some integers $a\geq b>0$.
  Suppose that the complete intersection $Z= H_a\cap H_b$ is irreducible and generically reduced, and let $\pi\colon Y\rightarrow X$ be a terminal divisorial contraction with centre $Z$. Then 
$$
\Eff(Y)\ = \ \langle E, \widetilde{H}_b\rangle_+ 
$$
where $E$ denotes the exceptional divisor of $\pi$, and
$\widetilde{H}_b$ the strict transform of ${H}_b$ in $Y$.
\end{Lemma}

\begin{proof}
  Set $H :=\pi^*H_X$, $\ell:= H^{n-1}\in N_1(Y)$, and
  $\lambda = H^{n}>0$, where $n=\dim(X)$.  
At the generic point of $Z$, the
  divisorial contraction $\pi$ coincides with the blow-up of $Z$.
In particular, $\pi_{|E}\colon E\rightarrow Z$ is a $\PP^1$-bundle over the generic point of $Z$.
Denote by $e\subset E$ a
  general fiber of $\pi_{|E}$.  We have $H \cdot \ell = \lambda$, $E\cdot e = -1$, and
  $H\cdot e = E \cdot \ell = 0$.  Since $\rho(Y)=2$, the cone
  $\overline{\Eff}(Y)$ has two extremal rays, one of which is generated by $E$.

We show that $\widetilde{H}_b$ is dominated by a family of curves with class in the ray 
$$
\bR_{\geq 0}\ [\ell-(a\lambda)e]\ \subset \ N_1(Y).
$$
Let $k$ be a positive integer such that $kH_X$ is very ample, and let
$C\subset H_b\subset X$ be a curve cut out in $H_b$ by $n-2$ general
members of $\big|kH_X\big|$.  Then $C\cdot H_a = abk^{n-2}\lambda$
and, inside $H_b$, $C$ intersects $Z$ transversally in
$abk^{n-2}\lambda$ smooth points. Therefore, the class of its strict
transform $\widetilde{C}\subset Y$ in $N_1(Y)$ is precisely
$bk^{n-2}[\ell-(a\lambda)e]$.

Let $D\subset Y$ be a prime divisor distinct from $\widetilde{H}_b$
and $E$, and write $D\equiv dH - mE$. Since $D\neq E$, it has
non-negative intersection with $e$, and thus $m\geq 0$.  Since
$D\neq \widetilde{H}_b$, it has non-negative intersection with
$\ell-(a\lambda)e$, and thus $d\geq ma\geq mb$. So we can write
$$
D\ \equiv \ (d-mb)H \ + \ m (bH-E)\ \equiv \ (d-mb)H \ + \ m \widetilde{H}_b
$$
with $d-mb\geq 0$. This  shows that $\Eff(Y)= \langle E, \widetilde{H}_b\rangle_+$.
\end{proof}

\begin{Proposition}\label{codim2}
Let $(X,D)/\Spec \CC$ be a $\Q$-factorial  (t,~lc) Mf CY pair, and let 
\[
\psi \colon (X,D)/\Spec \CC\dasharrow (X',D')/S'
\]
be a volume preserving Sarkisov link of Mf CY pairs. 
Then $\psi$ starts with a divisorial contraction with centre $Z\subset D$.
\begin{enumerate}[(1)]
    \item If $\codim_{X} Z=2$, then $Z$ is not a complete intersection of $D$ with a hypersurface; in other words, there is no hypersurface $H\subset X$ such that $Z=D\cdot H$.
    \item If the restriction homomorphism $\Cl(X)\to \Cl(D)$ is an isomorphism\footnote{The restriction homomorphism is well defined because, since $X$ has terminal singularities and hence it is nonsingular in codimension $2$, every Weil divisor on $X$ is Cartier in codimension $2$.}, then for every reduced irreducible divisors $W\subset D$ there exists a divisor $H\subset X$ such that $W=D\cdot H$. In particular, it follows from~(1) that $\codim_{X} Z\geq 3$.
\end{enumerate}
\end{Proposition}

\begin{proof}
  The link begins with a divisorial contraction $\pi\colon Y\rightarrow X$ and by
  Proposition~\ref{divextter} the centre $Z$ of $\pi$ is contained in $D$.  

  To prove (1), we draw a contradiction assuming $Z=D\cdot H$. As in the statement of Lemma~\ref{l1}, let $H_X\subset X$ be a positive generator of $\Cl X/\text{torsion}$. There are integers $\iota, d>0$ such that $D\equiv \iota H_X$, $H\equiv d H_X$. The link $\psi$ is a diagram:
 \[
  \begin{tikzpicture}[xscale=1.5,yscale=-1.2]
    \node (A0_1) at (1, 0) {$Y$};
    \node (A0_2) at (2, 0) {$Y'$};
    \node (A1_0) at (0, 1) {$X$};
    \node (A1_3) at (3, 1) {$W \ ,$};
    \path (A0_1) edge [->,swap]node [auto] {$\scriptstyle{\pi}$} (A1_0);
    \path (A0_2) edge [->]node [auto] {$\scriptstyle{\pi'}$} (A1_3);
    \path (A0_1) edge [->,dashed]node [auto] {$\scriptstyle{\chi}$} (A0_2);
  \end{tikzpicture}
  \]
where $\chi$ is a (possibly trivial) composition of Mori flips, flops and antiflips,
$(Y', D_Y')$ is a CY pair, where $D_Y'=\chi_*D_Y$ and $\pi'\colon Y'\to W$ is a divisorial contraction or a Mori fiber space.

We first show that $\iota > d$. For convenience, we write $D=H_\iota$, $H=H_d$.
By Lemma~\ref{l1}, $\Eff(Y') =  \langle E', \widetilde{H}'_b\rangle_+$, where $b=\min \{\iota, d\}$, $E'$ and $\widetilde{H}'_b$ denote the strict transforms in $Y'$ of $E$ and $\widetilde{H}_b$, respectively. 
This implies that either the morphism $\pi'\colon Y'\to W$ is a divisorial contraction with exceptional divisor $\widetilde{H}'_b$, or it is a Mori fibration given by the linear system $\big|m\widetilde{H}'_b\big|$ for $m\gg 0$. The latter occurs if and only if $\iota=d$. 
Since $(Y', D_Y')$ is a CY pair,  and $-K_{Y'}\sim D_Y'$ is $\pi'$-ample, we conclude that $D'_Y\neq \widetilde{H}'_b$. So $\iota>d$, $D=H_\iota$, and $\pi'\colon Y'\to W$ is a divisorial contraction with exceptional divisor $\widetilde{H}'_d$. 
Note also that $W$ is a $\Q$-factorial terminal Fano variety with $\rho(W)=1$. 

Since $\pi$ coincides with the blow-up of $Z$ at the generic point of $Z$, $W$ is terminal, and $\chi$ is a small birational map, we can write 
$$
\pi^{*}K_X+E \ =\  K_Y \ = \  \chi^*K_{Y'} \ = \ \chi^*(\pi'^{*}K_W + t \widetilde{H}'_d) \ = \ \chi^*\pi'^{*}K_W \ + \ t \widetilde{H}_d
$$
for some positive rational number $t$. 
Recall from the proof of Lemma~\ref{l1} that $\widetilde{H}_d$ is dominated by a family of curves with class in the ray 
$$
\bR_{\geq 0}\ [\ell-(\iota\lambda)e]\ \subset \ N_1(Y) \ ,
$$
where $\lambda = H_X^{\dim(X)}>0$. By intersecting the divisors above with $[\ell-(\iota\lambda)e]$, we get
\begin{align*}
0 \ = \ -\lambda\iota + \lambda \iota \ & = \ (\pi^{*}K_X\ +\ E)\ \cdot \ [\ell-(\iota\lambda)e] \\
& = \ (\chi^*\pi'^{*}K_W \ +\  t \widetilde{H}_d)\ \cdot \ [\ell-(\iota\lambda)e] \\
& = \ K_W\cdot (\pi'\circ \chi)_*[\ell-(\iota\lambda)e] \ + \ t\lambda (d-\iota)<0.
\end{align*}
This contradiction concludes the proof of~(1).

\smallskip

Let us now prove~(2). We must rule out the possibility that $\codim_{X}(Z)=2$.  Suppose that this is the case. Since the restriction homomorphism $\Cl(X)\to \Cl(D)$ is an isomorphism, there exists a positive integer $d>0$ such that $Z\sim d H_{X|D}$ as integral Weil divisors on $D=H_\iota$. By~\cite[Prop.~5.26]{KM98} we have an exact sequence of coherent sheaves: 
  \[
  0 \to \OO_X\bigl((d-\iota)H_X\bigr) \to \OO_X\bigl(d H_X\bigr) \to \OO_D\bigl(dH_{X|D}\bigr) \to 0 ,
  \]
where $\OO_D\bigl(dH_{X|D})$ denotes the divisorial (reflexive, rank~$1$) sheaf on $D$ corresponding to the Weil divisor $dH_{X|D}\sim Z$. Now note that $(d-\iota)H_X=K_X+dH_X$ and hence, by Kodaira and Kawamata--Viehweg vanishing, $H^1\bigl(X,(d-\iota)H_X\bigr)=H^1(X, K_X+dH_X)=(0)$. It follows that the restriction homomorphism
  \[
  H^0\bigl(X,dH_X\bigr) \to H^0\bigl(D, dH_{X|D}\bigr)
  \]
is surjective and there exists an irreducible hypersurface $H_d\sim dH_X$ such that
$Z= D\cap H_d$. This contradicts~(1) and proves~(2).
\end{proof}

\begin{proof}[Proof of Theorem~A]
Propositions~\ref{divextter} and \ref{codim2} together show that any volume preserving
birational map $(X,D)/\Spec \CC\dasharrow (X',D')/S'$ between Mf CY pairs is in fact an isomorphism.
This shows that $(X,D)\to \Spec \CC$ is birationally rigid, and it also shows that $\Bir (X,D) = \Aut (X, D)$.
\end{proof}

\begin{rem}
  It follows from Theorem~A that a \emph{very general} quartic surface
  $D\subset \mathbb{P}^3$ satisfies
  $\Bir (\mathbb{P}^3, D) = \Aut (\mathbb{P}^3,D)$.  By very general
  we mean that $D$ is smooth and satisfies the Lefschetz condition:
  $\Cl D =\ZZ \cdot \OO_{\mathbb{P}^3}(1)_{|D}$.  In this special
  case, a stronger result holds, and $D$ can be taken to be only
  \emph{general}.  Indeed, to conclude that
  $\Bir (\mathbb{P}^3, D) = \Aut (\mathbb{P}^3,D)$, it is enough to
  assume that $D$ is smooth and any curve of degree $<16$ in $D$ is a
  complete intersection of $D$ with another surface in
  $\mathbb{P}^3$ (see~\cite[Theorem 2.3 and Remark
  2.4]{Takahashi} and~\cite[Theorem 3.1]{Og13}).
\end{rem}

\section{Extremal contractions}
\label{sec:extr-contr}

In this section we classify extremal contractions of different types
in various contexts. These results are the toolkit that we need to prove
Theorems~B and~C. Indeed, the proofs of these results centre on the classification of
all possible volume preserving Sarkisov links
\[
(X,D)/S \dasharrow (X^\dagger, D^\dagger)/S^\dagger
  \]
  that start from a known and given Mf CY pair $(X,D)/S$. The links
  are themselves constituted of elementary modifications that are
  naturally associated to contractions of extremal rays: volume
  preserving extremal divisorial contractions; flips, flops,
  antiflips; and Mori fiber spaces. These elementary constituents are the
  topic of this section. The results are technical and we advise the
  reader to skip this section on first reading and use it as
  reference, coming back to it as needed.

\subsection{Divisorial contractions.}
\label{sec:divis-extr-1}

\begin{Lemma} \label{buA1} Let $D\subset\mathbb{P}^3$ be a quartic
  surface with canonical singularities, having exactly one singular
  point $z\in D$. Let $f\colon (Y,D_Y)\rightarrow (\mathbb{P}^3,D)$ be
  a volume preserving terminal divisorial contraction mapping a
  divisor $E\subset Y$ to a closed point $x\in \mathbb{P}^3$.  Then
  $x=z\in D$, and $D_Y$ is the strict transform of $D$ in $Y$.
\begin{enumerate}[(1)]
\item If $z$ is a singularity of $D$ of type $A_1$, then $f$ is the blow-up of
  $\mathbb{P}^3$ at $z$.
\item If $z$ is a singularity of $D$ of type $A_2$, choose homogeneous coordinates
  such that $D$ is defined by the equation
$$
 x_0x_1x_3^{2}+x_3B+C\ = \ 0,
$$
where $B=B_{3}(x_0,x_1,x_2)$ and $C=C_{4}(x_0,x_1,x_2)$ are homogeneous
polynomials of degree $3$ and $4$, respectively. Then either $f$ is the
blow-up of $\mathbb{P}^3$ at $z$, or $f$ is the weighted blow-up of $\mathbb{P}^3$ at $z$ with weights $(2,1,1)$ or $(1,2,1)$ --- with respect to the affine coordinates $x_0,x_1,x_2$ in the open affine chart $(x_3=1)$. 
\end{enumerate} 
\end{Lemma} 

\begin{proof} By Proposition~\ref{divextter}, $x=z$ is the singular
  point of $D$. By 
Proposition~\ref{group1}, $D_Y$ is the strict transform of $D$ in $Y$. Choose homogeneous coordinates on $\mathbb{P}^3$ such that $z=[0:0:0:1]$, and the
  equation of $D$ has the form
\[
x_3^2Q+x_3B+C = 0 \ ,
\]  
where $Q,B,C\in \CC[x_0,x_1,x_2]$ are homogeneous of degree $2,3$ and $4$, respectively. Then
$z\in D$ is an $A_1$-singularity if and only if $Q$ is a quadratic
form of rank $3$.  If $z\in D$ is an $A_2$-singularity, then $Q$ is
a quadratic form of rank two, which --- possibly after changing
homogeneous coordinates --- we may assume to be
$Q=x_0x_1$. 

By~\cite[Theorem 1.1]{Ka01}, in suitable \emph{analytic} coordinates
at $z\in \PP^3$, the divisorial contraction $f\colon Y\to \P^3$ is the
weighted blow-up of $z$ with weights $(1,a,b)$ where $a$ and $b$ are
coprime integers. The difficulty in using this result is that at this
point we do not know that these analytic coordinates at $z$ are induced
from homogeneous coordinates on $\PP^3$. Instead of using the result
directly, we will use an equivalent statement that is coordinate-free. 

In general, suppose that $z\in Z$ is a nonsingular point on a \mbox{$3$-fold} $Z$,
and that there are analytic coordinates at $z\in Z$ such that
$f\colon E\subset Y \to z\in Z$ is the weighted blow-up with weights
$(1,a,b)$, where $1\leq a<b$ and $\gcd (a,b)=1$. 
The toric description of the weighted blow-up allows us to realize the valuation associated to $E$ as follows. 
Construct inductively the tower of blow-ups:
\[
\cdots \to Z_{i}\to Z_{i-1}\to \dots \to Z_1\to Z_0=Z,
\]
where $Z_{i}\to Z_{i-1}$ is the blow-up of the centre
$\centre_{i-1}=\centre_E Z_{i-1}$ of the valuation $E$ on $Z_{i-1}$. Note that
$Z_1\to Z_0$ is the blow-up of $\centre_0=z$. For every $i$, we denote by
$E_i\subset Z_i$ the exceptional divisor, and for $j>i$ we denote by
$E^j_i\subset Z_j$ the strict transform of $E_i$ in $Z_j$. The
following key properties follow directly from the toric description of the weighted blow-up:
\begin{enumerate}
\item For all $0\leq j< a$, the centre $\centre_{j}$ is a closed point of $Z_{j}$. If $j\geq 1$, then 
$\centre_{j}\in E_{j}\subset Z_{j}$, and if $j\geq 2$, then  
\[
\centre_{j}\in E_{j}\setminus E^{j}_{j-1}.
\]
\item The centre $\centre_a\in Z_a$ is the generic point of a line
  $L_a\subset E_a \cong \PP^2$. If $a\geq 2$, then
\[
L_a\not \subset E^a_{a-1}.
\]
\item For all $a+1\leq j < b$, the centre $\centre_{j}\in Z_{j}$ is a
  section 
\[
L_{j}\subset E_{j}\setminus E^{j}_{j-1}
\] 
of the projection $E_{j}\to L_{j-1}$.
\item $E_b=E$ (by this we mean that the exceptional divisors $E_b$ and $E$ induce the same valuation on $Z$).
\end{enumerate}

\medskip

We go back to our volume preserving terminal divisorial
contraction $f\colon (Y, D_Y)\to (\PP^3, D)$ mapping $E\subset Y$ to
$z\in \PP^3$. Consider the tower above, starting with the blow-up
$\sigma_1\colon Z_1\to Z_0=\PP^3$ of $\centre_0=z\in \PP^3$, and denote by $D_i\subset Z_i$
the strict transform of $D$. We have that $K_{Z_1}=\sigma_1^* (K_{Z_0})+2E_1$
and $D_1\sim \sigma_1^* (D_0)-2E_1$, hence
\[
K_{Z_1}+D_1 = \sigma_1^* (K_{Z_0}+D_0).
\]
In other words, the birational morphism $(Z_1,D_1)\to (Z_0,D_0)$ is
volume preserving, and 
$$a(E, K_{Z_1}+D_1)=0.$$
If $z$ is a singularity of $D$ of type $A_1$ or $A_2$ then 
$D_1\subset Z_1$ is a smooth surface. 
Since $a(E, K_{Z_1}+D_1)=0$, either $E=E_1$ --- in which case
we are done --- or, by Proposition~\ref{divextter}, the centre $\centre_E Z_1$ is the generic point of a
curve on $D_1\cap E_1$.
In any case, key property (1) above implies that $a=1$. 

Suppose that $z\in D$ is an $A_1$-singularity. Then 
$D_1\cap E_1$ is a nonsingular conic in $E_1\cong \PP^2$.
On the other hand, if $b>1$, then the centre $\centre_E Z_1$ is the generic point of 
a line on $E_1\cong \PP^2$ by key property (2) above. 
This implies that $1=a=b$, i.e. $f$ is the blow-up of $\mathbb{P}^3$ at $z$.

Suppose now that $z\in D$ is an $A_2$-singularity. If $E=E_1$, then we are
done. Otherwise, by key property (2) above, the centre
$\centre_1=\centre_E Z_1$ is the generic point of one of the two lines
$(x_0=0)$, $(x_1=0)$ in $E_1\cong \PP^2$. Assume $\centre_1$ is the
generic point of the line $L_1=(x_0=0)$ --- the other case is
similar. Write $\sigma_2\colon Z_2\to Z_1$ for the blow-up of the line
$L_1\subset Z_1$. Note that $K_{Z_2}=\sigma_2^*(K_{Z_1})+E_2$, and
$D_2=\sigma_2^* (D_1)-E_2$. Hence
\[
K_{Z_2}+D_2=\sigma_2^* (K_{Z_1}+D_1),
\]
in other words, the composed birational morphism $(Z_2,D_2)\to
(\PP^3, D)$ is volume preserving and $a(E,K_{Z_2}+D_2)=0$. 
If $E=E_2$, then $f$ is the weighted blow-up with
weights $(2,1,1)$ in the native homogeneous coordinates of $\PP^3$,
and we are done. We show that the assumption that $E\neq E_2$ leads to
a contradiction. Consider the centre $\centre_2=\centre_E Z_2$.
\begin{enumerate}[(i)]
\item By key property (3), $\centre_2$ is the generic
  point of a section $L_2\subset E_2$ of the projection $E_2\to L_1$ disjoint from $E_1^2$.
\item On the other hand, since $D_2\subset Z_2$ is a smooth surface,
  $\centre_2$ is the generic point of a curve in $D_2\cap E_2$ by Proposition~\ref{divextter}.
Since $\Gamma = D_2 \cap E_2$ is irreducible, 
$\centre_2$ is the generic point of $\Gamma$. 
\end{enumerate}
Finally, since $D_1\cap E_1$ consists of the union of the two lines $(x_0=0)$ or $(x_1=0)$ in $E_1\cong \PP^2$,
the curves $\Gamma = D_2 \cap E_2$ and $E_1^2\cap E_2$ intersect.
This contradicts (i) and concludes the proof. 
\end{proof}

\subsection{Extremal neighbourhoods}\ 
\label{sec:extr-neighb}

\begin{dfn}
  \label{dfn:extremal_nbd} \label{dfn:trivial_nbds}
  An \emph{extremal neighbourhood}  is the analytic germ of a \mbox{3-fold} $X$ around a projective curve
$\Gamma \subset X$. 
 
We say that the extremal neighbourhood is \emph{nonsingular} if $X$ is nonsingular.

We say that a nonsingular extremal neighbourhood  is \emph{trivial} if it is isomorphic to the analytic germ around
$\Gamma$ in its normal bundle $N_{\Gamma/X}$.


 In our situation, the curve $\Gamma$ will always be contained in a given analytic germ of a surface $S\subset X$.
\end{dfn}

\begin{Conv}\label{conv_germs}
 An extremal neighbourhood is an analytic germ, that is, an equivalence class of analytic neighbourhoods $(\Gamma \subset X)$, called \emph{representatives}. Two representatives 
 \[
(\Gamma \subset X_1) \quad \text{and} \quad (\Gamma \subset X_2)
 \]
 are equivalent if there is a third representative contained in both:
 \[
 \xymatrix{
  & (\Gamma \subset X_1) \\
  (\Gamma \subset X_3)\ar@{^(->}[ru] \ar@{_(->}[rd]& \\
  & (\Gamma \subset X_2).
 }
 \]
 When we speak of an extremal neighbourhood $(\Gamma \subset X)$ we are often (and without saying so explicitly) abusing language and actually taking $(\Gamma \subset X)$ to be a representative, which we may allow ourselves subsequently to shrink at our convenience. 
 
 We often (and sometimes implicitly) assume that the extremal neighbourhood is \emph{contractible}, that is, that there is a projective morphism between suitable representatives
 \[
 \pi \colon (\Gamma \subset X) \to (z \in Z),
 \]
 where
 \begin{enumerate}[(i)]
 \item  $\pi$ is projective and $Z$ is Stein (and hence, for any coherent sheaf $\mathcal{F}$ of $\cO_Z$-modules,  $H^q(Z,\mathcal{F})=(0)$ $\forall q>0$); 
 \item $\Gamma = \pi^{-1}(z)$ and $\pi$ maps $X\setminus \Gamma$ isomorphically to $Z\setminus \{z\}$.
 \end{enumerate}

 All of this is common practice when writing about germs, and we trust that it will be clear from the context what is meant. 
\end{Conv}


\begin{lem}
  \label{lem:trivial_neighbourhoods}  \label{lem:inverse flips}
Consider a nonsingular extremal neighbourhood $\Gamma \subset S\subset X$, where $S\in |-K_X|$ is the germ of a nonsingular surface and $\Gamma\cong \PP^1$.
Assume that $k=K_X\cdot \Gamma = -S\cdot \Gamma \geq 1$. Then the following hold.
\begin{enumerate}
\item The extremal neighbourhood is trivial.
\item The antiflip $X\dasharrow X^-$ exists, and $X^-$
has terminal singularities if and only if $k=1$.
\end{enumerate}
\end{lem}

\begin{proof} 
 To prove (1), we use~\cite[Lemma~3.33]{MR662120}, which in fact goes back to \cite[Lemma~9]{MR171784}. 
 It is enough to show that the analytic germ around $S\subset X$ is  trivial, that is, $X$ is analytically isomorphic to the analytic germ of $S$ in its normal bundle $N_{S/X}$. 
 Indeed, $S$ is the germ of a K3 surface, $\Gamma$ is a $(-2)$-curve on $S$, and $S$ itself is analytically isomorphic to the analytic germ of $\Gamma$
 in its normal bundle $N_{\Gamma/S}$ on $S$.

  First we check that the second infinitesimal neighbourhood of $S$ in $X$ is trivial. 
  We denote by $\OO_S(1)$ the unique line bundle on $S$ such that $\OO_S(1)_{|\Gamma}=\OO_\Gamma(1)$. 
  The second infinitesimal neighbourhood is an
  infinitesimal extension of $\OO_S$ by $\OO_S(-S)=\OO_S(k)$:
    \[
0 \to \OO_S(-S) \to \OO_{2S} \to \OO_S \to 0 \ ,
    \]
    and these extensions are classified by
    \[
      H^1(S, T_S(-S))=H^1\Bigl(\OO_S(-2+k)\oplus \OO_S(2+k)\Bigr)=0.
    \]
Now that we have shown that the second infinitesimal neighbourhood of $S$ in $X$ is
trivial, we may apply ~\cite[Lemma~3.33]{MR662120}, which
states that the neighbourhood of $S\subset X$ is formally trivial if 
$H^1\bigl( S,T_{X|S}\otimes \OO_S(-nS)\bigr)=0$ for all $n\geq 2$.
The latter statement in turn follows from the exact sequence
\[
0\to T_S\to T_{X|S} \to \OO_S(S) \to 0 \ ,
\]
giving
\[
H^1\Bigl(S, \OO_S(-2+nk)\oplus \OO_S(2+nk)\Bigr) \to H^1\Bigl(
S,T_{X|S}\otimes \OO_S(-nS)\Bigr)\to H^1\Bigl(S, \OO_S\bigl((n-1)k\bigr)\Bigr)
\]
for all $n$. We have shown that the neighbourhood of $S$ in $X$ is formally
trivial. The neighbourhood is analytically trivial by the main result
of~\cite{MR171784}.

Now let us prove (2). 
It follows from (1) that the extremal neighbourhood $\Gamma \subset S\subset X$ is isomorphic to the
neighbourhood of $\PP^1$ in the total space of the bundle $\OO(-2)\oplus \OO(-k)$. 
Equivalently, $X$ is isomorphic to the neighbourhood of
  $\PP^1=(y_0=y_1=0)$ in the geometric quotient $\CC^4/\!\!/\CC^\times$ for the action  given by the weights:
\[
  \begin{array}{cccc}
x_0 & x_1 & y_0 & y_1 \\
\hline
1 & 1 & -2 & -k     
  \end{array} \ ,
\]
with $(>0)$ stability condition.
Under this identification, $S$ is given by the equation $(y_1=0)$.

The antiflip $X\dasharrow X^-$ is obtained by changing the stability condition to $(<0)$. 
If $k=1$, then $X^-$ has terminal singularities. 
If $k>1$, then $X^-$ has a strictly canonical (and Gorenstein) quotient singularity of type
\[
\frac1{k}(1,1,-2).
\]
\end{proof}

\begin{rem}
  \label{rem:trivial_neighbourhoods}
  In Lemma~\ref{lem:trivial_neighbourhoods} above, the existence of the surface $S$ is needed for the triviality of the neighbourhood,
  as shown by simple counterexamples.
 \end{rem}

 In the proof of Theorem~C, we need to understand neighbourhoods
where the curve $\Gamma$ is not irreducible.

\begin{lem}
\label{lem:extremal_picard}
Let $\pi \colon (\Gamma \subset S \subset X)\to (z\in Z)$ be a contractible extremal neighbourhood where $X$ has terminal singularities, $S\in |-K_X|$ has Du Val singularities, and hence $\Gamma =\cup \Gamma_{i=1}^r \subset S$ is a tree of $\PP^1$s.

Then $H^1(X, \OO_X)=H^2(X, \OO_X)=0$. In particular, $c_1\colon \Pic X \to H^2(X,\ZZ)$ is an isomorphism. If in addition $X$ is smooth, then $\Pic X = \ZZ^r$ where a basis is given by transverse $2$-disks $D_i\subset X$ such that for all $i$, $D_i\cdot \Gamma_i=1$ and $D_i$ is disjoint from $\Gamma_j$ ($j\neq i$).
\end{lem}

\begin{proof}
From the exact sequence
\[
0 \to \OO_X(K_X)\to \OO_X \to \OO_S \to 0 , 
\]
we get an exact sequence for each $q>0$:
\[
\cdots \to R^q\pi_\star\OO_X(K_X)\to R^q \pi_\star \OO_X \to R^q\pi_\star \OO_S \to \cdots \,.
\]
For each $q>0$,  $R^q\pi_\star \OO_X(K_X) =(0)$ by the Kawamata--Viehweg vanishing theorem, and $R^q\pi_\star \OO_S=(0)$.
So we get $R^q\pi_\star \OO_X=(0)$ for all $q>0$, and hence $H^q(X, \OO_X)=H^0(Z, R^q\pi_\star \OO_X)=(0)$.\footnote{Following our Convention~\ref{conv_germs}, here we consider representatives $\pi \colon (\Gamma \subset X)\to (z\in Z)$ where $\pi$ is projective and $Z$ is Stein.}
\end{proof}

\begin{lem}
  \label{lem:A2-nbds}
Consider an extremal neighbourhood $\Gamma \subset S\subset X$, where $X$ is a smooth  3-fold,
$S\in |-K_X|$ is a smooth surface, and 
  $\Gamma=\Gamma_0\cup \Gamma_1\subset S$ is a chain of two $(-2)$-curves
  intersecting transversally.  Suppose that $-K_X\cdot \Gamma_0 =S\cdot \Gamma_0=-a\leq 0$ and
  $-K_X\cdot \Gamma_1 =S\cdot \Gamma_1=-b<0$.
Then the extremal neighbourhood is isomorphic to the analytic germ
  around the curve
  \[
    \Gamma_0\cup\Gamma_1=(x_0=x_4=0)\cup(x_0=x_1=0)
  \]
in the geometric quotient $\CC^5 /\!\!/\CC^\times$ for the action   given by the weights:
\[
  \begin{array}{ccccc}
x_0 & x_1 & x_2 & x_3 & x_4 \\
\hline
    -a & 1 & 1 & 0 & -2 \\
    -b &-2& 0 & 1 & 1 \\
  \end{array},
\]
where the stability condition is taken in the quadrant
$\langle (1, 0), (0,1) \rangle_+$.
Under this identification, $S$ is given by the equation $(x_0=0)$.
\end{lem}

\begin{proof}
  By Lemma~\ref{lem:trivial_neighbourhoods}, the neighbourhood is
  trivial around each of the two curves $\Gamma_0$, $\Gamma_1$. It
  follows from this that we can find divisors $D_0, \dots, D_4$ on $X$ as follows:
  \begin{enumerate}[(i)]
  \item $D_0=S$;
  \item $D_0\cap D_1=\Gamma_1$
    scheme-theoretically. It follows from this that $D_1$ intersects
    $\Gamma_0$ transversally at the point where $\Gamma_0$ intersects
    $\Gamma_1$;
  \item $D_2$  intersects $\Gamma_0$ transversally at one
    point and is disjoint from $\Gamma_1$;
  \item  $D_3$  intersects $\Gamma_1$ transversally at
    one point and is disjoint from $\Gamma_0$;
  \item $D_0\cap D_4=\Gamma_0$ scheme-theoretically. It follows from this that
    $D_4$ intersects $\Gamma_1$ transversally at the point where $\Gamma_1$
    intersects $\Gamma_0$.
  \end{enumerate}
Note that the intersection multiplicities of these divisors with the
curves $\Gamma_0$ and $\Gamma_1$ are as follows:
\[
  \begin{array}{cccccc}
 &D_0 & D_1 & D_2 & D_3 & D_4 \\
\hline
 \cdot \Gamma_0  &  -a & 1 & 1 & 0 & -2 \\
 \cdot \Gamma_1  &  -b &-2& 0 & 1 & 1 \\
  \end{array}
\]
and so, by Lemma~\ref{lem:extremal_picard}, the main result of \cite{Cox95} gives that the divisors $D_0,\dots,D_4$ give a morphism from the extremal neighbourhood
to the model toric quotient neighbourhood given above, and it is easy to see that this morphism is an isomorphism.
\end{proof}
 
Next we describe extremal neighbourhoods $\Gamma \subset S\subset X$
in the case when $S$ is nonsingular outside a single ordinary node, where $X$ has a quotient singularity
of type $1/2(1,1,1)$, and $K_X\cdot \Gamma = -S\cdot \Gamma>0$.

\begin{lem}
  \label{lem:BadAntiflips}
  Let $\Gamma \subset S \subset X$ be a 3-dimensional extremal neighbourhood where 
  $\Gamma \cong \PP^1$ is a smooth rational curve, and $S \in |-K_X|$ is nonsingular outside a single ordinary double point $P\in S$, where $X$ has a quotient singularity of type $1/2(1,1,1)$. 

Assume that  $K_X\cdot \Gamma = -S\cdot \Gamma= k/2 >0$,  with $k\geq 3$ an odd integer. 

Then the antiflip $X\dasharrow  X^{-}$ exists, and $X^{-}$ has worse than terminal singularities. 
\end{lem}

\begin{proof}
  Consider the locally toric blow-up $f\colon (E\subset Y)\to (\Gamma \subset X)$ of the
  curve $\Gamma \subset X$ such that
  \begin{itemize}
      \item outside of $P\in X$, $f$ is the ordinary blow up of $\Gamma$;
      \item choosing coordinates $\left(0\in (x=y=0)\subset (x=0)\subset \frac1{2}(1,1,1) \right)$ near $\left(P\in \Gamma \subset S \subset X\right)$, $f$ is the toric blow-up with weights $(1,1,0)$ near $P\in X$.
  \end{itemize} 
  (We give an explicit model construction of this modification below.) Denote by $T\subset Y$ the strict
  transform of $S$, and set $\Delta = T \cap E$. We have that $T\in |-K_Y|$, and $f^{-1}(P)\subset Y$ is a curve of $A_1$-singularities. A computation shows that:
  \[
    E\cdot \Delta = -\frac{3}{2},
\quad \text{and} \quad
    T \cdot \Delta =-m+2, \; \text{where
      $m=\frac{k}{2}+\frac{1}{2}\geq 2$ is an integer.}
  \]
  (The computation can be done directly or using the model construction below.) 

\smallskip

\textsc{Claim.} \emph{Denote by ($\Delta \subset \widehat{Y}$) the extremal neighbourhood of $\Delta$ inside $Y$, that is, the germ of $Y$ around $\Delta$. If $m\geq 2$, the extremal neighbourhood
$(\Delta \subset  \widehat{Y})$ is analytically isomorphic to the analytic germ around the curve
  $(x_0=x_1=0)$ in the geometric quotient $\CC^4 /\!\!/\CC^\times$ for the action  given by the weights:
\[
  \begin{array}{cccc}
x_0 & x_1 & y_0 & y_1\\
\hline
    -2m+4 & -3 & 1 & 2\\
  \end{array} ,
\]
where the stability condition is $(>0)$.
Under this identification, $E=(x_1=0)$ and $T=(x_0=0)$.}

\smallskip

We first show that the claim implies the statement. The goal is to construct the antiflip $X\dasharrow X^{-}$ of the contraction $p\colon (\Gamma \subset X) \to (z\in Z)$. Write $\widetilde{p}=p\circ f\colon Y \to Z$. There are two cases to discuss: $m=2$ and $m>2$.

If $m>2$, then, by the claim, the curve $\Delta \subset Y$ is contractible, and hence there is a factorisation $\widetilde{p}=\pi \circ q$:
\[
\xymatrix{
 & (\Delta \subset Y) \ar[dl]_f\ar[dr]^q& \\
(\Gamma \subset X)\ar[dr]_p & & (w\in W)\ar[dl]^\pi \\
 & (z\in Z)& }
\]
where $\Delta = q^{-1}(w)$. Note that $\pi \colon W\to Z$ is projective and it contracts the image of $E$. We use the description of the neighbourhood $(\Delta \subset Y)$ given by the claim. To perform the antiflip 
\[
\xymatrix{
Y \ar[dr]_q\ar@{-->}[rr]& & Y^{-}\ar[dl]^{q^{-}} \\
  & W & 
}
\]
we change the stability condition to $(<0)$.
Denote by $T^{-}$ and $E^{-}\subset
Y^{-}$ the strict transforms of $T$ and $E$, respectively.
Then $T^{-}$ and $E^{-}$ are disjoint in $Y^{-}$, and the antiflip $X^{-}$ of the original $(\Gamma \subset X)$
is given by
\[
X^{-} =\Proj R(Y^{-}, T^{-}) = \underline{\Proj}_{\OO_Z} \Bigl(\bigoplus_{n\geq 0}(\pi \circ q^{-})_\star \OO_{Y^{-}}(nT^{-})\Bigr)\, .
\]
Indeed, first of all, $T^{-}\in |-K_{Y^{-}}|$ is nef and big (over $Z$) and hence, by the base point free theorem, $R(Y^{-}, T^{-})$ is finitely generated, $T^{-}$ is eventually free and defines a morphism $f^{-} \colon Y^{-}\to X^{-}$. 
The divisor $E^{-}$ is disjoint from $T^{-}$ and hence it is contracted to a point in $X^{-}$, and
 this point is a strictly canonical singularity. It is also clear that the birational map $X\dasharrow X^{-}$ is an isomorphism in codimension one, and that $S^{-}=f^{-}(T^{-})\sim -K_{X^{-}}$, hence $X^{-}\to Z$ really is the antiflip of $X\to Z$.  

If $m=2$, then $E\cong \PP^1\times \PP^1$, and $X^{-}$ is the contraction of $E^{-}\subset Y$ along the ruling other than the one contracted to $X$. The image of $E$ in $X^{-}$ is a curve of $A_2$-singularities, and so $X^{-}$ has worse than terminal singularities. 

It remains to prove the claim. Concretely, we construct the modification $f\colon (E\subset Y)\to (\Gamma \subset X)$ so that it fits into the following diagram: 
\[
\xymatrix{(\Delta_0\cup\Delta_1\subset F_1\cup E_1\cup T_1 \subset Y_1) \ar[r]^g\ar[d]_\sigma & (\Gamma_0\cup \Gamma_1 \subset F \cup S_1 \subset X_1)\ar[d]^\varepsilon  \\
(\Delta \subset T\cup E \subset Y)\ar[r]_f & (\Gamma \subset S \subset X).}\]
The maps in the diagram are defined as follows. 
\begin{itemize}
\item $\varepsilon \colon X_1\to X$ is the blow-up of the singular point $P\in X$. The exceptional divisor $F\subset X_1$ is isomorphic to $\PP^2$ and it has normal bundle $\OO(-2)$. We denote by $\Gamma_1, S_1$ the proper transforms of $\Gamma, S$ and set $\Gamma_0=F\cap S_1$. Note that 
 \[
    F\cdot \Gamma_1 = 1,
\quad \text{and} \quad
    S_1 \cdot \Gamma_1 =-m, \; \text{where
      $m=\frac{k}{2}+\frac{1}{2}\geq 2$ is an integer.}
  \]
\item $g\colon Y_1\to X_1$ is the blow up of $\Gamma_1$. The exceptional divisor $E_1\subset Y_1$ is isomorphic to the ruled surface $\FF_{m-2}$. We denote by $\Delta_0, \Delta_1, T_1, F_1$ the proper transforms of $\Gamma_0,\Gamma_1,S,F$. Note that $F_1$ is isomorphic to the ruled surface $\FF_1$. We have
\[
    T_1\cdot \Delta_1 = -(m-2) \,.
  \]
  \item $\sigma\colon Y_1\to Y$ is the contraction of the ruled surface $F_1$ to a curve of $A_1$-singularities in $Y$. It is now easy to compute, as claimed:
  \[
  E\cdot \Delta = \bigl(E_1+\frac1{2} F_1\bigr)\cdot \Delta_1 = -2+\frac1{2}=-\frac{3}{2}
  \]
  and
  \[
  T\cdot \Delta = T_1\cdot \Delta_1=-(m-2) \,.
  \]
\end{itemize}
By Lemma~\ref{lem:A2-nbds}, the extremal neighbourhood  $(\Delta_0\cup \Delta_1\subset Y_1)$ is isomorphic to the analytic germ around the curve
  \[
    \Delta_0\cup\Delta_1=(x_0^\prime = z=0)\cup(x_0^\prime=x_1^\prime=0)
  \]
in the geometric quotient $\mathbb{Y}_1=\CC^5 /\!\!/\CC^\times$ for the action given by the weights:
\[
  \begin{array}{ccccc}
x_0^\prime & x_1^\prime & y_0^\prime & y_1^\prime & z \\
\hline
    0& 1 & 1 & 0 & -2 \\
    -(m-2) &-2& 0 & 1 & 1 \\
  \end{array},
\]
where the stability condition is taken in the quadrant
$\langle (1, 0), (0,1) \rangle_+$.
Under this identification, $T_1=(x_0^\prime=0)$, $E_1=(x_1^\prime=0)$ and $F_1=(z=0)$.

To prove the claim, we need to show that the extremal neighbourhood  $(\Delta \subset Y)$ is analytically isomorphic to the analytic germ around the curve
  $(x_0=x_1=0)$ in the geometric quotient $\mathbb{Y}=\CC^4 /\!\!/\CC^\times$ for the action given by the weights:
\[
  \begin{array}{cccc}
x_0 & x_1 & y_0 & y_1\\
\hline
    -2(m-2) & -3 & 1 & 2\\
  \end{array} ,
\]
where the stability condition is $(>0)$.
Indeed, $\sigma\colon {Y}_1\to {Y}$ is the contraction of $F_1$ and it is induced by the morphism $\widetilde{\sigma}\colon \mathbb{Y}_1\to \mathbb{Y}$ given by 
\begin{equation}
\label{eq:proj}
(x_0,x_1,y_0,y_1) \mapsto (x_0^\prime,x_1^\prime \sqrt{z}, y_0^\prime \sqrt{z}, y_1^\prime) \,.
\end{equation}
For sake of clarity, we spell out in detail how Equation~\ref{eq:proj} gives a morphism $\widetilde{\sigma}\colon \mathbb{Y}_1\to \mathbb{Y}$. 

The ambient space $\mathbb{Y}$ is not smooth: it has a quotient singularity $\frac1{2}(0,1,1)$ at the origin of the chart $(y_1=1)$. The sheaf $\OO_{\mathbb{Y}}(1)$ is not a line bundle on $\mathbb{Y}$ (it is a $\QQ$-line bundle), which explains some of the strange square roots in Equation~\ref{eq:proj}. However, $A=\OO_{\mathbb{Y}}(2)$ is an ample line bundle on $\mathbb{Y}$ and
\[
\mathbb{Y} = \Proj R(\mathbb{Y}, A) \quad \text{where} \quad R(\mathbb{Y, A})=\oplus_{n\geq 0} H^0(\mathbb{Y},nA).
\]
Similarly, $B=\OO_{\mathbb{Y}_1} \begin{pmatrix} 0\\1\end{pmatrix}$ is a (nef) line bundle on $\mathbb{Y}_1$. Equation~\ref{eq:proj} gives a morphism $\widetilde{\sigma}\colon \mathbb{Y}_1\to \mathbb{Y}$ such that $\widetilde{\sigma}^\star (A) = B$ where $\widetilde{\sigma}^\star \colon H^0(\mathbb{Y},A) \to H^0(\mathbb{Y}_1,B)$ is computed from Equation~\ref{eq:proj}:
\begin{multline*} 
\big(y_0^2,\; y_1,\; x_1y_0^5,\; x_1y_0^3y_1, \; x_1y_0y_1^2, \; \dots,\; x_0 y_0^{2m-2},\;\dots \big) \quad \mapsto \\ \mapsto \quad 
 \big( (y_0^\prime)^2z, \; y_1^\prime, \; x_1^\prime (y_0^\prime)^5 z^3, \; x_1^\prime (y_0^\prime)^3y_1^\prime z^2, \; x_1^\prime y_0^\prime (y_1^\prime)^2z, \; \dots, \; x_0^\prime (y_0^\prime)^{2m-2}z^{m-1}, \; \dots \big) \;.
\end{multline*}
\end{proof}

\subsection{Fano and weak Fano $\PP^1$-bundles over $\PP^2$}
\label{sec:fano-weak-fano}

We state the classification of rank~two vector bundles $\cE$ on
$\PP^2$ such that $\PP(\cE)$ is Fano or weak Fano. We also collect some
elementary facts and formulas on projective space bundles for later
use. 

We begin by clarifying our conventions regarding vector bundles. 
In this section, a vector bundle on a scheme is a locally free sheaf on
it. Following Grothendieck, if $\cE$ is a vector bundle on $Y$ we denote by
\[
\PP(\cE)=\uProj_{\O_X} \bigoplus_{n\in \NN} \Sym^n \cE
\]
the space of $1$-dimensional \emph{quotients} of $\cE$, by
\[
  \pi \colon \PP(\cE)\to Y
 \]
the natural projection, by $\O_{\PP(\cE)} (1)$ --- or simply $\OO(1)$ --- the tautological line bundle on
$\PP(\cE)$, so that
\[
  \cE = \pi_* \O(1),
\]
and by $\xi =c_1 \left( \O(1)\right)\in
A^1\left(\PP(\cE)\right)$ its first Chern class. 
The Chow ring of $\PP(\cE)$ admits the following description:
\[
A^\bullet \left(\PP(\cE)\right) =
\frac{A^\bullet (Y)[\xi]}{(\xi^r+\sum_{i=1}^r(-1)^i\xi^{r-i}c_i(\cE))} \ , 
\]
where $r=\rk \cE$, and the formula defines the $i$-th \emph{Chern class} $c_i(\cE)\in A^i(Y)$. 
The \emph{Chern polynomial} of $\cE$ is the polynomial
\[
c_t (\cE)=1+\sum_{i=1}^r c_i(\cE) t^i.
\]
\medskip

In what follows, $\cE$ is a rank~two vector bundle on $\PP^2$. We denote by $\ell\in A^1(\PP^2)$ the
class of a line and we abuse notation slightly and write the first
and second Chern classes of $\cE$ as
\[
c_1 \ell, \quad c_2 \ell^2, \quad \text{where $c_1,c_2\in \ZZ$.}
\]
We say that $\cE$ is \emph{normalised} if $c_1\in \{0,-1\}$. We can
always achieve this by tensoring $\cE$ with a line bundle.

\begin{thm}
  \label{thm:weakFanoVectorBundles}
  Let $\cE$ be a normalised rank~two vector bundle on $\PP^2$. Then
  \begin{enumerate}[(1)]
  \item $\PP(\cE)$ is Fano if and only if $\cE$ is one of the bundles
    in List~$1$ below~\cite{MM81,SW90}.
  \item $\PP(\cE)$ is strictly weak Fano --- i.e., $-K_{\PP(E)}$ is nef
    and big but not ample --- if and only if $\cE$ is one of the bundles
    in List~$2$ below~\cite[Theorem B]{Ya12}, \cite[Theorem 3.4]{JPR05}.
  \end{enumerate}
\end{thm}

\noindent {\bf List~$1$. Rank~$2$ vector bundles $\cE$ on $\P^2$ with $c_1(\cE)\in \{0,-1\}$ such that $\mathbb{P}(\cE)$ is Fano}
\begin{enumerate}
\item $\mathcal{O}_{\P^2}\oplus \mathcal{O}_{\P^2}(-1), \ c_1=-1, c_2 =0$;
\item $\mathcal{O}_{\P^2}(1)\oplus \mathcal{O}_{\P^2}(-1), \ c_1=0, c_2=-1$;
\item $\mathcal{O}_{\P^2}\oplus \mathcal{O}_{\P^2}, \ c_1=c_2=0$;
\item $T_{\P^2}(-2), \ c_1 = -1, c_2 = 1$;
\item $\mathcal{E}$ is determined by the exact sequence $0\ \to \mathcal{O}_{\P^2}\ \to \mathcal{E}\ \to \mathcal{I}_p\ \to 0$, where $\mathcal{I}_p$ is the ideal sheaf of a point $p\in\mathbb{P}^2$, $c_1 = 0, c_2 = 1$ (see Remark~\ref{bunseq} below);
\item $\mathcal{E}$ is a stable bundle with $c_1 = 0, c_2=2$;
\item $\mathcal{E}$ is a stable bundle with $c_1 = 0, c_2=3$.
\end{enumerate}

\medskip

\noindent {\bf List~$2$. Rank~$2$ vector bundles $\cE$ on $\P^2$ with $c_1(\cE)\in \{0,-1\}$ such that $\mathbb{P}(\cE)$ is strictly weak Fano}

\begin{enumerate} \setcounter{enumi}7
\item $\mathcal{O}_{\P^2}(1)\oplus\mathcal{O}_{\P^2}(-2)$, $c_1 = -1, c_2 = -2$;
\item $\mathcal{E}$ is determined by the exact sequence $0\ \to \mathcal{O}_{\P^2}\ \to \mathcal{E}\ \to \mathcal{I}_p(-1)\ \to 0$, where $\mathcal{I}_p$ is the ideal sheaf of a point $p\in\mathbb{P}^2$, $c_1 = -1, c_2 = 1$ (see Remark~\ref{bunseq} below);
\item $\mathcal{E}$ is a stable bundle with $c_1 = -1, 2\leq c_2\leq 5$;
\item $\mathcal{E}$ is a stable bundle with $c_1 = 0, 4\leq c_2\leq 6$.
\end{enumerate}

\begin{Remark}\label{bunseq}
  One may compute the Chern classes of the vector bundles (5) and (9)
  above as follows. For $k\leq 2$, let $\mathcal{E}$ be the unique
  vector bundle on $\mathbb{P}^2$ that sits in the exact
  sequence\footnote{A simple computation shows that
\[
\text{for all $k\in \ZZ$,}\quad \Ext^1_{\O_{\PP^2}}
(\cI_p(k),\O_{\PP^2}) =
\begin{cases}
  \CC \;& \text{if $k \leq 2$}; \\
  0 \; & \text{if $k\geq 3$};
\end{cases} \ ,
\]
and that the nontrivial extensions are also locally
nontrivial.}
\[
0\ \to \mathcal{O}_{\P^2}\ \to \mathcal{E}\ \to \mathcal{I}_p(k)\ \to 0 \ .
\]
By multiplicativity of Chern polynomials, we have
$c_t(\mathcal{E}) = c_t(\mathcal{I}_p(k))$.  The Koszul complex
twisted by $\mathcal{O}_{\P^2}(k)$,
\[
0\rightarrow \mathcal{O}_{\P^2}(-2+k)\rightarrow
\mathcal{O}_{\P^2}(-1+k)^{\oplus 2}\rightarrow
\mathcal{I}_p(k)\rightarrow 0,
\]
yields that $c_t(\mathcal{I}_p(k))(1+(-2+k)t) =
(1+(-1+k)t)^2$, and from this we conclude that $c_1(\mathcal{E}) = k$ and
$c_2(\mathcal{E}) = 1$.
\end{Remark}

\begin{lem}
  \label{lem:canonical_in_PP(E)}
  Let $\cE$ be a rank $r+1$ vector bundle on $Y$. Then the
  anticanonical class of $\PP(\cE)$ is
  \[
-K_{\PP(\cE)} = (r+1)\xi -\pi^* \bigl(c_1(\cE)+K_Y\bigr) \ .
   \]
 \end{lem}

 \begin{proof}
   This follows from the Euler sequence computing the relative tangent
   bundle:
\[
0 \to \O_{\PP(\cE)} \to \bigl(\pi^* \cE^\vee\bigr)\otimes
\O_{\PP(\cE)}(1) \to T_{\pi} \to 0 \ ,
\]
and the exact sequence $0 \to T_\pi \to T_{\PP(\cE)} \to \pi^* T_Y \to 0$. 
\end{proof}

\section{Proof of Theorem~B}\label{section:A1}

In this section we prove Theorem~B. Consider a Mf
CY pair $(\PP^3,D)$, where $D\subset \PP^3$ is a quartic surface that
is nonsingular outside a unique singular point $z\in D$ of type
$A_1$, and such that the class group $\Cl(D)\cong \mathbb{Z}$ is
generated by the class of a hyperplane section.  In this case,
birational rigidity fails. The following Mf CY pair
$(X, D_X)\to \P^2$ is a nontrivial element in $\mathcal{P}(\P^3,D)$.
Let $\sigma\colon X\rightarrow\mathbb{P}^3$ be the blow-up of $z$, and
$D_X\subset X$ the strict transform of $D$.  It is a smooth $K3$
surface.  The projection $\P^3\map \P^2$ from $z$ induces a Mori
fibration $\pi\colon X\to \P^2$.  One easily computes that
\[
K_X  +  D_X \ = \ \sigma^*(K_{\mathbb{P}^3} + D).
\]
Theorem~B states that the Mf CY pair $(X, D_X)\to \P^2$ is
the only nontrivial element in the pliability set
$\mathcal{P}(\P^3,D)$.

\medskip

In the course of the proof, we will need the following result.

\begin{Proposition}
  \label{pic}
  Let $Y$ be a smooth projective variety of dimension $n$, $A$ a
  smooth semi-ample divisor on $Y$, and
  $\varphi\colon Y\rightarrow\mathbb{P}^N$ the morphism induced by the
  linear system $|mA|$ for $m\gg 0$.  Suppose that:
  \begin{enumerate}[(i)]
  \item the generic fibre dimension of $\varphi$ is $\leq n-3$;
  \item for all $x\in \mathbb{P}^N$, $\dim\big(\varphi^{-1}(x)\big)\leq n-2$.
  \end{enumerate}
  Then the cokernel of the restriction
  homomorphism $H^2(Y;\ZZ)\rightarrow H^2({A};\ZZ)$ is torsion-free.
\end{Proposition}

\begin{proof}
  Set $\mathcal{U} = Y\setminus {A}$, and denote by
  $i\colon \mathcal{U}\rightarrow Y$ and $j\colon A\rightarrow Y$ the
  inclusions.  Consider the long exact sequence
$$
\dots \ \rightarrow H^{2}_c(\mathcal{U},\ZZ)\
\xrightarrow{\textit{i}_*} \ H^2(Y,\ZZ)\ \xrightarrow
{\textit{j}^*}\  H^2({A},\ZZ)\ \rightarrow \ H^3_c(\mathcal{U},\ZZ)\ \rightarrow \ \dots \ 
$$
where $H^i_c(\mathcal{U},\ZZ)$ denotes the singular cohomology with
compact support. It is enough to show that
$H^3_c(\mathcal{U},\ZZ)$ is torsion free. By Poincar\'e duality,
$H^3_c(\mathcal{U},\ZZ)\cong H_{2n-3}(\mathcal{U},\ZZ)$.

Let $H\subset\mathbb{P}^N$ be the hyperplane such that
$\varphi^{*} H = m{A}$.  Then $\varphi$ restricts to a proper
morphism
$\varphi_{\mathcal{U}}\colon
\mathcal{U}\rightarrow\mathbb{P}^N\setminus H$.
For each integer $k$, we denote by $\phi(k)$ be the dimension of the
set of points $y\in \varphi(\mathcal{U})$ such that
$\dim\big(\varphi^{-1}(y)\big)=k$. If this set is empty, we set
$\phi(k) = -\infty$.  By the main result of~\cite[Part~II, Chap.~1,
Sec.~1.1$\star$]{GM88} (Homotopy Dimension with Large Fibres),
$\mathcal{U}$ has the homotopy type of a CW complex of real dimension
less than or equal to
\[
  n+\sup_{k}\big\{2k-n+\phi(k)+\inf\{\phi(k),0\}\big\}\ \leq \ 2n-3.
\]
--- where the last inequality follows from the assumptions~(i) and
~(ii). Finally, by \cite[Theorem 3]{Na67},
$H_{2n-3}(\mathcal{U},\ZZ)$ is torsion free.
\end{proof}

We are now ready to proceed with the proof of Theorem~B. We start by collecting a few useful facts on the
geometry of $D_X$. 

\begin{say}[The geometry of
  $D_X$] 
\label{mc1}\label{reminv}\label{autD} 
The  surface $D_X\subset X$  intersects the exceptional divisor
  $E\cong \P^2$ of $\sigma\colon X\rightarrow\mathbb{P}^3$ transversely along a smooth conic $e$.  Denote
  by $h$ the pull-back of a general hyperplane under
  $\sigma_{|D_X}\colon D_X\to \mathbb{P}^3$.  Then
  $\Pic(D_X)=\mathbb{Z}[h]\oplus \mathbb{Z}[e]$, and the intersection
  matrix of $\Pic(D_X)$ with respect to the basis $\big( [h],[e]\big)$
  is
\begin{equation}\label{intersection_matrix}
\left(\begin{array}{cc}
4 & 0 \\
0 & -2
\end{array}\right).
\end{equation} 
The condition that $\Cl(D)\cong \mathbb{Z}$ is generated by the class
of a hyperplane section guarantees that $D$ contains no lines, and so
the restriction $\pi_{|D_X}\colon D_X\to \mathbb{P}^2$ is finite of
degree $2$. It ramifies over a sextic curve.
 The associated involution $\tau\colon D_X\to D_X$ maps
the $(-2)$-curve $e$ to another $(-2)$-curve $e'$, the strict
transform of the intersection of $D$ with its tangent cone at
$z$. Note that $\NE(D_{X}) = \left\langle [e],[e']\right\rangle_+$, and
$e+e'\sim (\pi_{|D_X})^*\O_{\P^2}(2)\sim 2h-2e$. Thus, $e'\sim 2h-3e$.

\begin{lem}
  \label{lem:bir_k3}
 With the notation and assumptions of the preceding
 discussion~\ref{mc1}, we have
$$
\Bir(D_X) = G_D \rtimes \langle \tau\rangle,
$$
where $G_D\subset \PGL_4(\CC)$ is the group of projective automorphisms of
$D\subset \PP^3$.
\end{lem}

\begin{proof}
  Since $D_X$ is the unique minimal model of $D$, it is follows that
  $\Bir(D)=\Aut (D_X)$. Let us denote this group by $G$. It naturally
  acts on $\Pic(D_X)$ respecting the quadratic form and the Mori
  cone, thus we get a group homomorphism
\[
\rho \colon G \to \{\pm 1\}
\]
 where $\rho (g)=1$ if $\rho$ fixes $e$ and $e^\prime$, and $\rho (g)=-1$
 if $g$ exchanges $e$ with $e^\prime$. The involution $\tau$ exchanges
 $e$ and $e^\prime$, thus $G=G_D \rtimes \langle \tau \rangle$, where
 $G_D=Ker(\rho)$ is the subgroup of $\Aut(D_X)$ of elements that act trivially on $\Pic(D_X)$. Thus $G_D$ fixes the class $h\in \Pic(D_X)$ and
 hence it acts on the linear system $|\OO_D(1)|$ as projective linear
 transformations. 
\end{proof}

\begin{rem}
It is easy to see that if $D$ is general then the group $G_D$ of projective automorphisms of
$D\subset \PP^3$ is trivial. 
\end{rem}

\end{say}

We are now ready to prove the first conclusion of Theorem~B, i.e., that the pliability of the pair $(\PP^3, D)$ is the set with two
elements $\{(\PP^3,D)/\Spec \CC, (X, D_X)/ \PP^2\}$.

\begin{proof}[Proof of Theorem~B(1)]
  Let $(\P^3,D)$ be a Mf CY pair, where $D\subset\mathbb{P}^3$ is a
  quartic surface having exactly one singular point $z\in D$ of type
  $A_1$, and such that $\Cl(D)\cong \mathbb{Z}$ is generated by the
  class of a hyperplane section. Note that $(\P^3,D)$ is canonical.

 Let  $(Y,D_Y)\rightarrow S_Y$ be a Mf CY pair, and 
  $\varphi\colon (\mathbb{P}^3,D)\dasharrow (Y,D_Y)$ a volume
  preserving birational map that is not biregular.
By Theorem~\ref{fact}, the map $\varphi$ factors as the composition
  of volume preserving Sarkisov links. By Proposition~\ref{codim2}, the first step of any
  Sarkisov factorisation is a divisorial contraction with
  zero-dimensional centre.  By Lemma~\ref{buA1}, this divisorial
  contraction $\sigma\colon X\rightarrow\mathbb{P}^3$ is the blow-up of $z$.
Let $\pi\colon X\to \P^2$ be the Mori fibration induced by the projection
  from $z$, and assume that $\varphi\circ \sigma\colon X \dasharrow Y$ is not biregular. 
Let $D_X\subset X$ be the strict transform of $D$.
Since by assumption $D$ does not contain any line, the
  restriction $\pi_{|D_X}\colon D_X\to \mathbb{P}^2$ is a finite morphism of degree $2$.

Note that $X$ has Picard rank two, and the morphisms $\sigma$ and $\pi$
  are the two extremal contractions of $X$.  Therefore, the Sarkisov
  factorisation of $\varphi$ must proceed with a volume preserving
  divisorial contraction $g\colon (Z,D_Z)\rightarrow (X,D_X)$.  Since
  $D_X\subset X$ is smooth, Proposition~\ref{divextter} implies that the
  centre of the divisorial contraction $g\colon Z\rightarrow X$ is a
  curve $\cC\subset D_X$, and $D_Z$ is the strict transform of $D_X$ in
  $Z$.  At the generic point of $\cC$, the morphism
  $g\colon Z\rightarrow X$ coincides with the blow-up of $\cC$.  We
  denote by $F\subset Z$ the exceptional divisor of
  $g\colon Z\rightarrow X$.

  In order to describe the next link in a Sarkisov factorisation of
  $\varphi$, we describe the classes of irreducible curves that are
  contracted by $\pi\circ g$. They are either contained in fibers of
  $g_{|F}\colon F\rightarrow \cC$, or they are strict transforms of
  fibers of $\pi\colon X\to \P^2$.  Set
  $\Gamma = \pi(\cC)\subset\mathbb{P}^2$, and denote by $d$ the degree
  of the finite morphism $\pi_{|\cC}\colon \cC\to \Gamma$. We know that
  $d\in\{1,2\}$ (as we just said, $\cC\subset D_X$), but we will soon see that $d=1$. For all
  $q\in \mathbb{P}^2\setminus \Gamma$,
  $(\pi\circ g)^{-1}(q) \cong \P^1$.  For a general point
  $q\in \Gamma$, $(\pi\circ g)^{-1}(q)$ has $d+1$ rational components:
  the strict transform of $\pi^{-1}(q)$, and $d$ fibers of
  $g_{|F}\colon F\rightarrow \cC$, corresponding to the $d$ points of
  intersection $\pi^{-1}(q)\cap \cC$.

The next link in the Sarkisov factorisation of $\varphi$ is either of type~(I) or of type~(II). 
We show that it cannot be a Sarkisov link of type~(I), as in the following diagram:
  \[
  \begin{tikzpicture}[xscale=1.5,yscale=-1.2]
    \node (A0_1) at (1, 0) {$Z$};
    \node (A0_2) at (2, 0) {$X'$};
    \node (A1_0) at (0, 1) {$X$};
    \node (A1_2) at (2, 1) {$S$};
    \node (A2_0) at (0, 2) {$\mathbb{P}^2$};
    \path (A1_2) edge [->,swap]node [auto] {$\scriptstyle{r}$} (A2_0);
    \path (A0_2) edge [->]node [auto] {$\scriptstyle{\pi'}$} (A1_2);
    \path (A0_1) edge [->,dashed]node [auto] {$\scriptstyle{}$} (A0_2);
    \path (A0_1) edge [->,swap]node [auto] {$\scriptstyle{g}$} (A1_0);
    \path (A1_0) edge [->,swap]node [auto] {$\scriptstyle{\pi}$} (A2_0);
  \end{tikzpicture}
  \]
  Here $Z\dasharrow X'$ is a sequence of Mori flips, flops and antiflips, $\pi'\colon X'\to S$ is a Mori fiber space, and
  $r\colon S\rightarrow \mathbb{P}^2$ is a divisorial contraction with
  centre a point $q\in \mathbb{P}^2$.  Note that
  $(r\circ \pi')^{-1}(q)$ is a surface in $X'$, and so its strict
  transform in $Z$ is also a surface.  The commutativity of the above
  diagram then implies that $(\pi\circ g)^{-1}(q)$ is a surface, which
  is impossible since the fibers of $\pi\circ g$ are $1$-dimensional.

We conclude that the next link in a Sarkisov factorisation of $\varphi$ is of type~(II), as in the following diagram:
  \[
  \begin{tikzpicture}[xscale=1.5,yscale=-1.2]
    \node (A0_1) at (1, 0) {$Z$};
    \node (A0_2) at (2, 0) {$Z'$};
    \node (A1_0) at (0, 1) {$X$};
    \node (A1_3) at (3, 1) {$X'$};
    \node (A2_0) at (0, 2) {$\mathbb{P}^2$};
    \node (A2_3) at (3, 2) {$\mathbb{P}^2$};
    \path (A0_1) edge [->,swap]node [auto] {$\scriptstyle{g}$} (A1_0);
    \path (A1_3) edge [->]node [auto] {$\scriptstyle{\pi'}$} (A2_3);
    \path (A2_0) edge [-,double distance=1.5pt]node [auto] {$\scriptstyle{}$} (A2_3);
    \path (A1_0) edge [->,swap]node [auto] {$\scriptstyle{\pi}$} (A2_0);
    \path (A0_2) edge [->]node [auto] {$\scriptstyle{g'}$} (A1_3);
    \path (A0_1) edge [->,dashed]node [auto] {$\scriptstyle{\chi}$} (A0_2);
    \path (A1_0) edge [->,dashed]node [auto] {$\overline{\scriptstyle{g}}$} (A1_3);
  \end{tikzpicture}
  \]
  Here $\chi\colon Z\map Z'$ is a sequence of Mori flips, flops and
  antiflips, and $g'\colon Z'\to X'$ is a divisorial contraction.  We
  denote by $D_{X'}$ the strict transform of $D_X$ in $X'$.  It is
  normal by Lemma~\ref{D_normal}.  We will show that
  $\pi'\colon X'\to \P^2$ is a $\P^1$-bundle square equivalent to
  $\pi\colon X\to \P^2$, and that
  $\overline{g}=g'\circ\chi\circ g^{-1}\colon X\dasharrow X'$
  restricts to an isomorphism between $D_{X}$ and $D_{X'}$.

From the description of the curves contracted by $\pi\circ g$ above,
and the fact that $\chi\colon Z\map Z'$ is an isomorphism over the
complement of a finite subset of $\mathbb{P}^2$, we see that
$g'\colon Z'\to X'$ contracts the strict transform of
$\pi^{-1}(\Gamma)$ in $Z'$ onto a curve $\cC'\subset X'$, which is
mapped to $\Gamma$ by $\pi'$. Recall that $\pi_{|\cC}\colon \cC\to \Gamma$
is a finite morphism of degree $d\in\{1,2\}$. If $d=2$, then $D_{X'}$
would be singular along $\cC'$, and hence not normal. So we conclude
that $d=1$, and the general fiber of $\pi'\colon X'\to \P^2$ over
$\Gamma$ is irreducible, and thus isomorphic to $\P^1$. Thus,
$\pi'\colon X'\to \P^2$ is a $\P^1$-bundle over the complement of a
finite subset of $\mathbb{P}^2$. It follows from \cite[Theorem
5]{AR14} that $\pi'\colon X'\to \P^2$ is a $\P^1$-bundle.  It is
clearly square equivalent to $\pi\colon X\to \P^2$ via $\overline{g}$.

To show that the restricted map
$\overline{g}_{|D_X}\colon D_X\dasharrow D_{X'}$ is an isomorphism, we
first note that it does not contract any curve. This follows from the
commutativity of the diagram above, and the fact that $D_X$ does not
contain any fiber of $\pi$. By Zariski's Main Theorem, the birational
inverse of $\overline{g}_{|D_X}$ is a morphism.  Adjunction yields
that $K_{D_X}\sim 0$ and $K_{D_{X'}}\sim 0$. Since $D_X$ is smooth, in
particular terminal, we conclude that
$(\overline{g}_{|D_X})^{-1}\colon D_{X'}\to D_X$ is an isomorphism.

The same argument as above shows that the next link in the Sarkisov
factorisation of $\varphi$ cannot be of type~(I). It also shows that,
if it is of type~(II), then it ends with a $\P^1$-bundle
$\pi''\colon X''\to \P^2$, square equivalent to
$\pi'\colon X'\to \P^2$, and the birational map $X'\dasharrow X''$
restricts to an isomorphism between $D_{X'}$ and its strict transform
$D_{X''}$.  So, after a finite number of Sarkisov links of type~(II),
we reach a $\mathbb{P}^1$-bundle, which we keep denoting by
$\pi'\colon X'\to \P^2$, square equivalent to $\pi\colon X\to \P^2$,
and either the Sarkisov factorisation of $\varphi$ is finished, or it
must continue with a link of type~(III) or~(IV).  Note moreover that
the strict transform $D_{X'}$ of $D_X$ in $X'$ is a smooth member of
$|-K_{X'}|$ isomorphic to $D_X$, and it does not contain any fiber of
$\pi'$.  In order to prove the theorem, assuming that the Sarkisov
factorisation of $\varphi$ is not finished, we must show that $X'$ is
isomorphic to the blow-up of $\P^3$ at a point.

If the Sarkisov factorisation of $\varphi$ is not finished, then the
next link starts with a birational map corresponding to an extremal ray $R\subset
\NE(X')$. Let $\gamma\subset X'$ be a reduced and irreducible curve such that
$R=\bR_{\geq 0}[\gamma]$.  We shall show that $-K_{X'}\cdot \gamma \geq 0$.

Suppose for a contradiction that $-K_{X'}\cdot \gamma <0$.  Then the
contraction of the extremal ray $R=\bR_{\geq 0}[\gamma]$ is small, and
the Sarkisov link starts with an antiflip
$X' = Y^{+}\dasharrow Y^{-}$. We will show that $Y^{-}$ has worse than terminal
singularities, which is not allowed in the definition of volume
preserving Sarkisov link (Definition~\ref{dfn:4}).

The assumption that $-K_{X'}\cdot \gamma = D_{X'}\cdot \gamma <0$
implies that $\gamma\subset D_{X'}$. Recall from Paragraph~\ref{mc1}
that $\NE(D_{X'}) = \left\langle e, e'\right\rangle_+$, where $e$ and
$e'$ are $(-2)$-curves in $D_{X'}$. Therefore, either $\gamma=e$ or
$\gamma=e'$. Set $k:=-D_{X'}\cdot \gamma >0$.  It follows from
(\ref{intersection_matrix}) that $k$ is even, and hence $k\geq 2$.  By
Lemma~\ref{lem:inverse flips}, $Y^{-}$ has worse than terminal
singularities, as anticipated.

This contradiction proves that $-K_{X'}\cdot \gamma \geq 0$. The other extremal
ray of $\NE(X')$ is generated by the class of a fiber of the
$\P^1$-bundle $\pi'\colon X'\to \P^2$, which has positive intersection
with $-K_{X'}$.  Therefore, if $-K_{X'}\cdot \gamma > 0$ then $X'$ is
Fano.  If $-K_{X'}\cdot \gamma = 0$, then the contraction of the
extremal ray $R=\bR_{\geq 0}[\gamma]$, which is induced by the linear
system $\big|-mK_{X'}\big|$ for $m\gg 0$, is small --- otherwise
after the contraction we get a variety with strictly canonical
singularities, which is not allowed --- and the Sarkisov
link starts with a Mori flop. In this case $X'$ is weak Fano (i.e.,
$-K_{X'}$ is nef and big).

Theorem~\ref{thm:weakFanoVectorBundles} and the two lists that accompany it
show the rank~$2$ vector bundles $\cE$ on $\P^2$ with $c_1(\cE) \in
\{0,-1\}$ such that $\mathbb{P}(\cE)$ is Fano or weak
Fano. Below we follow the conventions on vector bundles summarised in
section~\ref{sec:fano-weak-fano}.

Let $\cE$ be the rank two vector bundle on $\P^2$ with
$c_1(\cE)\in \{0,-1\}$ such that $X'\cong \mathbb{P}(\cE)$, 
and denote by $\pi'\colon X'\to  \P^2$ the natural projection. 
In order to prove the theorem, we must show that
$\cE\cong \mathcal{O}_{\P^2}\oplus \mathcal{O}_{\P^2}(-1)$. To do so,
we compare the lattice $\Pic (D_{X'})$ with the sublattice obtained as
the image of the restriction homomorphism
\[
r\colon  \ \Pic (X') \ \to \ \Pic (D_{X'}).
\]
The Picard group $\Pic({X'})$ is generated by $L' =
\big[(\pi')^{*}\big(\O_{\P^2}(1)\big)\big]$ and $\xi =
c_1\big(\O_{\mathbb{P}(\cE)}(1)\big)$. 
Working in $A^\bullet (X^\prime)$ and using that $D_{X'} \ \sim \ -K_{X'} \ \sim 2\xi+\ (-c_1+3)L'$
(by Lemma~\ref{lem:canonical_in_PP(E)}), we compute the intersection matrix
of $r \Big(\Pic\big( \PP(\cE)\big)\Big)\subset \Pic (D_{X^\prime})$ in the basis $r(L'), r(\xi)$:
\begin{equation}\label{restriction_matrix}
\left(\begin{array}{cc}
L'^2\cdot D_{X^\prime} & \xi \cdot L' \cdot D_{X^\prime}\\ 
\xi \cdot L' \cdot D_{X^\prime} & \xi^2 \cdot D_{X^\prime}
\end{array}\right)
=
\left(\begin{array}{cc}
2 & c_1+3 \\ 
c_1+3 & c_1^2+3c_1-2c_2
\end{array}\right).
\end{equation}
Direct inspection shows that this matrix has rank two for all vector
bundles in the two lists, except for the one in (8), namely
$\mathcal{O}_{\P^2}(1)\oplus\mathcal{O}_{\P^2}(-2)$. On the other
hand, we cannot have
$X'\cong\P\big(\mathcal{O}_{\P^2}(1)\oplus\mathcal{O}_{\P^2}(-2)\big)$
because the anti-canonical contraction of
$\P\big(\mathcal{O}_{\P^2}(1)\oplus\mathcal{O}_{\P^2}(-2)\big)$ is
divisorial and not small.  In all other cases, the restriction
homomorphism $r\colon \Pic(X')\to \Pic(D_{X'})$ is injective with 
finite cokernel. We claim that it is also surjective. Pick $\alpha\in \Pic
(D_{X^\prime})$. By Proposition~\ref{pic}, the first Chern class $c_1(\alpha)\in
H^2(D_{X^\prime},\ZZ)$ is in the image of $H^2(X^\prime,\ZZ)$. By the
Lefschetz theorem on $(1,1)$ classes, $H^2(X^\prime,\ZZ)=\Pic (X^\prime)$, hence
there is a line bundle $\widetilde \alpha$ on $X^\prime$ such that
\[
c_1\left(r(\widetilde{\alpha})\right) =c_1(\alpha).
\]
But $c_1\colon \Pic (D_{X^\prime})\to H^2(D_{X^\prime},\ZZ)$ is
injective, hence $r(\widetilde{\alpha})=\alpha$.  It follows that the
matrices \eqref{intersection_matrix} and \eqref{restriction_matrix}
must have the same determinant.  One checks easily that this only
happens in case (1), i.e., when
$X'\cong \mathbb{P}\big(\mathcal{O}_{\P^2}\oplus
\mathcal{O}_{\P^2}(-1)\big)$ is the blow-up of $\P^3$ at a point.
Denote by $D'$ the image of $D_{X^\prime}$ in $\P^3$.
Since the blow-up $(X^\prime, D_{X^\prime})\rightarrow (\P^3, D')$ is volume preserving, 
its restriction to $D_{X^\prime}$ contracts one of the two $(-2)$-curves $e$ or $e'$.
So we have $D\cong D'$, and  $(\P^3, D)\cong (\P^3, D')$.
\end{proof}

Next, we prove the second conclusion of Theorem~B, that is, we describe the group $ \Bir(\PP^3, D)$ of volume preserving birational self-maps of $(\P^3,D)$.

\begin{proof}[{Proof of Theorem B}(2)]
By Proposition~\ref{group1}, since $(\P^3,D)$ is canonical, there is a restriction homomorphism 
\begin{equation}\label{restriction}
r\colon \Bir(X, D) \ \to \ \Bir(D).
\end{equation}

We use the same notation as above: $\sigma\colon X\rightarrow\mathbb{P}^3$
denotes the blow-up of the singular point $z \in D$, $D_X\subset X$
the strict transform of $D$, and $\pi\colon X\to \P^2$ the fibration
induced by the projection from $z$.  Recall from Paragraph~\ref{autD}
that the restriction $\pi_{|D_X}\colon D_X\to \mathbb{P}^2$ is a
double cover, and denote by $\tau\colon D_X\to D_X$ the associated
involution. Lemma~\ref{lem:bir_k3} states that
$$
\Bir(D)= \Aut(D_X) \cong \Aut(\PP^3,D) \rtimes \langle \tau \rangle.
$$
After a change of coordinates, we can write the equation of $D$ in $\P^3$ as 
$$
x_3^{2}A(x_0,x_1,x_2)+x_3B(x_0,x_1,x_2)+C(x_0,x_1,x_2)=0,
$$
where $A=A(x_0,x_1,x_2)$, $B=B(x_0,x_1,x_2)$ and $C=C(x_0,x_1,x_2)$ are homogeneous of degree $2$, $3$ and $4$, respectively. 
The singular point of $D$ has coordinates $z=[0:0:0:1]$.

\begin{claim}\label{splitting}
The restriction homomorphism \eqref{restriction} is surjective and admits a splitting
$$
\Bir(\mathbb{P}^3,D) \ \xrightarrow{\curvearrowleft} \ \Bir(D) = \Aut(\PP^3,D) \rtimes \langle \tau \rangle. 
$$
\end{claim}

\begin{proof}[{Proof of Claim~\ref{splitting}}]
To see that $r\colon \Bir(\PP^3, D) \ \to \ \Bir(D)= \Aut(\PP^3,D) \rtimes \langle \tau \rangle$ is surjective notice that the birational involution 
\begin{equation}\label{induce_tau}
\varphi(x_0:x_1:x_2:x_3) = ( Ax_0 : Ax_1 : Ax_2 : -A x_3-B),
\end{equation}
restricts to the nontrivial birational involution $\tau$.

Consider the splitting $\Bir(D) \to \Bir(\PP^3, D)$ of $r$ that is canonical on the normal subgroup  $\Aut(\PP^3,D)$, and sends $\tau$ to $\varphi$. 
To show that this is well-defined, we must check that, for any automorphism $h\in \Aut(\PP^3,D)$, we have $\varphi\circ h\circ \varphi\in \Aut(\PP^3,D)$. 
In order to prove this, we first describe a volume preserving Sarkisov factorisation of $\varphi$:
\[
  \begin{tikzpicture}[xscale=1.5,yscale=-1.2]
    \node (b) at (-1.5, 2.5) {$\mathbb{P}^3$};
    \node (e) at (5.5, 2.5) {$\mathbb{P}^3$};
    \path (A1_0) edge [->,swap]node [auto] {$\scriptstyle{\sigma}$} (b);
    \node (A0_1) at (1, 0) {$Z$};
    \node (A0_2) at (3, 0) {$Z'$};
    \node (A1_0) at (0, 1) {$X$};
    \node (A1_2) at (2, 1) {$X'$};
    \node (A1_3) at (4, 1) {$X$};
    \node (A2_0) at (0, 2) {$\mathbb{P}^2$};
    \node (A2_2) at (2, 2) {$\mathbb{P}^2$};
    \node (A2_3) at (4, 2) {$\mathbb{P}^2$};
    \path (A0_1) edge [->,swap]node [auto] {$\scriptstyle{\alpha}$} (A1_0);
    \path (A0_1) edge [->,swap]node [auto] {$\scriptstyle{\beta}$} (A1_2);
    \path (A1_3) edge [->]node [auto] {$\scriptstyle{\pi}$} (A2_3);
    \path (A1_2) edge [->]node [auto] {$\scriptstyle{\pi'}$} (A2_2);
    \path (A2_0) edge [-,double distance=1.5pt]node [auto] {$\scriptstyle{}$} (A2_2);
    \path (A2_2) edge [-,double distance=1.5pt]node [auto] {$\scriptstyle{}$} (A2_3);
    \path (A1_0) edge [->,swap]node [auto] {$\scriptstyle{\pi}$} (A2_0);
    \path (A0_2) edge [->]node [auto] {$\scriptstyle{\delta}$} (A1_3);
    \path (A0_2) edge [->]node [auto] {$\scriptstyle{\gamma}$} (A1_2);
    \path (A1_0) edge [->,dashed]node [auto] {} (A1_2);
    \path (A1_2) edge [->,dashed]node [auto] {} (A1_3);
    \path (A1_3) edge [->]node [auto] {$\scriptstyle{\sigma}$} (e);
    \path (b) edge [->,dashed,bend right=-15]node [auto] {$\scriptstyle{\varphi}$} (e);
  \end{tikzpicture}
  \]
  The factorisation starts with the blow-up
  $\sigma\colon X\rightarrow\mathbb{P}^3$ of $z$. Denote by $E\subset X$
  its exceptional divisor. The base locus of $\varphi\circ \sigma$ contains
  the curve $e=E\cap D_X$.  Note that $\pi\colon X\to \P^2$ maps $e$
  isomorphically onto the conic $\big(A=0\big)\subset \PP^2$, and
  the cylinder $\pi^{-1}\big(A=0\big)\subset X$ is precisely the strict
  transform of the tangent cone of $D\subset \PP^3$ at $z$.  The next
  link in the Sarkisov factorisation is the composition
  $\beta\circ\alpha^{-1}$, where $\alpha$ and $\beta$ are described as
  follows.  The morphism $\alpha:Z\to X$ is the blow-up of
  $e= E\cap D_X$. Denote by $F\subset Z$ its exceptional divisor. The
  base locus of $\varphi\circ \sigma\circ \alpha$ contains the curve
  $e_F=F\cap D_Z$.  The morphism
  $\beta:Z\to X'\cong\PP\big(\mathcal{O}_{\P^2}\oplus
  \mathcal{O}_{\P^2}(3)\big)$ contracts the rulings of the strict
  transform of the tangent cone of $D\subset \PP^3$ at $z$, and maps
  $F\subset Z$ isomorphically onto the cylinder
  $F'=(\pi')^*\big(A=0\big)\subset X'$. The base locus of
  $\varphi\circ \sigma\circ \alpha \circ \beta^{-1}$ consists of the curve
  $e_{F'}=\beta(e_F)\subset F'\cap D_{X'}$, which is mapped by $\pi'$
  isomorphically onto the conic $\big(A=0\big)\subset \PP^2$.  The
  next link in the Sarkisov factorisation is the composition
  $\delta\circ\gamma^{-1}$, where $\gamma:Z'\to X'$ is the blow-up of
  $e_{F'}$, and $\delta:Z'\to X$ contracts the rulings of the strict
  transform of the cylinder $F'$ on $Z'$.  The factorisation then ends
  with the blow-up $\sigma\colon X\rightarrow\mathbb{P}^3$ of $z$.

  Any automorphism $h\in \Aut(\PP^3,D)$ fixes $z$ and stabilises the tangent cone of $D$ at $z$. 
  So it lifts to an automorphism of $X$ that
 stabilises $E$, $D_{X}$, and the rulings over
  $\big(A=0\big)$. Since $\alpha$ blows-up $E\cap D_{X}$, $h$ lifts to
  an automorphism of $Z$. On the other hand, $\beta$ contracts the
  ruling of the strict transform of $\pi^{-1}\big(A=0\big)$, and thus $h$
  descends to an automorphism of $X'$ that stabilises $D_{X'}$, the
  ruling over $\big(A=0\big)$, and each of the components of
  $D_{X'}\cap (\pi')^*\big(A=0\big)$. The same argument shows that $h$ lifts
  to $Z'$, and then descends to $X$, always stabilising the strict
  transform of $E$.  This shows that the birational map
  $\varphi\circ h$ has the same Sarkisov decomposition as $\varphi$,
  and therefore $\varphi\circ h\circ \varphi$ is biregular.
\end{proof}

It follows from Claim~\ref{splitting} that there is a split exact sequence:
 \begin{equation}
   \label{eq:exact_seq}
1\to \mathbb{G} \to \Bir (\PP^3, D) \ \xrightarrow{\curvearrowleft} \Bir (D) \to 1,
 \end{equation}
 where $\mathbb{G}$ is the group of birational self maps of $\PP^3$
 fixing $D$ pointwise.  We saw in the proof of Theorem B(1) that any
 $\psi\in \SBir(\P^3,D) $ preserves the star of lines through
 $z$. Therefore, we can identify $\mathbb{G}$ with the group of
 birational self-maps of $X$ over $\P^2$ fixing $D_X$ pointwise.

 We view $X$ as a model of $\P^1$ over $\mathbb{C}(x,y)$ with
 projective coordinates $(u:v)$.  Setting $a(x,y)= A(1,x,y)$,
 $b(x,y)= B(1,x,y)$ and $c(x,y)= C(1,x,y)$, $\mathbb{G}$ becomes the identity
 component of the subgroup $G_Q$ of $PGL\big(2,\mathbb{C}(x,y)\big)$
 of projective transformations preserving the quadratic form
$$
Q(u,v) \ =\ a(x,y)u^2+b(x,y)uv+c(x,y)v^2
$$
up to scaling. 

\begin{Lemma}\label{quad}
Let $Q(u,v)=Au^2+Buv+Cv^2$ be a quadratic form with coefficients in a field $K$, and let $G_{Q}$ be the subgroup of $PGL(2,K)$ of projective transformations preserving $Q$ up to scalar:
$$G_Q := \left\lbrace
\phi = \left(\begin{array}{cc}
  \alpha & \beta \\ 
  \gamma & \delta
  \end{array}\right)\in PGL(2,K)\: \Big| \: Q(\phi(u,v)) = \lambda Q(u,v),\: \lambda\in K\setminus\{0\}\right\rbrace.
$$
Then $G_{Q}$ has two irreducible components, given by 
\begin{equation}\label{comp}
\left\lbrace\begin{array}{l}
  \alpha = -\frac{B}{A}\gamma+\delta, \\ 
  \beta = -\frac{C}{A}\gamma    
  \end{array}\right.  \ \ \ and \ \ \ \
\left\lbrace\begin{array}{l}
  \alpha = -\delta, \\ 
  \beta = \frac{C}{A}\gamma-\frac{B}{A}\delta.
  \end{array}\right.   
\end{equation}
\end{Lemma}
\begin{proof}
We have that
$$
Q(\phi(u,v)) = u^2(\alpha^2A+\alpha\gamma B+\gamma^2C)+uv(2\alpha\beta A+\alpha\delta B + \beta\gamma B +2\gamma\delta C)+v^2(\beta^2 A+\beta\delta B+\delta^2 C).
$$
Therefore, $
\phi = \left(\begin{array}{cc}
  \alpha & \beta \\ 
  \gamma & \delta
  \end{array}\right)\in G_{Q}$ if and only if 
$$
\rank\left(\begin{array}{ccc}
A & B & C\\ 
\alpha^2A+\alpha\gamma B+\gamma^2C & 2\alpha\beta A+\alpha\delta B + \beta\gamma B +2\gamma\delta C & \beta^2 A+\beta\delta B+\delta^2 C
\end{array}\right) \leq 1,
$$
or, equivalently, if and only if
\begin{equation}\label{sysquad}
\left\lbrace\begin{array}{l}
\alpha^2 AC + \alpha\gamma BC + \gamma^2 C^2-\beta^2 A^2-\beta\delta AB-\delta^2 AC = 0,\\
\alpha^2 AB + \alpha\gamma B^2 + \gamma^2 BC - 2\alpha\beta A^2 - \alpha\delta AB - \beta\gamma AB - 2\gamma\delta AC = 0.
\end{array}\right.  
\end{equation}
If the $4$-tuple $(\alpha,\beta,\gamma,\delta)\in \mathbb{C}^4$ satisfies the pair of equations \eqref{sysquad},
then it satisfies one of the following pairs of equations: 
$$
\left\lbrace\begin{array}{l}
  \alpha = -\delta, \\ 
  \beta = \frac{C}{A}\gamma-\frac{B}{A}\delta,
  \end{array}\right.   
\left\lbrace\begin{array}{l}
  \alpha = -\frac{B}{A}\gamma+\delta, \\ 
  \beta = -\frac{C}{A}\gamma,
  \end{array}\right. 
\left\lbrace\begin{array}{l}
  \alpha = \frac{-B+\varepsilon}{2A}\gamma, \\ 
  \beta = \frac{B(B-\varepsilon)-4AC}{2A\varepsilon}\delta,
  \end{array}\right.   
\left\lbrace\begin{array}{l}
  \alpha = \frac{-B-\varepsilon}{2A}\gamma, \\ 
  \beta = -\frac{B(B+\varepsilon)-4AC}{2A\varepsilon}\delta,
  \end{array}\right.
$$
where $\varepsilon\in \mathbb{C}$ is such that $\varepsilon^2 = B^2-4AC$. 
On the other hand, if $(\alpha,\beta,\gamma,\delta)\in \mathbb{C}^4$ satisfies the third or the fourth pair of equations, then 
$\det\left(\begin{array}{cc}
  \alpha & \beta \\ 
  \gamma & \delta
  \end{array}\right) =0$.
Therefore $G_{Q}$ has two irreducible components, given by (\ref{comp}).
\end{proof}

Lemma~\ref{quad} shows that $G_Q$ has two irreducible components over $\mathbb{C}(x,y)$:
\begin{equation}\label{components_of_G}
G_Q^1  = \left\{\left(\begin{array}{cc}
-\frac{b(x,y)}{a(x,y)}\alpha+\beta & -\frac{c(x,y)}{a(x,y)}\alpha\\
\alpha & \beta  
\end{array}\right)\right\}, \ \ 
G_Q^2  =  \left\{\left(\begin{array}{cc}
-\beta & \frac{c(x,y)}{a(x,y)}\alpha-\frac{b(x,y)}{a(x,y)}\beta\\
\alpha & \beta  
\end{array}\right)\right\}.
\end{equation}
The component $G_Q^1$ contains the identity, and is precisely the
group $\mathbb{G}$ of~\eqref{eq:exact_seq}. In order to describe $\mathbb{G}$ as a form of $\mathbb{G}_m$, set $w \ = \ 2a(x,y)u + b(x,y)v$, and compute
$$
4a(x,y)Q(u,v) \  = \ w^2-\delta(x,y)v^2,
$$
where $\delta(x,y)=\Delta(1,x,y)$ and $\Delta = B^2-4AC$. 
Applying Lemma~\ref{quad} with the new projective coordinates $(v:w)$, $\mathbb{G}$ can be presented as 
the subgroup of $PGL\big(2,\mathbb{C}(x,y)\big)$
of elements of the form 
$$
\left(\begin{array}{cc}
U & \delta(x,y) V\\
V & U  
\end{array}\right)
$$
where $U,V\in \mathbb{C}(x,y)$ are such that $U^2-\delta(x,y) V^2 = 1$. 
This is the form of $\mathbb{G}_m$ described in \cite[Chapter 2, Example 2.3.2 (c)]{PS94}.
\end{proof}

\begin{Remark}
We can write down all the elements of $\SBir(\mathbb{P}^3,D)$ explicitly. 
It follows from \eqref{components_of_G} above that they can be written in one of the following two forms: 
$$
\varphi^1\colon  \ (x_0:x_1:x_2:x_3) \ \mapsto \ \big( A(Fx_3+G)x_0 : A(Fx_3+G)x_1 : A(Fx_3+G)x_2 : (AG-BF)x_3-CF \big)
$$ 
where either $F=0$ and $\deg(G)=0$, or $F,G\in \mathbb{C}[x_0,x_1,x_2]$ are homogeneous with $\deg(G)=\deg(F)+1$;
$$
\varphi^2\colon \ (x_0:x_1:x_2:x_3) \ \mapsto \ \big( A(Fx_3+G)x_0 : A(Fx_3+G)x_1 : A(Fx_3+G)x_2 : -AGx_3+CF-BG \big)
$$
where either $F=0$ and $\deg(G)=0$, or $F,G\in \mathbb{C}[x_0,x_1,x_2]$ are homogeneous with $\deg(G)=\deg(F)+1$.
\end{Remark}

\section{Proof of Theorem~C}
\label{section:A2}

In this section, we determine the pliability of Mf CY pairs of the
form $(\mathbb{P}^3,D)$, where $D$ is a quartic surface having exactly
one singular point $z\in D$ of type $A_2$, and such that
$\Cl(D)\cong \mathbb{Z}\cdot \OO_D(1)$ is generated by the class of a
hyperplane section. This last condition implies in particular that
$D$ does not contain lines. After a change of coordinates, we may
assume that the singular point is $z=[0:0:0:1]$, and write the
equation of $D$ as
\[
D = \Bigl(x_0x_1x_3^2+Bx_3+C
=0\Bigr) \subset \PP^3,
\]
where $B=B(x_0,x_1,x_2)$ and $C=C(x_0,x_1,x_2)$ are 
homogeneous polynomials of degree $3$ and $4$, respectively.

At the end of Section~\ref{sec:sarkisov}, we constructed volume
preserving Sarkisov links between the Mf CY pairs from Table 1:
$$
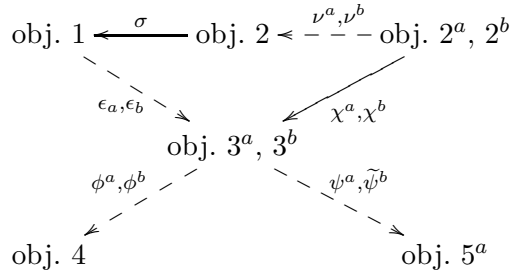
\xymatrix{\text{obj.~$1$}\ar@{-->}[dr]_{\epsilon_a, \epsilon_b} &
  \text{obj.~$2$} \ar@{->}[l]_\sigma & 
\text{obj.~$2^a$, $2^b$} \ar@{-->}[l]_{\nu^a, \nu^b}
\ar@{->}[dl]^{\chi^a, \chi^b}\\
 & \text{obj.~$3^a$, $3^b$} \ar@{-->}[dl]_{\phi^a, \phi^b}
 \ar@{-->}[dr]^{\psi^a, \widetilde{\psi}^b}& \\
 \text{obj.~$4$}& &\text{obj.~$5^a$} }
$$
In order to prove Theorem~C, we will show that these are all the
Sarkisov links from these Mf CY pairs, except for chains of square
equivalent Sarkisov links of type~(II) from objects 2, $2^a$ and $2^b$.

\begin{proof}[Proof of Theorem~C]
  Let $(Y,D_Y)/T$ be a Mf CY pair, and
\[
\Phi \colon (\PP^3, D) \dasharrow (Y, D_Y)
\]
a volume preserving birational map.
The goal is to show that $(Y,D_Y)\to T$ is square equivalent to one of
the objects in the conclusion of Theorem~C:
\begin{enumerate}[(1)]
\item $(\PP^3, D)$ (object~$1$);
\item $(X,D_X) = \left(\FF^3_1, D_{\binom{2}{2}}\right)$ (object~$2$); 
\item $(\PP(1^3,2),D^a_5)$ or $(\PP(1^3,2),D^b_5)$ (object~$3^a$ or $3^b$);
\item a member of the $3$-parameter family $\big\{(X_4,
  D_{3,4})\big\}$, with $X_4\subset \PP(1^3,2^2)$ (object~$4$); 
\item a member of the $7$-parameter family $\big\{(X_4,
  D_{2,4})\big\}$, with $X_4\subset \PP(1^4,2)$ (object~$5^a$).
\end{enumerate}

By the Sarkisov program for volume preserving birational maps of Mf CY
pairs (Theorem~\ref{fact}), there is a sequence of Mf CY pairs 
$$(\PP^3,D)/\Spec \CC=(X_0,D_0)/S_0 \ , \
(X_1,D_1)/S_1 \ , \ \dots \ , \ (X_n,D_n)/S_n=(Y,D_Y)/ T \ ,$$ 
and  volume preserving Sarkisov links
\[
\Phi_i \colon (X_{i-1},D_{i-1})/S_{i-1} \dasharrow (X_i,D_i)/S_i,
\]
$i=1,\dots, n$, such that $\Phi=\Phi_n\circ \cdots \circ \Phi_1$.

We prove by increasing induction on $i$ that each $(X_i,D_i)/S_i$ is
either isomorphic to one of the objects $1$, $3^a$, $3^b$, $4$, and $5^a$,
or it is obtained from $(X,D_X)/{\PP^2}$ (object $2$) after finitely
many (possibly none) volume preserving Sarkisov links of type~(II). In the latter
case, it is in particular square equivalent to object $2$. The base case
$i=0$ is clear since $(X_0,D_0)/S_0=(\PP^3,D)/\Spec \CC$ is
object~$1$.

For $i>0$, we assume by induction that $(X_{i-1},D_{i-1})/S_{i-1}$ is
either isomorphic to one of the objects 1, $3^a$, $3^b$, $4$, and $5^a$, or 
is obtained from $(X,D_X)/{\PP^2}$ after finitely many volume preserving Sarkisov links of type~(II).
We shall prove that the same holds for $(X_i,D_i)/S_i$. 
We discuss several cases depending on the nature
of $(X_{i-1},D_{i-1})/S_{i-1}$.

\smallskip

\paragraph{\bf{Case 1.}} Suppose that $(X_{i-1},D_{i-1})/S_{i-1}\cong (\PP^3,D)/\Spec \CC$  (object~$1$). We prove in Lemma~\ref{lem:Link1} below that the only volume
preserving Sarkisov links 
$$(\PP^3,D)/\Spec \CC \dasharrow (X^\dagger,D^\dagger)/S^\dagger$$ 
are the maps $\sigma^{-1},
\epsilon_a, \epsilon_b$ to objects~$2$, $3^a$ and~$3^b$, respectively. It follows that
$(X_i,D_i)/S_i$ is isomorphic to one of the objects~$2$, $3^a$ or~$3^b$.

\smallskip

\paragraph{\bf{Case 2.}}  Suppose that $(X_{i-1},D_{i-1})/S_{i-1}$ is
obtained from $(X,D_X)/{\PP^2}$ after finitely many volume preserving Sarkisov links of type~(II).
We show in Lemma~\ref{lem:Link2} below that one of the following holds: 
\begin{enumerate}[(a)]
\item $\Phi_i$ is a Sarkisov link of type~(II);
\item $(X_{i-1},D_{i-1})/S_{i-1}$ is isomorphic to object~$2$, and
  $\Phi_i=\sigma$;
\item $(X_{i-1},D_{i-1})/S_{i-1}$ is isomorphic to object~$2^a$, and
  $\Phi_i = \chi^a$;
\item $(X_{i-1},D_{i-1})/S_{i-1}$ is isomorphic to object~$2^b$, and
  $\Phi_i=\chi^b$. 
\end{enumerate}
In case (a), it follows that $(X_i,D_i)/S_i$ is obtained from
$(X,D_X)$ after finitely many volume preserving Sarkisov links of
type~(II). 
In cases (b), (c) and (d), it follows that $(X_i,D_i)/S_i$ is
isomorphic to object~$1$, $3^a$ and $3^b$, respectively. 

\smallskip

\paragraph{\bf{Case 3.}}  Suppose that $(X_{i-1},D_{i-1})/S_{i-1}$ is isomorphic to
    object~$3^a$ (case $3^b$ is similar). We prove in Lemma~\ref{lem:Link3} below
that  the only volume preserving Sarkisov links 
$$(\PP(1^3,2),D^a_5)/\Spec \CC \dasharrow (X^\dagger,D^\dagger)/S^\dagger$$ 
are the maps $\epsilon_a^{-1},(\chi^a)^{-1}, \phi^a, \psi^a$ to
objects~$1$, $2^a$, $4$ and~$5^a$, respectively.

\smallskip

\paragraph{\bf{Case 4.}}  Suppose that $(X_{i-1},D_{i-1})/S_{i-1}$ is isomorphic to
    object~$4$. We prove in Lemma~\ref{lem:Link4} below that  the only volume
preserving Sarkisov links 
$$(X_4, D_{3,4})/\Spec \CC \dasharrow (X^\dagger,D^\dagger)/S^\dagger$$ 
are the maps
$(\phi^a)^{-1}$, $(\phi^b)^{-1}$, and thus $(X_i,D_i)/S_i$ is isomorphic to object~$3^a$ or $3^b$.

\smallskip

\paragraph{\bf{Case 5.}}  Suppose that $(X_{i-1},D_{i-1})/S_{i-1}$ is
isomorphic to object~$5^a$. We prove in  Lemma~\ref{lem:Link5} below that the only volume
preserving Sarkisov links 
$$(X_4, D_{2,4})/\Spec \CC \dasharrow (X^\dagger,D^\dagger)/S^\dagger$$ 
are the maps $(\psi^a)^{-1}$ and $(\widetilde{\psi}^b)^{-1}$, and hence
$(X_i,D_i)/S_i$ is isomorphic to object~$3^a$ or~$3^b$.
\end{proof}

\begin{lem}\label{lem:Link1}
The only volume preserving Sarkisov links 
\[\Psi \colon (\PP^3,D)/\Spec \CC \dasharrow
(X^\dagger,D^\dagger)/S^\dagger\]
are the maps $\sigma^{-1}, \epsilon_a$, and $\epsilon_b$ described in Examples~\ref{link:1->2}
and \ref{link:1->3a}.
\end{lem}

\begin{proof}
  Since $\PP^3$ has Picard rank~$1$, the link $\Psi$ begins with a
  volume preserving divisorial contraction
  $f\colon (Z,D_Z)\to (\PP^3, D)$.  By Proposition~\ref{codim2} and
  Lemma~\ref{buA1}, the contraction $f$ is either the usual blow-up of
  the singular point $z$, or the weighted blow-up at $z$ with weights
  $(2,1,1)$ or $(1,2,1)$ with respect to the affine coordinates $x_0$,
  $x_1$ and $x_2$ in the affine chart $(x_3=1)$.  If $f=\sigma$ is the
  blow-up of $z$, then $\Psi=\sigma^{-1}$.  If $f$ is the weighted
  blow-up at $z$ with weights $(2,1,1)$ (respectively $(1,2,1)$), then
  $\Psi$ is the composition of $f$ with the contraction of the strict
  transform of the divisor $(x_0=0)$ (respectively $(x_1=0)$), as
  described in Example~\ref{link:1->3a}. Hence, $\Psi=\epsilon_b$ and
  $(X^\dagger,D^\dagger)/S^\dagger\cong (\PP(1^3,2),D^b_5)$
  (respectively $\Psi=\epsilon_a$ and
  $(X^\dagger,D^\dagger)/S^\dagger\cong (\PP(1^3,2),D^a_5)$).
\end{proof}

As before, we denote by $\sigma\colon X\to \PP^3$ the blow-up of
the singular point $z\in D$, by $D_X$ the (smooth) strict transform of $D$ in $X$, 
and by $\pi\colon X\to \P^2$ the fibration induced by the projection from $z$. 
In order to determine all the volume preserving Sarkisov links from $(X,D_X)/{\PP^2}$, 
we need a good understanding of the geometry of the smooth K3 surface $D_X$.

\begin{say}[The geometry of $D_X$]\label{D_A2}
Denote by $E\cong \P^2$ the exceptional divisor of the blow-up $\sigma:X\to \PP^3$. 
The intersection of $D_X$ with $E$ 
is the union of two $(-2)$-curves $e_0$ and $e_1$. 
These curves are mapped isomorphically via $\pi$ to the lines $(x_0=0)$ and $(x_1=0)$ in $\mathbb{P}^2$, respectively. 
Denote by $h$ the pull-back of the hyperplane class under
$\sigma_{|D_X} \colon D_X\to \mathbb{P}^3$. 
Then $\Pic(D_X)= \mathbb{Z}[h]\oplus \mathbb{Z}[e_0] \oplus \mathbb{Z}[e_1]$, 
and the intersection matrix of $\Pic(D_X)$ with respect to the basis $\big( [h],[e_0], [e_1]\big)$  is 
\[
\left(\begin{array}{ccc}
4 & 0 & 0 \\
0 & -2 & 1\\
0 & 1 & -2
\end{array}\right).
\]
By assumption, $D$ does not contain lines. Hence, the morphism
$\pi_{|D_X}\colon D_X\to \mathbb{P}^2$ is finite of degree $2$.  Set
$\alpha = (\pi_{|D_X})^*\mathcal{O}_{\mathbb{P}^2}(1)$, and denote by
$\tau \colon D_X\to D_X$ the involution associated to $\pi_{|D_X}$.
Then $\tau\colon D_X\to D_X$ maps the $(-2)$-curves $e_0$ and $e_1$ to
other $(-2)$-curves $e'_0$ and $e'_1$, respectively. Note that
\[
\alpha \ = \ (\pi_{|D_X})^* \mathcal{O}_{\mathbb{P}^2}(1) \ \sim \ h-e_0-e_1 \ \sim \ e_0+e'_0 \ \sim \ e_1+e'_1 \  .
\]
Thus  $e'_0\sim h-2e_0-e_1$, and $e'_1\sim h-2e_1-e_0$. 
The intersection matrix of $\Pic(D_X)$ with respect to the basis $\big( [\alpha] ,[e_0], [e_1]\big)$  is 
\[
\left(\begin{array}{ccc}
2 & 1 & 1 \\
1 & -2 & 1\\
1 & 1 & -2
\end{array}\right) .
\]
The following intersection numbers will be useful later on
\[
\left\lbrace\begin{array}{l}
\alpha\cdot e_0 = \alpha\cdot e_1 = \alpha\cdot e'_0 = \alpha\cdot e'_1 = 1,\\ 
e_0\cdot e'_1 = e'_0\cdot e_1 =  0,\\ 
e_0\cdot e_0' = e_1\cdot e_1'  = 3.
\end{array}\right. 
\]
The $(-2)$-curves $e_0$, $e_1$, $e'_0$ and $e'_1$ generate extremal rays of $\NE(D_{X})$. 
Next we show that these are all:
\begin{equation}\label{mc2}
\NE(D_{X}) \ = \  \big\langle [e_0], [e_1], [e'_0], [e'_1]\big\rangle_+. 
\end{equation}
Indeed, let $C\subset D_X$ be an irreducible curve different from $e_0$, $e_1$, $e'_0$ and $e'_1$, and write 
$C\sim dh-m_0e_0-m_1e_1$. By intersecting $C$ with $h$, $e_0$, $e_1$, $e'_0$ and $e'_1$, we get that 
$$
\left\lbrace
\begin{array}{l}
 d > 0, \\
0 \leq m_0 \leq \frac{4}{3}d,\\ 
0 \leq m_1 \leq \frac{4}{3}d.
\end{array} 
\right.
$$
Therefore, we may write
$$
C\ \equiv \ \frac{d}{2}{e'}_0 \ +\ \frac{d}{2}{e'}_1\ +\ \left(\frac{3}{2}d-m_0\right)e_0\ +\ \left(\frac{3}{2}d-m_1\right)e_1,
$$
with $\frac{3}{2}d-m_0> 0$ and $\frac{3}{2}d-m_1> 0$.  This gives
\eqref{mc2}.
\end{say}

\begin{Remark}\label{Cox_coords_D}
Fix homogeneous coordinates $x_0, x_1, x_2, x_3, x$ on $X\cong \FF_1^3$ with weights:
\[
\begin{array}{ccccc}
x_0 & x_1 & x_2 & x_3 & x \\
\hline
1 & 1 & 1 & 0 & -1 \\
0 & 0 & 0 & 1 &  1 
\end{array}
\]
as in Table~1. In these coordinates, $\sigma \colon X\to \PP^3$ is given by
$$
(x_0, x_1, x_2, x_3,x) \mapsto (xx_0, xx_1,  xx_2, x_3)
$$
while $\pi \colon X\to \PP^2$ is given by $(x_0, x_1, x_2, x_3,x) \mapsto (x_0,x_1,x_2)$.
The equation of $D_X\subset X$ is
\[
x_0x_1x_3^2+Bx_3x+Cx^2 =0
\]
and the $(-2)$-curves $e_0$, $e_1$, $e'_0$ and $e'_1$ of the
discussion above are given by:
\[
\begin{array}{l}
e_0 = \{x = x_0 = 0\},\\ 
e_1 = \{x = x_1 = 0\},\\ 
e'_0 = \{x_0 = Bx_3+Cx = 0\},\\ 
e'_1 = \{x_1 = Bx_3+Cx = 0\}.
\end{array} 
\]
\end{Remark}

We are ready to determine the volume preserving Sarkisov links from $(X,D_X)/{\PP^2}$, and, more generally, from any 
Mf CY pair obtained from $\pi \colon (X,D_X) \to \PP^2$ after finitely
many volume preserving Sarkisov links of type~(II). 

\begin{lem}
  \label{lem:Link2}
Suppose that $\pi^\prime \colon (X^\prime, D^\prime)\to \PP^2$ is a Mf CY pair 
obtained from $\pi \colon (X,D_X) \to \PP^2$ after finitely many volume preserving Sarkisov links of type~(II).
Let $\Psi \colon (X^\prime, D^\prime)/\PP^2\dasharrow (X^\dagger,
D^\dagger)/S^\dagger$ be a volume preserving Sarkisov link. Then one of the following holds:
\begin{enumerate}[(a)]
\item $\Psi$ is a Sarkisov link of type~(II);
\item $(X^\prime,D^\prime)/\PP^2$ is isomorphic to $(X,
  D_X)/\PP^2$ and $\Psi=\sigma$;
\item $(X^\prime,D^\prime)/\PP^2$ is isomorphic to object~$2^a$ and
  $\Psi = \chi^a$;
\item $(X^\prime,D^\prime)/\PP^2$ is isomorphic to object~$2^b$ and
  $\Psi=\chi^b$. 
\end{enumerate}
\end{lem}

\begin{proof}
Let $\Phi \colon (X,D_X)/\PP^2 \dasharrow (X^\prime, D_{X^\prime})/\PP^2$ be a composition 
of finitely many (possibly none) volume preserving Sarkisov links of type~(II):
\[
  \begin{tikzpicture}[xscale=1.5,yscale=-1.2]
    \node (A0_1) at (1, 0) {$Z_{i-1}$};
    \node (A0_2) at (2, 0) {$Z_{i}$};
    \node (A1_0) at (0, 1) {$(X_{i-1}, D_{i-1})$};
    \node (A1_3) at (3, 1) {$(X_{i},D_{i})$};
    \node (A2_0) at (0, 2) {$\mathbb{P}^2$};
    \node (A2_3) at (3, 2) {$\mathbb{P}^2$};
    \path (A0_1) edge [->,swap]node [auto] {$\scriptstyle{g_{i-1}}$} (A1_0);
    \path (A1_3) edge [->]node [auto] {$\scriptstyle{\pi_{i}}$} (A2_3);
    \path (A2_0) edge [-,double distance=1.5pt]node [auto] {$\scriptstyle{}$} (A2_3);
    \path (A1_0) edge [->,swap]node [auto] {$\scriptstyle{\pi_{i-1}}$} (A2_0);
    \path (A0_2) edge [->]node [auto] {$\scriptstyle{g_{i}}$} (A1_3);
    \path (A0_1) edge [->,dashed]node [auto] {$\scriptstyle{\varphi_{i}}$} (A0_2);
    \path (A1_0) edge [->,dashed]node [auto] {$\scriptstyle{\Phi_i}$} (A1_3);
  \end{tikzpicture}
  \]
Here  $g_{i-1}\colon Z_{i-1}\rightarrow X_{i-1}$ and $g_{i}\colon Z_{i}\to X_{i}$ are
 divisorial contractions centered at curves $\cC_{i-1}\subset D_{i-1}$
and $\cC_{i}\subset D_{i}$, and  $\varphi_i\colon Z_{i-1}\map Z_{i}$ is a sequence of flips, flops
  and antiflips.
Notice that $g_{i}\colon Z_{i}\to X_{i}$ contracts the strict transform of the surface $\pi_{i-1}^{-1}(\pi_{i-1}(\cC_{i-1}))$ onto $\cC_{i}$.

\textsc{Step 1.} We  prove the following facts about $(X^\prime,
D_{X^\prime})/\PP^2$:
\begin{enumerate}[(a)]
  \item $\pi^\prime \colon X^\prime \to
    \PP^2$ is a $\PP^1$-bundle;
  \item the induced birational map $D_X\dasharrow D_{X^\prime}$ is an
    isomorphism;
  \item no fibre of $\pi^\prime$ is contained in $D_{X^\prime}$. 
\end{enumerate}
To prove this, we proceed by induction on $i$ as in the proof of Theorem~B.
Note that at each step, the curve $\cC_{i-1}\subset D_{i-1}$ is mapped birationally to its images under $\pi_{i-1}$, since $D_{i}$ is normal by Lemma~\ref{D_normal}.
Thus, $\pi_{i}\colon X_i\to \P^2$ is a $\P^1$-bundle over the complement of a finite subset of $\mathbb{P}^2$. So (a) follows from \cite[Theorem 5]{AR14}.

Next we show (b). Note that $D_{i-1}$ does not contain any fiber of $\pi_i$, and thus $D_{i-1}\dasharrow D_{i}$ does not contract any curve. By Zariski's Main Theorem $D_{i}\to D_{i-1}$ is a morphism. Adjunction yields that $K_{D_{i-1}}\sim 0$ and $K_{D_i}\sim 0$. Since $D_{i-1}$ is smooth, we conclude that $D_{i}\to D_{i-1}$ is an isomorphism. So (b) is proven, and (c) follows from (b).

\smallskip

\textsc{Step 2.} Let
$\Psi \colon (X^\prime, D_{X^\prime})/\PP^2\dasharrow (X^\dagger,
D^\dagger)/\PP^2$ be a volume preserving Sarkisov link.  The link
$\Psi$ cannot be a link of type~(I) --- for the same reasons as in
the proof of Theorem~B(1) --- and we suppose it is not of type~(II).
So the link $\Psi$ starts with a birational
modification along an extremal ray $R\subset \NE(X^\prime)$. Let
$\Gamma \subset X^\prime$ be an irreducible curve such that
$R=\bR_{\geq 0}[\Gamma]$. In this step, we show that
$-K_{X^\prime}\cdot \Gamma \geq 0$. This implies in particular that $X^\prime$ is weak Fano. 

\medskip

Suppose for a contradiction that
$-K_{X^\prime}\cdot \Gamma = D_{X^\prime}\cdot \Gamma <0$. It follows
that the extremal contraction $f_R\colon X^\prime \to W$ is small, and
$\Gamma \subset D_{X^\prime}$ is contracted to a point. By
Paragraph~\ref{D_A2},
$\NE(D_{X^\prime}) \ = \ \big\langle [e_0], [e_1], [e'_0],
[e'_1]\big\rangle_+$, and thus $\Gamma$ must be one of the curves
$e_0, e_1, e'_0, e'_1$. By relabelling these curves if necessary, we
may assume that $\Gamma =e_0$. We discuss two cases in turn:

\begin{enumerate}[(i)]
\item $\Gamma$ is not a connected component of the exceptional set of the
  contraction $f_R$;
\item $\Gamma$ is a connected component of the exceptional set of the contraction $f_R$.
\end{enumerate}

\subsection*{Case (i): $\Gamma=e_0$ is not a connected component of the exceptional set of $f_R$}
  The exceptional set of $f_R$ must be $e_0\cup e_1$. The class of $e_1$ in $H_2(X^\prime)$ is
proportional to $e_0$ and then it is immediate that actually
$[e_1]=[e_0]\in H_2(X^\prime)$. It follows that 
\[D_{X^\prime} \cdot e_0=D_{X^\prime}\cdot e_1=-a<0.\]
By Lemma~\ref{lem:A2-nbds},  the extremal neighbourhood around
$e_0\cup e_1\subset D_{X^\prime} \subset X^\prime$ is isomorphic to the analytic germ
  around the curve
  \[
    \Gamma_0\cup \Gamma_1=(x_0=x_2=0) \cup (x_0=x_3=0)
  \]
in the geometric quotient $\CC^5 /\!\!/\CC^\times$ for the action given by the weights:
\[
  \begin{array}{ccccc}
x_0 & x_1 & x_2 & x_3 & x_4 \\
\hline
    -a & 1 & 1 & 0 & -2 \\
    -a &-2& 0 & 1 & 1 \\
  \end{array} ,
\]
where the stability condition is taken in the quadrant $\langle (1, 0), (0,1) \rangle_+$. In these coordinates, $D_{X^\prime}$ is given by $(x_0=0)$. To perform the antiflip $X^\prime \dasharrow X^-$ we need to make
$D_{X^\prime}$ ample, that is, we change the stability condition to
$D_{X^\prime} = (-a,-a)$. The new irrelevant ideal is $(x_0, x_1x_4)$;
thus $X^-$ is covered by the charts $\{x_0\neq 0\}$ and
$\{x_1\neq 0, x_4\neq 0\}$, and we see by looking at the chart
$\{x_1\neq 0, x_4\neq 0\}$ that the antiflip $X^-$ has strictly canonical
singularities of type $\frac{1}{3}(0,1,2)$ (it has a curve of
$A_2$-singularities), a contradiction.

\medskip

\subsection*{Case~(ii): $\Gamma=e_0$ is a connected component of the exceptional set of $f_R$} 
It follows from Lemma~\ref{lem:inverse flips}(2) that $-K_{X^\prime}\cdot \Gamma = D_{X^\prime}\cdot
e_0=-1$ and $\Gamma=e_0$ has normal bundle $N_{\Gamma/X^\prime} \cong \OO(-2) \oplus \OO(-1)$. 
It follows from \textsc{Step~1}
that $X'=\mathbb{P}(\cE)$ for some rank two vector bundle $\cE$ on
$\P^2$. After twisting $\cE$ with a line bundle if necessary, we
may assume that $c_1 \in \{0,-1\}$.
Write $L' =\big[(\pi')^{*}\big(\O_{\P^2}(1)\big)\big]$ and $\xi = \big[\O_{\mathbb{P}(\cE)}(1)\big]$. 
By Lemma~\ref{lem:canonical_in_PP(E)}, 
 \[
D_{X^\prime} \ \sim \ -K_{X^\prime} \ \sim \ (3-c_1)L'+2\xi.
 \]
We compare the lattice $\Pic(D_{X^\prime})$
with the sublattice obtained as the image of the restriction homomorphism
\[
r\colon  \ \Pic({X^\prime}) \ \to \ \Pic(D_{X^\prime}).
\]
As in the proof of Theorem~B(1), one computes that the intersection matrix of $r\big(\Pic({X^\prime})\big)$ in
the restricted basis $(r(L'),r(\xi))$ is 
$$
\left(\begin{array}{cc}
2 & c_1+3 \\ 
c_1+3 & c_1^2+3c_1-2c_2
\end{array}\right).
$$
We will show that this cannot be a sublattice of $\Pic(D_{X^\prime})\cong \Pic(D_{X})$.
Suppose otherwise, and write
\[
r(\xi) = a \alpha + b_0 e_0 +b_1 e_1
\]
for some $a,b_0, b_1\in \ZZ$. Intersecting with $r(L')=\alpha$ we get:
\[
c_1+3=r(L')\cdot r(\xi)=2a+b_0+b_1
\]
and hence 
\[
r(D_{X^\prime}) = 2r(\xi) + (3-c_1)r(L)= 2r(\xi) + (3-c_1)\alpha =
(6-b_0-b_1)\alpha + 2b_0e_0 +2b_1e_1\ .
\]
From this we conclude that
$$
  D_{X^\prime} \cdot e_0 = 6 -5b_0 +b_1=-1.
$$
On the other hand, $\NE (X^\prime)=\langle e_0, f\rangle_+$ where $f$ is a fibre of
$\pi$, and $e_0, f$ are a basis of $H_2(X^\prime)$. It follows, by considering the projection to $\mathbb{P}^2$, that in
$H_2(X^\prime)$ we can write:
\[
  e_1=e_0+\lambda f\;\text{and}\;
  e_1^\prime = e_0+ \mu f, \; \text{for some}\; \lambda, \mu \geq 0.
\]
By computing intersection numbers, we get:
\begin{align*}
  D_{X^\prime} \cdot e_1 & =  6+b_0 -5b_1 \geq -1,\\
  D_{X^\prime} \cdot e_1^\prime & =  6-b_0 +5b_1 \geq -1. 
\end{align*}
Combining these equations we get
\begin{align*}
  5b_0- b_1 & =7,\\
  -7 \leq b_0-5b_1 & \leq 7,
\end{align*}
which do not have common integer solutions, a contradiction.
We conclude that $-K_{X^\prime}\cdot \Gamma \geq 0$, and thus $X^\prime=\PP(\cE)$ is weak Fano. 

\medskip

\textsc{Step~3.}  We determine the rank~$2$ vector bundles $\cE$ on $\PP^2$ for which
$\PP(\cE)$ is weak Fano and contains an anti-canonical divisor isomorphic to $D_X$. 
Theorem~\ref{thm:weakFanoVectorBundles} and the two lists that accompany it
show the rank~$2$ vector bundles $\cE$ on $\PP^2$ with
$c_1 \in \{0,-1\}$ such that $\PP(\cE)$ is Fano or weak Fano. We
continue with the set-up and notation of the proof of Case (ii) of
\textsc{Step~2} above. In oder to determine the possible vector
bundles $\cE$, we shall determine the possible values of
$c_2 = \frac{1}{2}(c_1^2+3c_1 -r(\xi)^2)$.  Write $r(\xi)$
in terms of the basis $\big( [\alpha] ,[e_0], [e_1]\big)$ of
$\Pic(D_X)$:
\[
r(\xi) = \ a\alpha + b_0e_0 + b_1e_1,
\]
for some $a,b_0,b_1\in \ZZ$. Intersecting with $r(L')=\alpha$ we get:
$$
c_1 +3 = 2a + b_0+b_1,
$$
and hence
\[
r(D_{X^\prime})=(6-b_0-b_1)\alpha
+ 2b_0e_0+2b_1e_1.\]
Since $-K_{X^\prime} \sim D_{X^\prime}$ is nef we have:
\begin{align*}
D_{X^\prime}\cdot e_0&=  -5b_0 +b_1 +6 \geq 0,\\ 
D_{X^\prime}\cdot e_0^\prime & =  5b_0 -b_1 +6 \geq 0,\\ 
D_{X^\prime}\cdot e_1& = b_0 - 5b_1 +6  \geq 0,\\ 
D_{X^\prime}\cdot e_1^\prime &= -b_0 + 5b_1 +6 \geq 0.
\end{align*}

The region of the $(b_0,b_1)$-plane defined by these inequalities is
pictured here together with its integral points:
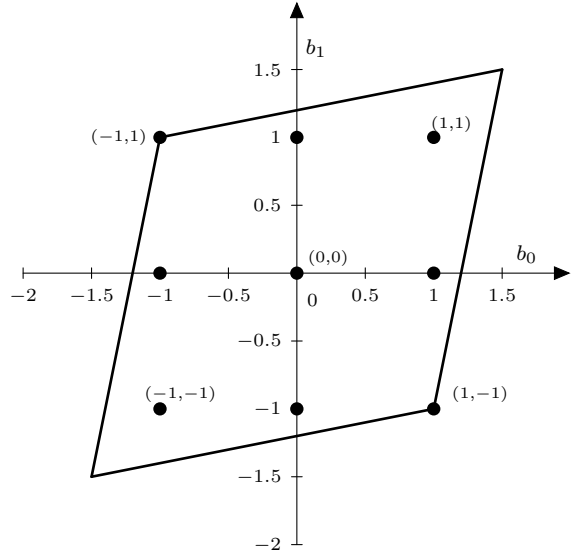
\[
\begin{tikzpicture}[line cap=round,line join=round,>=triangle 45,x=1.8cm,y=1.8cm]
\draw[->,color=black] (-2,0) -- (2,0);
\foreach \x in {-2,-1.5,-1,-0.5,0.5,1,1.5}
\draw[shift={(\x,0)},color=black] (0pt,2pt) -- (0pt,-2pt) node[below] {\tiny $\x$};
\draw[->,color=black] (0,-2) -- (0,2);
\foreach \y in {-2,-1.5,-1,-0.5,0.5,1,1.5}
\draw[shift={(0,\y)},color=black] (2pt,0pt) -- (-2pt,0pt) node[left] {\tiny $\y$};
\draw[color=black] (0pt,-10pt) node[right] {\tiny $0$};
\clip(-2,-2) rectangle (2,2);
\draw [line width=1pt] (-1,1)-- (1.5,1.5);
\draw [line width=1pt] (-1,1)-- (-1.5,-1.5);
\draw [line width=1pt] (-1.5,-1.5)-- (1,-1);
\draw [line width=1pt] (1,-1)-- (1.5,1.5);
\draw (-0.01,1.8) node[anchor=north west] {$\scriptstyle{b_1}$};
\draw (1.53,0.28) node[anchor=north west] {$\scriptstyle{b_0}$};
\begin{scriptsize}
\fill [color=black] (-1,1) circle (2.5pt);
\draw[color=black] (-1.3,1.0) node {$\scriptstyle{(-1, 1)}$};
\fill [color=black] (1,-1) circle (2.5pt);
\draw[color=black] (1.34,-0.89) node {$\scriptstyle{(1, -1)}$};
\fill [color=black] (-1,-1) circle (2.5pt);
\draw[color=black] (-0.85,-0.89) node {$\scriptstyle{(-1, -1)}$};
\fill [color=black] (1,1) circle (2.5pt);
\draw[color=black] (1.13,1.1) node {$\scriptstyle{(1, 1)}$};
\fill [color=black] (0,0) circle (2.5pt);
\draw[color=black] (0.23,0.11) node {$\scriptstyle{(0, 0)}$};
\fill [color=black] (-1,0) circle (2.5pt);
\fill [color=black] (1,0) circle (2.5pt);
\fill [color=black] (0,1) circle (2.5pt);
\fill [color=black] (0,-1) circle (2.5pt);
\end{scriptsize}
\end{tikzpicture}
\]

If $(b_0,b_1)=(0,0)$, then one computes that $a=1$, $c_1=-1$ and $c_2=-2$. 
In this case we must have $\cE\cong
\mathcal{O}_{\P^2}(1)\oplus\mathcal{O}_{\P^2}(-2)$, which is case (8)
of List~2. This leads to a bad link: the nonfibering contraction of
$\P\big(\mathcal{O}_{\P^2}(1)\oplus\mathcal{O}_{\P^2}(-2)\big)$
contracts a divisor to a strictly canonical singularity.

\smallskip 

If $(b_0,b_1)\in\big\{(1,-1), (-1,1)\big\}$, then one computes that
$a=1$, $c_1=-1$ and $c_2=1$.  In this case, $-K_{X'}\cdot e = 0$ for
some $e\in\{e_0,e_1,e'_0,e'_1\}$, and thus $X'$ is weak Fano but not
Fano.  So $\cE$ must be as in case (9) of List~2.  Again this is not
possible because the nonfibering contraction of this $\P^1$-bundle is
divisorial and not small, so it cannot lead to a Sarkisov link.

\smallskip 

If $(b_0,b_1)\in\big\{(1,1), (-1,-1)\big\}$, then one computes that
$r(\xi) =(e_0+e_1)$ or $(e'_0+e'_1)$, $c_1 =-1$ and $c_2 =0$.  In this
case, $\cE\cong \mathcal{O}_{\P^2}\oplus \mathcal{O}_{\P^2}(-1)$, and
$X'$ is the blow-up of $\P^3$ at a point. The contraction $X'\to \P^3$
maps $D_{X'}$ to $D$, contracting either $e_0\cup e_1$ or
$e'_0\cup e'_1$ to the $A_2$ singular point of $D$.

\smallskip 

If $(b_0,b_1)\in\big\{(0,1), (0,-1), (1,0), (-1,0)\big\}$, then one
computes that $r(\xi) =\alpha +e$ for some
$e\in\{e_0,e_1,e'_0,e'_1\}$, $c_1=0$ and $c_2=-1$.  In this case,
$\cE\cong \mathcal{O}_{\P^2}(1)\oplus \mathcal{O}_{\P^2}(-1)$, and
$X'$ is the blow-up of $\P(1,1,1,2)$ at its singular point $q=[0:0:0:1]$.  
This contraction $\chi\colon X'\to \P(1,1,1,2)$ is induced by the linear
system
$$
\big|L'+\xi\big| \ = \ \big|(\chi)^*\O_{\P}(2)\big|.
$$
It contracts the section $E'\sim \xi-L'$ of $\pi'$ containing $e$ to
the singular point $q\in \P(1,1,1,2)$.  The image of $D_{X'}$ is a
quintic hypersurface $D_5\subset \P(1,1,1,2)$. The restriction
$\chi_{|D_{X'}}\colon D_{X'}\to D_5$ contracts $e$ to $q$, and is an
isomorphism elsewhere.  We now compute the equation of $D_5$.  By
assumption, $(X^\prime, D^\prime)/{\PP^2}$ is obtained from
$(X,D_X)/{\PP^2}$ after finitely many volume preserving Sarkisov links
of type~(II):
$$
\begin{tikzcd}
                                                      & X \arrow[rr, "\phi", dashed] \arrow[ld, "\sigma"'] \arrow[d] &  & X' \arrow[rd, "\chi"] \arrow[d] &                       \\
\mathbb{P}^3 \arrow[rrrr, "\psi", dashed, bend right=20] & \mathbb{P}^2                        \arrow[rr,equal]  &  & \mathbb{P}^2                  & {\mathbb{P}(1,1,1,2).}
\end{tikzcd} 
$$
Hence, the composed birational map $\psi=\chi\circ\phi\circ (\sigma^{-1})\colon \mathbb{P}^3\dasharrow\mathbb{P}(1,1,1,2)$ is 
given in coordinates by $\psi(x_0,x_1,x_2,x_3) = (x_0,x_1,x_2,x_3L+Q)$, where $L\in \mathbb{C}[x_0,x_1,x_2]$ is a linear form and 
$Q\in \mathbb{C}[x_0,x_1,x_2]$ is a quadratic form. After a change of variables in $\mathbb{P}(1,1,1,2)$, we may assume that $\psi$ is 
given by $\psi(x_0,x_1,x_2,x_3) = (x_0,x_1,x_2,x_3L)$, and so 
$\psi^{-1}\colon \mathbb{P}(1,1,1,2) \to  \mathbb{P}^3$ is given in coordinates by
$\psi^{-1}(x_0,x_1,x_2,y) = (x_0L,x_1L,x_2L,y)$.     

The strict transform of $D\subset\mathbb{P}^3$ has equation dividing
$$
(x_0x_1y^2+yLB+L^2C = 0)\subset\mathbb{P}(1,1,1,2).
$$
For it to be a quintic, there are just two possibilities: either $L = x_0$ and  $D_5=D_5^b$, or $L = x_1$ and $D_5=D_5^a$.
\end{proof}

\begin{lem}\label{lem:Link3}
The only volume preserving Sarkisov links 
\[
  \Psi \colon (\PP(1,1,1,2),D_5^a)/\Spec \CC \dasharrow
  (X^\dagger,D^\dagger)/S^\dagger
\]
are the maps $\epsilon_a^{-1},(\chi^a)^{-1}, \phi^a, \psi^a$ described
in Examples~\ref{link:1->3a}, \ref{link:3a->2sharp}, \ref{link:3b->4}
and \ref{link:3a->5}.

Similarly, the only volume preserving Sarkisov links 
\[\Psi \colon (\PP(1,1,1,2),D_5^b)/\Spec \CC \dasharrow
  (X^\dagger,D^\dagger)/S^\dagger\] are the maps
$\epsilon_b^{-1},(\chi^b)^{-1}, \phi^b, \psi^b$ described in
Examples~\ref{link:1->3a}, \ref{link:3a->2sharp}, \ref{link:3b->4} and
\ref{link:3a->5}.
\end{lem}

\begin{proof} 
We prove the statement
for the pair $\left(\PP(1,1,1,2),D_5^a\right)/\Spec \CC$. 
The other case is similar. The volume preserving Sarkisov link 
\[
\Psi \colon \left(\PP(1,1,1,2),D_5^a\right)/\Spec \CC \dasharrow
(X^\dagger, D^\dagger)/S^\dagger
\]
must begin with a divisorial contraction $g\colon Y\to \PP(1,1,1,2)$.
By Proposition~\ref{divextter}, the centre $Z$ of this divisorial
contraction is either a curve $\Gamma\subset D_5^a$, or the singular
point $[0:0:0:1]\in\PP(1,1,1,2)$, which is also the singular point of
$D_5^a$.  By \cite[Theorem 5]{Kaw96}, either
$g\colon Y\to \PP(1,1,1,2)$ is the blow-up of $[0:0:0:1]$ with weights
$\left(\frac{1}{2},\frac{1}{2},\frac{1}{2}\right)$, or
$Z=\Gamma\subset D_5^a$ is a curve not passing through $[0:0:0:1]$.
In the first case, if we view $\PP(1,1,1,2)$ as the cone over a
Veronese surface, then $g\colon Y\to \PP(1,1,1,2)$ is the standard
blow-up of the vertex, and it leads to the link $(\chi^a)^{-1}$. From
now on we assume that $Z=\Gamma\subset D_5^a$ is a curve not passing
through $[0:0:0:1]$.

Recall that $D_5^a \subset \mathbb{P}(1,1,1,2)_{(x_0,x_1,x_2,y)}$ is given by the equation
\[
x_0y^2+B_3(x_0,x_1,x_2)y+x_1C_4(x_0,x_1,x_2)=0,
\]
where $B_3$ and $C_4$ are homogeneous polynomials of degree three and
four, respectively. The point $[0:0:0:1]$ is the unique singular point
of $D_5^a$. It is a singularity of type $A_1$.  The divisor class
group of $D_5^a$ is generated by the curves
$e_1 = \{x_1 = x_0y+B_3=0\}$, $\overline{e}_1 = \{y = x_1 = 0\}$, with
intersection matrix
$$
\left(\begin{array}{cc}
-\frac{3}{2} & 3 \\ 
3 & -2
\end{array}\right).
$$
Let $\Gamma\subset D_5^a \subset \mathbb{P}(1,1,1,2)$ be a
reduced and irreducible curve not passing through the singularity of
$D_5^a$, and write its class in $\Cl(D_5^a)$ as
$[\Gamma]=a\overline{e}_1+2b e_1$, with $a,b\in \NN$. 
Let $\pi:Y\rightarrow \mathbb{P}(1,1,1,2)$ be the blow-up of
$\mathbb{P}(1,1,1,2)$ along $\Gamma$ with exceptional divisor
$E$. Then $\NE (Y)$ has two extremal rays. One of them is generated by
a curve $e\subset E$ that is contracted by $\pi$, and we denote  the other  one 
by $R$. Since
$$K_{Y} = \pi^{*}K_{\mathbb{P}(1,1,1,2)} + E \ ,$$
we have $-K_Y\cdot e = 1$.

Suppose  that $-K_Y\cdot R < 0$. 
Then the curve generating $R$ is contained in the strict transform 
$\widehat{D}_5^a$ of $D_5^a$, which is mapped isomorphically to $D_5^a$.
Hence, $R$ is generated either by $\overline{e}_1$ or by $e_1$.
(Here we denote by the same symbols the strict transforms of $\overline{e}_1$ and $e_1$
in $\widehat{D}_5^a\subset Y$.)
Suppose that $R$ is generated by $\overline{e}_1$. By
Lemma~\ref{lem:inverse flips}, $-K_Y\cdot \overline{e}_1 = -1$. Then $a = 3b-3$
and  $\overline{e}_1 = (4a-8b) e + \frac{2}{3} e_1$. So we must have 
$b < 3$, that is
\begin{equation}\label{cas1}
(a,b) \in \{ (0,1), (3,2)\}.
\end{equation} 
Suppose that $R$ is generated by $e_1$. Then
\[
  \widehat{D}_5^a\cdot e_1 = \frac{15}{2}+3b-3a = -\frac{k}{2} \ ,
\]
where $k=-15+6(a-b)$ is a positive integer, and hence $k\geq 3$. By
Lemma~\ref{lem:BadAntiflips}, the antiflip $Y^{-}$ of $R$ has worse than terminal singularities, a contradiction.

Now suppose that  $-K_Y\cdot R\geq 0$, and so $-K_Y$ is nef. In particular,
$-K_Y\cdot\bar e_1 \geq 0$, $-K_Y\cdot e_1\geq 0$, and these two
inequalities translate into the following system
\begin{equation}\label{const_2}
\left\lbrace\begin{array}{l}
5+2a-6b\geq 0,\\ 
\frac{15}{2}+3b-3a\geq 0.
\end{array}\right.
\end{equation}
The area in the $(a,b)$-plane delimited by (\ref{const_2}) and the
inequalities $a\geq 0$, $b\geq 0$, along with its integral points, is displayed in the following picture.
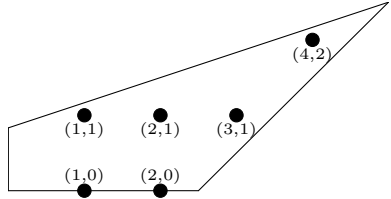
\begin{center}
\begin{tikzpicture}[line cap=round,line join=round,>=triangle 45,x=1cm,y=1cm]
\clip(-0.0,-0.1) rectangle (5.0,2.5);
\draw [line width=0.4pt,domain=-0.1:5.1] plot(\x,{(--2.0833333333333335--0.8333333333333334*\x)/2.5});
\draw [line width=0.4pt,domain=2.5:5.1] plot(\x,{(-6.25--2.5*\x)/2.5});
\draw [line width=0.4pt] (0,0)-- (0,0.8333333333333334);
\draw [line width=0.4pt] (0,0)-- (2.5,0);
\begin{scriptsize}
\draw [fill=black] (1,1) circle (2.5pt);
\draw[color=black] (1.0,0.8) node {$\scriptstyle{(1,1)}$};
\draw [fill=black] (2,1) circle (2.5pt);
\draw[color=black] (2.0,0.8) node {$\scriptstyle{(2,1)}$};
\draw [fill=black] (3,1) circle (2.5pt);
\draw[color=black] (3.0,0.8) node {$\scriptstyle{(3,1)}$};
\draw [fill=black] (4,2) circle (2.5pt);
\draw[color=black] (4.0,1.8) node {$\scriptstyle{(4,2)}$};
\draw [fill=black] (1,0) circle (2.5pt);
\draw[color=black] (1.0,0.2) node {$\scriptstyle{(1,0)}$};
\draw [fill=black] (2,0) circle (2.5pt);
\draw[color=black] (2.0,0.2) node {$\scriptstyle{(2,0)}$};
\end{scriptsize}
\end{tikzpicture}
\end{center}

Taking into account these integral points, together with (\ref{cas1}), 
we are left with the following possibilities for the class $[\Gamma]=a\overline{e}_1+2b e_1$:
$$(a,b) \in \{(1,0),(1,1),(3,1),(0,1),(2,0),(2,1),(4,2),(3,2)\}.$$
We shall show that $(a,b)=(1,0)$ leads to the link $\epsilon_a^{-1}$,
$(a,b)=(1,1)$ leads to $\phi^a$, $(a,b)=(3,1)$ leads to
$\psi^a$, while all the other cases will be excluded. We discuss first the cases
that occur.

\smallskip

\textsc{Case $(a,b)=(1,0)$.} In this case, $\Gamma\sim
\overline{e}_1$.
Blowing it up we get back to the weighted blow-up of $\mathbb{P}^3$ at
$[0:0:0:1]$ with weights $(1,2,1)$. This gives the link
$\epsilon_a^{-1}$.

\smallskip

\textsc{Case $(a,b) = (1,1)$.} In this case
$\Gamma\sim \overline{e}_1+2e_1$. Let $P$ be a divisor on $\Gamma$
associated to $\mathcal{O}_{\Gamma}(2)$. Then $\deg(P) = 8$ and, since
$\deg(K_{\Gamma}) = 4$, Riemann--Roch yields $h^0(\Gamma,P) = 6$. Since
$h^0(\mathbb{P}(1,1,1,2),\mathcal{O}_{\mathbb{P}(1,1,1,2)}(2)) = 7$, we
have $h^0(\mathcal{I}_{\Gamma}(2))\geq 1$. Let
$Q\subset\mathbb{P}(1,1,1,2)$ be a quadric containing $\Gamma$. Then
$Q\cap D_5^a$ has class
$2\overline{e}_1+2e_1 = \overline{e}_1+(\overline{e}_1+2e_1)$. So
$Q\cap D_5^a$ is the union of $\Gamma$ and a residual curve of class
$\overline{e}_1$, which must be $\overline{e}_1$ itself since it is
rigid in $D_5^a$.

We may write $Q = \{y-x_1L(x_0,x_1,x_2)=0\}$ where $L$ is a linear
form. Substituting $y = x_1L$ in the equation of $D_5^a$ we get
\[
\Gamma = 
\left\lbrace\begin{array}{l}
Q = y-x_1L = 0,\\ 
F = x_0x_1L^2+BL+C=0.
\end{array}\right. 
\]
We see from Lemma~\ref{l1} that blowing-up $\Gamma$ leads to the link $\phi^a$
described in Example~\ref{link:3b->4}. 

\smallskip

\textsc{Case $(a,b) = (3,1)$.} In this case
$\Gamma\sim 3\overline{e}_1+2e_1$. Let $P$ be a divisor on $\Gamma$
associated to $\mathcal{O}_{\Gamma}(3)$. Then $\deg(P) = 18$ and, since
$\deg(K_{\Gamma}) = 12$, Riemann--Roch yields $h^0(\Gamma,P) = 12$. Since
$h^0(\mathbb{P}(1,1,1,2),\mathcal{O}_{\mathbb{P}(1,1,1,2)}(3)) = 13$,
we have $h^0(\Gamma,\mathcal{I}_{\Gamma}(3))\geq 1$. Let
$S\subset\mathbb{P}(1,1,1,2)$ be a cubic containing $\Gamma$. Then
$S\cap D_5^a$ has class
$3\overline{e}_1+3e_1 = e_1+(3\overline{e}_1+2e_1)$. So $S\cap D_5^a$ is
the union of $\Gamma$ and a residual curve of class $e_1$, which must
be $e_1$ itself since it is rigid in $D_5^a$.

We may write $S = \{x_0y+B+x_1(cy+Q)=0\}$, where $Q=Q(x_0,x_1,x_2)$ is a quadratic polynomial. Then $S\cap D_5^a = \Gamma\cup e_1$ and $\Gamma\subset\mathbb{P}(1,1,1,2)$ is defined by
\[
\left\lbrace\begin{array}{l}
F_3 = x_0y+B+x_1(cy+Q)  = 0,\\ 
G_4 = y(cy+Q) -C= 0.
\end{array}\right. 
\] 
By Lemma~\ref{l1},
blowing-up $\Gamma$ leads to the link $\psi^a$ described in Example~\ref{link:3a->5}.

\smallskip

\textsc{Easy cases.} Next we exclude all other cases except
$(a,b)=(3,2)$. This case is more difficult and we deal with it at the end.

First, note that, since $e_1$ is rigid inside $D_5^a$, if
$\Gamma\sim 2e_1$ then $\Gamma = 2e_1$ and $\Gamma$ is not
reduced: this rules out the case $(a,b)=(0,1)$. Similarly, we rule out
the case $(a,b)=(2,0)$, that is, $\Gamma\sim 2\overline{e}_1$.

Note that
$\mathcal{O}_{\mathbb{P}(1,1,1,2)}(1)_{|D_5^a}\sim e_1 + \overline{e}_1$.
If $(a,b) = (2,1)$ then $\Gamma\sim 2(e_1+\overline{e}_1)\sim
\mathcal{O}_{\mathbb{P}(1,1,1,2)}(2)_{|D_5^a}$.
Consider the exact sequence
\[0
\rightarrow\mathcal{I}_{\Gamma}(2)\rightarrow\mathcal{O}_{\mathbb{P}(1,1,1,2)}(2)_{|\Gamma}\rightarrow\mathcal{O}_{\Gamma}(2)\rightarrow
0.\]
The degree of the divisor $P$ associated to $\mathcal{O}_{\Gamma}(2)$
on $\Gamma$ is given by
$\deg(P) = 2(e_1+\overline{e}_1)\cdot\Gamma = 4(e_1+\overline{e}_1)^2
= 10$.
By adjunction, $K_{\Gamma} = (K_{D_5^a}+\Gamma)_{|\Gamma}$.
Since $K_{D_5^a}$ is trivial, $\deg(K_{\Gamma}) = \Gamma^2 = 10$. 
(Notice that this holds if $\Gamma$ is reduced and irreducible, but not necessarily smooth \cite[Chapter II]{BPV84}.) Then
$\deg(K_{\Gamma}-P) = 0$ and hence $h^0(\Gamma,K_{\Gamma}-P)\leq 1$.
By Riemann--Roch, we get that $h^0(\Gamma,P)\in\{5,6\}$. Since
$h^0(\mathbb{P}(1,1,1,2),\mathcal{O}_{\mathbb{P}(1,1,1,2)}(2)) = 7$, we
have $h^0(\Gamma,\mathcal{I}_{\Gamma}(2))\geq 1$. Therefore there is a quadric
$Q\subset\mathbb{P}(1,1,1,2)$ containing $\Gamma$, and hence
$\Gamma = D_5^a\cap Q$ is a complete intersection. By the same argument in
the proof of Proposition~\ref{codim2}, the extraction of $\Gamma$
leads to a bad link.

If $(a,b) = (4,2)$, then
$\Gamma\sim 4(e_1+\overline{e}_1)\sim
\mathcal{O}_{\mathbb{P}(1,1,1,2)}(4)_{|D_5^a}$.
In this case, consider the exact sequence
$$0\rightarrow\mathcal{I}_{\Gamma}(4)\rightarrow\mathcal{O}_{\mathbb{P}(1,1,1,2)}(4)_{|\Gamma}\rightarrow\mathcal{O}_{\Gamma}(4)\rightarrow 0.$$
The degree of the divisor $P$ associated to $\mathcal{O}_{\Gamma}(4)$
on $\Gamma$ is given by
$\deg(P) = 4(e_1+\overline{e}_1)\cdot\Gamma = 16(e_1+\overline{e}_1)^2
= 40$,
while $\deg(K_{\Gamma}) = \Gamma^2 = 40$. Then $K_{\Gamma}-P \sim 0$,
and hence $h^0(\Gamma,K_{\Gamma}-P)\leq 1$. By Riemann--Roch, we get
that $h^0(\Gamma,P)\in\{20,21\}$. Since
$h^0(\mathbb{P}(1,1,1,2),\mathcal{O}_{\mathbb{P}(1,1,1,2)}(4)) = 22$,
we have $h^0(\Gamma,\mathcal{I}_{\Gamma}(2))\geq 1$. Therefore there is a quartic
$S\subset\mathbb{P}(1,1,1,2)$ containing $\Gamma$, and hence
$\Gamma = D_5^a\cap S$ is a complete intersection. Again, by the
argument in the proof of Proposition~\ref{codim2},
the extraction of $\Gamma$ leads to a bad link.

\smallskip

\textsc{Case $(a,b) = (3,2)$.} We will show that this case leads to a bad
link. In short: we will blow-up $\Gamma$, analyse the resulting
$2$-ray game with the method of the proof of
Lemma~\ref{l1}, and find that it leads to a bad
link.

We have $\Gamma\sim 3\overline{e}_1+4e_1$. Let $P$ be a
divisor on $\Gamma$ associated to $\mathcal{O}_{\Gamma}(4)$. Then
$\deg(P) = 36$ and, since $\deg(K_{\Gamma}) = 30$, Riemann--Roch yields
$h^0(\Gamma,P) = 21$. 
Since $h^0(\mathbb{P}(1,1,1,2),\mathcal{O}_{\mathbb{P}(1,1,1,2)}(4)) = 22$, we have $h^0(\Gamma,\mathcal{I}_{\Gamma}(4))\geq 1$. Let
$S\subset\mathbb{P}(1,1,1,2)$ be a quartic containing $\Gamma$. Then
$S\cap D_5^a$ has class
$4\overline{e}_1+4e_1 = \overline{e}_1+(3\overline{e}_1+4e_1)$. So
$S\cap D_5^a$ is the union of $\Gamma$ and a residual curve of class
$\overline{e}_1$, which then must be $\overline{e}_1$ itself.

We may write $S = \{yQ-x_1F_3=0\}$, where $Q=Q(x_0,x_1,x_2)$ is a quadratic polynomial and 
$F_3=F_3(x_0,x_1,x_2)$ is a cubic polynomial. Then
$S\cap D_5^a = \Gamma\cup\overline{e}_1$ and
$\Gamma\subset\mathbb{P}(1,1,1,2)$ is defined by
$$
\rank
\left(\begin{array}{ccc}
C_4 & F_3 & y\\ 
x_0y+B_3 & Q & x_1
\end{array}\right)< 2.
$$ 
Since $\Gamma$ cannot pass through the singular point, the monomial $y$ must appear in $Q$, 
and hence we may assume that $Q=y+A_2(x_0,x_1,x_2)$, where $A_2$ is a
quadratic polynomial.

Consider the toric variety $\FF$ with coordinates and weight matrix
\[
\begin{array}{ccccccc}
x_0 & x_1 & x_2 & y & u_0 & u_1 & u_2\\ 
\hline
1 & 1 & 1 & 2 & 0 & -1 & -2\\ 
0 & 0 & 0 & 0 & 1 & 1 & 1
\end{array} 
\]
and stability condition chosen so that the nef cone of $\FF$ is the
span $\langle x_i,u_0 \rangle_+$. This choice gives the irrelevant
ideal $(x_0,x_1,x_2,y)(u_0,u_1,u_2)$, and ensures that we have a
$\PP^2$-bundle morphism $\pi \colon \FF \to \PP(1,1,1,2)$. Consider
the variety $Z\subset \FF$ cut out by the equations
$$
\left(\begin{array}{ccc}
C_4 & F_3 & y\\ 
x_0y+B_3 & Q & x_1
\end{array}\right)
\left(\begin{array}{c}
u_2\\ 
u_1\\ 
u_0
\end{array}\right) = 
\left(\begin{array}{c}
 0 \\ 
 0
\end{array} \right).
$$
It is not hard to see that $Z$ has cDV singularities, that
$\pi_{|Z}\colon Z\to \P(1,1,1,2)$
is a birational morphism with exceptional set a divisor $E$ mapping to
$\Gamma\subset \PP(1,1,1,2)$, and that $-K_Z$ is $\pi_{|Z}$-ample. It
follows from all this that $\pi_{|Z}\colon E\subset Z \to \Gamma \subset
\PP(1,1,1,2)$ is the unique divisorial contraction that generically
blows up $\Gamma \subset \PP(1,1,1,2)$. \footnote{In general, if $W$ is a normal variety
  and $\pi \colon E\subset Z \to \Gamma \subset W$ is a proper birational
  morphism with exceptional set a prime divisor $E$, and such that $-K_Z$ is
  $\QQ$-Cartier and $\pi$-ample, then
  \[
    Z=\uProj_{\O_W} \bigoplus_{n\geq 0}f_\ast \O_Z(-nK_Z) \,.
    \]
  It follows from this characterisation that, if $\pi^\prime \colon
  E^\prime \subset Z^\prime \to \Gamma \subset W$ has the same
  properties and $E=E^\prime$ as valuations of the function field
  $\CC(W)$, then $Z=Z^\prime$. In other words, in the situation of our
proof, there is at most one extremal divisorial contraction $E\subset Z\to
\Gamma \subset \PP(1,1,1,2)$.}

The Mori chamber decomposition of $\FF$ is displayed in the following picture
$$
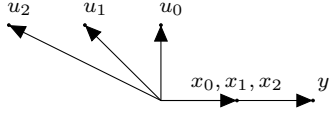
\begin{tikzpicture}[line cap=round,line join=round,>=triangle 45,x=1.0cm,y=1.0cm]
\clip(-2.1,-0.1) rectangle (2.35,1.3);
\draw [->,line width=0.4pt] (0.,0.) -- (1.,0.);
\draw [->,line width=0.4pt] (0.,0.) -- (2.,0.);
\draw [->,line width=0.4pt] (0.,0.) -- (0.,1.);
\draw [->,line width=0.4pt] (0.,0.) -- (-1.,1.);
\draw [->,line width=0.4pt] (0.,0.) -- (-2.,1.);
\begin{scriptsize}
\draw [fill=black] (1.,0.) circle (0.5pt);
\draw[color=black] (1.0,0.21) node {$x_0,x_1,x_2$};
\draw [fill=black] (2.,0.) circle (0.5pt);
\draw[color=black] (2.14,0.21) node {$y$};
\draw [fill=black] (0.,1.) circle (0.5pt);
\draw[color=black] (0.14,1.21) node {$u_0$};
\draw [fill=black] (-1.,1.) circle (0.5pt);
\draw[color=black] (-0.86,1.21) node {$u_1$};
\draw [fill=black] (-2.,1.) circle (0.5pt);
\draw[color=black] (-1.86,1.21) node {$u_2$};
\end{scriptsize}
\end{tikzpicture}
$$
The first wall-crossing $\FF\dasharrow \FF^\prime$ is the flip of
$\{u_1 = u_2 =0\}$, whose restriction to $Z$ is the flip
$Z\dasharrow Z^\prime$ of the strict transform of $\overline{e}_1$.
The next wall-crossing corresponds to a divisorial contraction $\pi^\prime \colon \FF^\prime \to
\PP(1,1,1,1,2,2)$.

We fix homogeneous coordinates $(\xi_0,\xi_1,\xi_2,\xi_3,w_0,w_1)$ on $\PP(1,1,1,1,2,2)$.
The composed birational map $\FF_{(x_0,x_1,x_2,y,u_0,u_1,u_2)} \dasharrow \PP(1,1,1,1,2,2)_{(\xi_0,\xi_1,\xi_2,\xi_3,w_0,w_1)}$ is given by
$$
(x_0,x_1,x_2,y,u_0,u_1,u_2)\mapsto (x_0u_2,x_1u_2,x_2u_2,u_1,yu_2^2,u_0u_2),
$$
and the image $X^\prime$ of $Z$ in  $\mathbb{P}(1,1,1,1,2,2)$ is given by
$$
\left\lbrace\begin{array}{l}
C_4+\xi_3F_3+w_0w_1 = 0,\\ 
\xi_0w_0+B_3+\xi_3Q_2+\xi_1w_1 =0,
\end{array}\right. 
$$
where $Q_2 = w_0+A_2(\xi_0,\xi_1,\xi_2)$. This is a bad link: the
point $[0:0:0:1:0:0]\in X^\prime$ is a hypersurface singularity of
multiplicity at least $3$, and hence it is not terminal. 
\end{proof}

\begin{lem}
  \label{lem:Link4}
Let $(X_4, D_{3,4})/S$ be a Mf CY pair of the family of objects~$4$ described in the end of Section~\ref{sec:sarkisov}. 
The only volume preserving Sarkisov links 
\[\Psi\colon (X_4, D_{3,4})/\Spec \CC \dasharrow
(X^\dagger,D^\dagger)/S^\dagger
\] 
are the maps $(\phi^a)^{-1},(\phi^b)^{-1}$ described in Example~\ref{link:3b->4}.
\end{lem}

\begin{proof}
  Since $X_4$ has Picard rank~$1$, any link $\Psi\colon
  (X_4,D_{3,4})/\Spec \CC \dasharrow (X^\dagger , S^\dagger)$ must
  start with a divisorial contraction $\pi \colon (Z, D_Z)\to
  (X_4,D_{3,4})$ with exceptional divisor $E\subset Z$. 

  The \mbox{$3$-fold} $X_4$ has two singular points of type
  $\frac{1}{2}(1,1,1)$ on the line $\PP^1_{y_0,y_1}\subset
  \PP(1,1,1,1,2,2)$, and $D_{3,4}$ contains these two points as
  $A_1$-singularities. By Proposition~\ref{divextter}, either $\pi(E)=\Gamma\subset D_{3,4}$ 
  is a curve, or $\pi(E)=x\in D_{3,4}$ is one of the two singular points.

  If $\pi(E)=\Gamma\subset D_{3,4}$ is a curve, then, by the main result~\cite{Kaw96}, $\Gamma$ avoids both the two singular points; and hence $\OO_{D_{3,4}}(\Gamma)$ is a Cartier divisor. 
  
  In the notation of \S~6.2, $D_{3,4}$ is obtained from $D_X$ by contracting $e_0$ and $e_1^\prime$. We can identify $\Pic D_{3,4}$ with $\langle e_0,e_1^\prime \rangle^\perp \subset \Pic D_X$, and hence $\Pic D_{3,4} \cong \ZZ$ is generated by (for example) the class of 
  \[
  C = 2e_1 + e_0 + 3 e_1^\prime .
  \]
  But $\OO_{D_{3,4}}(C)=\OO_{D_{3,4}}(2)$ --- for example because both line bundles have self-intersection $12$. 
  It follows from Proposition~\ref{codim2} that there is a positive integer $k$ and a Cartier divisor $H\in |\OO_X(2k)|$ such that $\Gamma=H_{|D_{3,4}}$, and this case in fact does not occur.
  
  If $\pi(E)=x\in D_{3,4}$ is one of the two singular points, then --- by the main result of~\cite{Kaw96} --- $\pi$ is the weighted blow-up with weights
  $\left(\frac{1}{2},\frac{1}{2},\frac{1}{2}\right)$. The two extremal
  contractions give the links $(\phi^a)^{-1}$ and
  $(\phi^b)^{-1}$.
\end{proof}

\begin{lem}
  \label{lem:Link5}
Let $(X_4, D_{2,4})/S$ be a Mf CY pair of the family of objects~$5^a$
described in the end of Section~\ref{sec:sarkisov}.  
The only volume preserving Sarkisov links 
\[\Psi\colon (X_4, D_{2,4})/\Spec \CC\dasharrow
(X^\dagger,D^\dagger)/S^\dagger
\] 
are the maps $(\psi^a)^{-1},(\widetilde{\psi}^b)^{-1}$ described in Example~\ref{link:3a->5}.
\end{lem}

\begin{proof}
    Since $X_4$ has Picard rank~$1$, any link $\Psi\colon
  (X_4,D_{2,4})/\Spec \CC \dasharrow (X^\dagger , S^\dagger)$ must
  start with a divisorial contraction $\pi \colon (Z, D_Z)\to
  (X_4,D_{2,4})$. 

  The key observation is this: $X_4$ has a unique singular point
  $[0:0:0:1:0]\in X_4$, analytically isomorphic to the germ at the
  origin of the hypersurface
  \begin{equation}
    \label{eq:analytics}
xy+z^3+t^3=0    
  \end{equation}
  in $\CC^4$. The surface $D_{2,4}$ passes through this point, and has
  an $A_2$-singularity there. In fact, upon substituting $y=-x_1x_3$,
  one sees that the surface $D_{2,4}$ is the original surface
  $D\subset \PP^3$. It follows from this that
  $\Cl(X_4)\to \Cl(D_{2,4})$ is an isomorphism.

  By Propositions~\ref{divextter} and~\ref{codim2},
  $\pi \colon Z \to X_4$ contracts the unique exceptional divisor to
  the singular point. By~\cite{Ka03}, up to isomorphism, there are
  precisely two divisorial contractions to a singular point as in
  Equation~\ref{eq:analytics}, given by the weighted blow-ups with
  weights $(2,1,1,1)$ and $(1,2,1,1)$.\footnote{If you find this
    statement confusing, note that the ``ordinary'' blow-up of the
    point is not an extremal contraction, because the exceptional
    divisor is not irreducible (it has two irreducible components).}
  These two extremal contractions
  give the links $(\psi^a)^{-1}$ and $(\widetilde{\psi}^b)^{-1}$.

  Indeed, consider the link
  $\psi^a\colon (\PP(1^3,2),D_5^a)\dasharrow (X_4, D_{2,4})$, as
  described in detail in Example~\ref{link:3a->5}. As described there,
  the link terminates with $\pi^\prime \colon \FF \to \PP(1^4, 2)$, and it
  is clear from Equation~\ref{eq:tricky_blowup} that $\pi$ is the
  weighted blow up of the point $x=[0:0:0:1:0]\in \PP(1^4,2)$ with weights
  $(1,1,1,2)$. The tangent cone of $x\in X_4$ is
  $y(y+x_0+x_1)$, hence $\pi$ induces the extremal divisorial
  contractions to $x\in X_4$ where $y$ has weight~$2$. The change of
  coordinates that transform object $5^a$ to $5^b$
  (Example~\ref{rem:objects5}) sets
  $-\widetilde{y}=y+x_3(x_0+x_1)$, and this shows that
  $\widetilde{\psi}^b$ terminates with the blow-up where $\widetilde{y}$ has
  weight two, that is, the ``other'' extremal divisorial contraction
  to $x\in X_4$.
\end{proof}

\bibliographystyle{amsalpha}

\providecommand{\bysame}{\leavevmode\hbox to3em{\hrulefill}\thinspace}
\providecommand{\MR}{\relax\ifhmode\unskip\space\fi MR }
\providecommand{\MRhref}[2]{%
  \href{http://www.ams.org/mathscinet-getitem?mr=#1}{#2}
}
\providecommand{\href}[2]{#2}

\end{document}